\documentclass[a4paper]{article}

\usepackage[T1]{fontenc}
\usepackage[utf8]{inputenc}
\usepackage{float}
\usepackage{graphicx}
\usepackage{subcaption} 
\usepackage[thinlines]{easytable}
\usepackage{hyperref}
\usepackage{wrapfig}
\usepackage{dingbat}
\usepackage{color}
\usepackage{refcount}
\usepackage[table]{xcolor}
\usepackage[toc,page]{appendix}
\usepackage{bbm}
\usepackage{comment}
\usepackage{amsmath,amssymb,amsthm}
\usepackage{mathtools}

\usepackage[english]{babel}

\usepackage{natbib}

\usepackage{xcolor}

\makeatletter
\@addtoreset{equation}{section}

\makeatother

\theoremstyle{Theorem}
\newtheorem{theo}{Theorem}
\theoremstyle{Lemma}
\newtheorem{lem}[theo]{Lemma}
\newtheorem{prop}[theo]{Proposition}
\newtheorem{rem}[theo]{Remark}
\newtheorem{assumption}[theo]{Assumption}

\theoremstyle{Definition}
\newtheorem{defi}[theo]{Definition}
\theoremstyle{Corollary}
\newtheorem{cor}[theo]{Corollary}

\newcommand{ \FP }{ \xi } 
\newcommand{ \SM }{ \gamma } 

\DeclareFontFamily{U}{matha}{\hyphenchar\font45}
\DeclareFontShape{U}{matha}{m}{n}{
      <5> <6> <7> <8> <9> <10> gen * matha
      <10.95> matha10 <12> <14.4> <17.28> <20.74> <24.88> matha12
      }{}
\DeclareSymbolFont{matha}{U}{matha}{m}{n}
\DeclareFontSubstitution{U}{matha}{m}{n}

\DeclareFontFamily{U}{mathx}{\hyphenchar\font45}
\DeclareFontShape{U}{mathx}{m}{n}{
      <5> <6> <7> <8> <9> <10>
      <10.95> <12> <14.4> <17.28> <20.74> <24.88>
      mathx10
      }{}
\DeclareSymbolFont{mathx}{U}{mathx}{m}{n}
\DeclareFontSubstitution{U}{mathx}{m}{n}

\DeclareMathDelimiter{\vvvert}{0}{matha}{"7E}{mathx}{"17}

\newcommand{\E}{{\mathbb E}}
\newcommand{\PP}{{\mathbb P}}
\newcommand{\N}{{\mathbb N}}
\newcommand{\Z}{{\mathbb Z}}

\newcommand{\R}{{\mathbb R}}

\newcommand{\Tr}{\textnormal{Tr}}
\newcommand{\SL}{S^{(1)}}
\newcommand{\SLL}{S^{(2)}}

\newcommand{\HH}{{\mathcal H}}

\newcommand{\be}{\begin{equation}} 
\newcommand{\ee}{\end{equation}}

\numberwithin{theo}{section}

\newcommand{\TK}{\textcolor{red}}

\newcommand{\change}[1]{{\color{magenta}{{#1}}}}

\allowdisplaybreaks

\begin{document}

\title{Statistical inference for the slope parameter in  functional linear regression}

%------------------- Content of abstract at interim state
\author{Tim Kutta, Gauthier Dierickx, Holger Dette }

\maketitle

\begin{abstract}
In this paper we consider the linear regression model $Y =S X+\varepsilon $ with functional  regressors and responses. We develop new inference tools to quantify deviations of the true slope $S$ from a hypothesized operator $S_0$ with respect to the Hilbert--Schmidt norm $\vvvert S- S_0\vvvert^2$, as well as the prediction error $\E \| S X - S_0 X \|^2$. 
Our analysis is applicable to functional time series and based on asymptotically pivotal statistics. This makes it particularly user friendly, because it avoids the choice of tuning parameters inherent in long-run variance estimation or bootstrap of  dependent data. We also discuss  two sample problems as well as change point detection. Finite sample properties are investigated by means of a simulation study.\\
Mathematically our approach is based on a sequential version of the popular spectral cut-off estimator $\hat S_N$ for $S$. It is well-known that the $L^2$-minimax rates in the functional regression model, both in estimation and prediction, are substantially slower than $1/\sqrt{N}$ (where $N$ denotes the sample size) and that standard estimators for $S$ do not converge weakly to non-degenerate limits. 
However, we demonstrate that simple plug-in estimators - such as $\vvvert \hat S_N - S_0 \vvvert^2$ for  $\vvvert S - S_0 \vvvert^2$ - are $\sqrt{N}$-consistent and its sequential versions satisfy weak invariance principles. These results are based on the smoothing effect of $L^2$-norms and established by a new proof-technique, the {\it smoothness shift}, which has potential applications in other statistical inverse problems. 
\end{abstract}

\bigskip

%
%
%--------------------------------------------------------------------------------------------------------------
%
%
% 	SUBSECTION : INTRODUCTION
%
%
%
%
%

\section{Introduction}
\label{sec1}

%
%
%--------------------------------------------------------------------------------------------------------------
%
%
% 	SUBSECTION : INTRODUCTION
%
%

%\TK{\bf Wir muessen entscheiden ob wir in der Gegenwart  ``xxx develop'' oder Vergangenheit
%``xxx developed'' scrheiben. Da der meisste Teil in der Gegenwart geschrieben ist bin ich fuer Gegenwart und habe Aenderungen geamcht - bitte pruefen!}
Over the past decades new branches of statistical research have 
%been 
developed to meet the needs of an economy with  growing data volumes at its disposal. One approach to analyze large data samples, particularly when detected on a dense grid, is interpolation of discrete measurements to continuous, functional observations.
 This method is known as functional data analysis (FDA) and nowadays has numerous  applications as diverse as economics, climatology and medicine \citep[see, for example,][among many others]{anderson2010, bonner2014, sorensen2013}. FDA benefits  users in several ways: From a theoretical perspective functional models - in contrast to standard multivariate analysis - can  incorporate smoothness in the data. From a computational viewpoint, the interpolation of thousands of observations to, say a yearly curve of price development, can drastically reduce the amount of data to be stored, since interpolations only consume a fraction of memory compared to the noisy raw data \citep[see, for example,][among many others]{liebl2013,stoehr2020}. 
From a practical point of view, random curves are easy to visualize and interpret for human users, who cannot possibly make sense from endless data lists.\\
One model that  has attracted particular interest in the context of FDA due to its parsimony and interpretability is the functional linear regression model
\begin{equation} \label{model_1}
        Y_n 
    = 
        S X_n + \varepsilon_n \qquad  n=1,...,N. 
\end{equation}
Here the  regressors, errors and dependent variables are functions.  
More mathematically, $X_n, \varepsilon_n$  and  $Y_n$ are elements of  
(potentially different) Hilbert spaces $H_1$ and $H_2$ and the slope parameter $S: H_1 \to H_2$ is a Hilbert--Schmidt operator. Such models extend existing ones for time series and  panel data, and have applications in different situations, where standard, non-functional approaches fail   \citep[see, for example,][]{anderson2010}. \\
Linear models are attractive for users because of their simple structure, where all information is stored in the slope parameter $S$. However, compared to the better known case of finite dimensional, linear regression, the standard tasks of estimation, prediction and statistical inference become substantially more difficult in the functional regime. Indeed all of these tasks require the  approximate inversion of the compact covariance operator $\Gamma := \mathbb{E}X \otimes X$ (we discuss this issue in detail in Section \ref{Subsection_linear_model}), which constitutes a {\it statistical inverse problem}. 
Statistical inverse problems extend {\it classical inverse problem}  \citep[the reconstruction of an entity using indirect observations, see fore example][]{Englbook,Helgasonbook}
by including noise in the model.
Naturally arising in different settings, statistical inverse problems have been studied extensively in tomography, deconvolution or the heat equation, to name but a few examples \citep[see][and the  references therein]{cavalier2008,bissantz2008}.  
Since our subsequent discussion is exclusively concerned with statistical inverse problems, we will for brevity just refer to them as inverse problems without qualifications. Characteristic of inverse problems is a need for regularization, which leads to slower than parametric convergence rates of the resulting   estimates.
%, depending on the regularity of the object of interest - a feature we will retrieve for estimators of $S$ in \eqref{model_1}. \\
%By  definition, these problems involve a regularization method to reconstruct an unobservable quantity of interest (in our case the operator $S$), which entails a trade-off between an increased bias for strong regularizations and an increased variance for weak regularizations. 
%How this trade-off plays out in model \eqref{model_1} is discussed in detail in Section \ref{Subsection_regression_estimator}.

The study of functional linear regression and the associated inverse problem has been a part of FDA for more than two decades   \citep[see, for example the monograph of][]{Ramsaybook}. Early work has focused on the scalar response model
\begin{equation} \label{functions_on_regressors}
        Y_n
    = 
        \int_0^1 \varphi_S(t) X_n(t) dt + \varepsilon_n~,
\end{equation}
which is a special case of \eqref{model_1}, where $H_1=L^2[0,1]$, the space of square integrable function defined on the interval $[0,1]$, $H_2= \R$  and
$S: L^2[0,1] \to \R$ is an integral operator with square integrable kernel $ \varphi_S$. 
%Such a model captures the relation between a regressor function and an economic or scientific key figure, as they often occur in practical applications.
For the investigation of scalar response models,  we refer the interested reader for instance to \cite{cardot2003}, \cite{hall2007}, \cite{yuan2012} and references therein. 
In \cite{hall2007} convergence rates for spectral cut-off estimators (a specific kind of regularization also used in this paper) are investigated with respect to the $L^2$-norm and it is shown that these estimators can achieve minimax optimal rates. Generalizations of these results to functional linear regression with functional responses can be found  in \cite{masaaki2017}. 
%As suggested in \cite{hall2007}, alternative regularization methods to spectral cut-off can be applied to  functional linear regression. 
Similarly,  \cite{benatia2017}  investigate minimax $L^2$-rates, as well as practical aspects of Tikhonov regularizations in the estimation of the slope parameter. 
Besides $L^2$-rates other aspects of model \eqref{model_1} have been investigated in a wide variety of works, such as consistency under weak dependence in a white noise model  \citep[see][]{hoermann2012}, aspects of identifiability \citep[see][]{scheipl2016}, minimax rates for  
prediction \citep[see][]{crambes2013} and robust estimation \citep[see][]{shin2016}. \\
The list of  cited references is by no means complete and only comprises a fraction of the larger body of research in this domain. 
Besides estimation and prediction, hypothesis testing in 
the functional regression model has attracted some  attention.
\cite{cardot2003, cardot2004} consider the problem of testing for a particular value of the slope,  
i.e.  $H_0:S=S_0$ vs. $H_1: S \neq S_0$, where $S_0$ is some hypothesized operator (see Section \ref{Subsection_relevant_hypothesis} for details and more references). 
It turns out that $H_0$ can be examined by $\sqrt{N}$-consistent tests, which employ transformed versions of both operators $S$ and $S_0$. Importantly such tests do not have to solve the inverse problem of reconstructing $S$, which makes them theoretically more parsimonious, but practically somewhat difficult to interpret, as they do not assess the deviation of the true operators of interest (for a more detailed  discussion of this problem we refer to  Section \ref{Subsection_relevant_hypothesis} below). As a consequence attention has recently shifted to inference methods, based on direct slope comparisons, to make the results statistically more meaningful.  For example, in the functional linear model \eqref{functions_on_regressors} with  scalar responses \cite{kato2019}  develop confidence bands  that  cover the slope function at most points with a prespecified probability. Other authors use Gaussian approximations to construct uniform confidence bands; see, for example, \cite{babii2020}, who 
devises honest confidence bands 
for the regression function in a
nonparametric instrumental variable regression using Tikhonov regularization.
%\TK{{\bf evtl. weglassen?} Besides confidence bands, other approaches have been investigated, such as studying the supremum of $\int_0^1 \varphi_S(t) x(t) dt $, on a slowly growing set, converging to the $L^2$-unit ball \citep[see][]{leung2021}. Roughly speaking, this approach  provides confidence bands w.r.t. the operator norm and the associated test statistic approximately equals the supremum of a Gaussian process. }
Notice that these % latter
approaches, based on reconstructing $S$ (or $\varphi_S$), have to pay the price of solving the inverse problem, by a convergence speed significantly slower than $1/\sqrt{N}$. \\
In this paper, we contribute to the discussion by providing a new method of statistical inference in the regression model \eqref{model_1}. Our inference concerns the two deviation measures  $\vvvert S-S_0 \vvvert^2$ (the distance in Hilbert--Schmidt norm) and  $\E\| SX-S_0X \|^2$ (the expected prediction error), where again $S_0$ is a hypothesized operator. In contrast to the hypothesis of the form $H_0: S=S_0$, we prefer a more quantitative approach, testing whether the deviation $\vvvert S-S_0 \vvvert^2$ or $\E\| SX-S_0X \|^2$ is smaller than some predetermined threshold, say  $\Delta >0$.
%(for a precise statement of the hypotheses see equations \eqref{relevant_difference} and \eqref{relevant_change}). 
Although one has to solve the inverse problem to reconstruct $S$, the proposed estimates  converge at a parametric rate of $1/\sqrt{N}$, due to a natural smoothing effect of the $L^2$-norms. 
%\TK{{\bf wuerde ich weglassen!} This phenomenon, that sometimes under the  alternative (to the point-hypothesis) faster convergence rates are possible than under the null, has been described in earlier works such as \cite{dette1999} and in the context of deconvolution  \cite{bissantz2012} (see their Theorem 2). However the proves of such results are usually cumbersome and this is particularly true in model \eqref{model_1}, not least because (in contrast to the aforementioned works) we want to investigate dependent time-series.
%}
In particular, we use a new proof-technique, called  {\it smoothness shift},  to
establish asymptotic normality for estimators of the deviation measures $\vvvert  S-S_0 \vvvert^2$  and $\E\|  SX-S_0X \|^2$. This technique can be also used in the study of other inverse problems, such as deconvolution or tomography and is therefore of independent interest. 

 A direct application of these results for statistical inference such as the construction of hypothesis tests or confidence intervals is theoretically possible, but practically difficult, because it requires the estimation of asymptotic (long-run) variances. This estimation is intricate in inverse problems even for i.i.d.\ data and becomes even more difficult for functional time series (see Proposition~\ref{proposition_1} below for a presentation of the long-run variance $\tau^2$). To circumvent these problems  we   investigate sequential versions of our estimators, prove weak invariance principles and use the concept of self-normalization 
 \citep[see, for example,][]{shao2015,DetKokVol20} to 
construct (asymptotically) pivotal estimates of the deviation measures. Users benefit from the principle of self-normalization, because it provides (robust) inference tools, which
do not require the choice of tuning parameters for long-run variance estimation (see, e.g. \cite{Horvath2011} and \cite{Kokoszka2012})
or for the block bootstrap  of dependent data  \citep[see, e.g.,][]{poliroma1994,bucher2013}.\\
The rest of this paper is organized as follows: In Section \ref{sec2}, we discuss the linear model in detail and construct the spectral cut-off estimator $\hat S_N$ for $S$. Next, in Section~\ref{sec3}, we present statistical inference for the distance in Hilbert--Schmidt norm and in Section~\ref{sec4} inference for the expected prediction error. Then, in Section~\ref{sec5} we propose extensions
of our methodology  to two sample and change point scenarios, while in Section~\ref{sec6} we investigate
finite sample properties by virtue of a simulation study. Finally,  the Appendix contains the technical proofs and mathematical details.

\section{Estimation of the slope parameter} 
\label{sec2} 

In this Section we  introduce the mathematical set-up for estimation in the functional  linear
regression model \eqref{model_1}. We begin by  recalling some basic  facts about Hilbert--Schmidt operators and  continue 
with a discussion of the estimation problem of the slope $S$  in the functional linear model. In particular,  we explain the necessity of regularization   and 
discuss the ensuing variance-regularization trade-off. 

\subsection{Operators on Hilbert spaces} \label{Subsection_Operators}

Throughout this paper we treat functional observations as elements of Hilbert spaces. Thus before we proceed to the statistics, we recall some fundamental aspects of operator theory on Hilbert spaces. For a more detailed overview we recommend the monographs of  \cite{HorKokBook12} (with particular emphasis on functional data) as well as \cite{WeidmannBook80}.  \\
Suppose two generic Hilbert spaces $(\mathcal{H}_1, \langle, \rangle_1)$ and $(\mathcal{H}_2, \langle, \rangle_2)$ are given. The corresponding norms on the spaces are denoted by $\|\cdot\|_i$, for $i=1,2$.  The space $\mathcal{L}(\mathcal{H}_1, \mathcal{H}_2)$ consists of all linear operators $L:\mathcal{H}_1 \to \mathcal{H}_2$, satisfying 
$$
	    \vvvert L\vvvert_\mathcal{L}
	:= 
	    \sup_{\|x\|_1 = 1} \|L x\|_2
	<
	    \infty.
$$
The norm $\vvvert \cdot \vvvert_\mathcal{L}$ is referred to as operator or spectral norm. Recall that all operators with bounded spectral norm are also continuous. An important subclass of $\mathcal{L}(\mathcal{H}_1, \mathcal{H}_2)$ is given by the compact operators, i.e., such operators $L$ which map the unit ball in $\mathcal{H}_1$ to a totally bounded set in $\mathcal{H}_2$.
%(i.e.\ a set with compact closure). 
In the special case where $\mathcal{H}_1=\mathcal{H}_2$ and the operator $L$ is both compact and symmetric, $L$ can - according to the spectral theorem for normal operators - be diagonalized, in the sense that for any $x \in \mathcal{H}_1$
\begin{equation} \label{spectral_decomposition}
 	    L x 
 	=  
 	    \sum_{n \in \N } \eta_n 
 	    \langle f_n, x \rangle_1 f_n,   
\end{equation}	
where $\eta_1, \eta_2,... \in \R$ are the eigenvalues and $f_1, f_2,... \in \mathcal{H}_1$ the corresponding eigenvectors of $L$. In the context of functional spaces the eigenvectors are usually referred to as eigenfunctions. The most restrictive class of operators, that we consider in this paper consists of the {Hilbert--Schmidt operators}. This subspace of $\mathcal{L}(\mathcal{H}_1, \mathcal{H}_2)$ denoted by $\mathcal{S}(\mathcal{H}_1, \mathcal{H}_2)$ consists of all linear operators $L$, which satisfy
$$
	\vvvert L \vvvert := \sum_{n \in \N} \|L f_n \|_2^2 <\infty,
$$
where $\{f_n\}_{n \in \N}$ is some orthonormal basis (ONB) of $\mathcal{H}_1$. The value of the norm is independent of the basis. Notice that $\vvvert L \vvvert< \infty$ directly entails compactness and hence boundedness w.r.t.\ the operator norm. The norm $\vvvert \cdot \vvvert$ is called {Hilbert--Schmidt norm} and is the infinite-dimensional analogue of the Frobenius norm. Just as the Frobenius norm it is induced by an inner product, which for two operators $L, T \in \mathcal{S}(\mathcal{H}_1, \mathcal{H}_2)$ is given by
$$
	    \langle L, T \rangle 
	:=  
	    \sum_{n \in \N} \langle L f_n, T f_n \rangle_2,
$$
where again the value of the inner product (on the left hand side) does not depend on the choice of basis. Equipped with this inner product the linear space $\mathcal{S}(\mathcal{H}_1, \mathcal{H}_2)$ becomes itself a Hilbert space. 
%Move to Appendix!
%It is not difficult to see that the two norms $\vvvert \cdot \vvvert$  and $\vvvert \cdot \vvvert_\mathcal{L}$ satisfy the following relations for any $L \in \mathcal{S}(\mathcal{H}_1, \mathcal{H}_2)$ and $T \in \mathcal{L}(\mathcal{H}_1, \mathcal{H}_2)$:
%\begin{equation}
%\vvvert L \vvvert_\mathcal{L} \le \vvvert L \vvvert \quad \quad \textnormal{and} \quad \quad  \vvvert T L \vvvert \le \vvvert T \vvvert_\mathcal{L}  \vvvert L \vvvert 
%\end{equation}
%In particular the composition of a bounded and a Hilbert--Schmidt operator is again Hilbert--Schmidt. 
Finally we introduce the outer product of two elements in $\mathcal{H}_1, \mathcal{H}_2$. For any $f \in \mathcal{H}_1$, $g \in \mathcal{H}_2$ we define the linear operator $g \otimes f \in \mathcal{S}(\mathcal{H}_1, \mathcal{H}_2)$, pointwise  by 
\be \label{Eq_DefTensProd}
g \otimes f [h] := g \langle f, h \rangle_1 \quad \quad \forall h \in \mathcal{H}_1.
\ee
By virtue of this definition it is possible to endow $\mathcal{S}(\mathcal{H}_1, \mathcal{H}_2)$ with a particularly natural basis: 
  If $\{f_n\}_{n \in \N}, \{g_n\}_{n \in \N}$ are ONBs of $\mathcal{H}_1$, $\mathcal{H}_2$ respectively, then the set $\{g_n \otimes f_m\}_{ n,m \in \N }$ is an ONB of 
  $\mathcal{S}(\mathcal{H}_1, \mathcal{H}_2)$. Finally, we notice that the outer product notation can be used to restate the spectral theorem for a compact, symmetric operator $L$ in \eqref{spectral_decomposition} as follows
$$
L=	  \sum_{n \in \N } \eta_n  f_n \otimes  f_n.
$$
In the next step we bring to bear these notations to 
the analysis of the functional regression problem \eqref{model_1}.

\subsection{The functional linear model}  \label{Subsection_linear_model}

In this Section we introduce the functional, linear regression model \eqref{model_1} in a more rigorous way. Let  $\mathcal{T} \subset \R^d$ denote a compact and  non-empty set and $\mu_1$, $\mu_2$ measures  defined on 
some $\sigma$-algebra on $\mathcal{T}$. Furthermore define
 $H_1:= (L^2(\mathcal{T}), \mu_1)$ and 
 $H_2:= (L^2(\mathcal{T}), \mu_2)$ as the
 spaces  of all measurable, real-valued functions on $\mathcal{T}$, that are square integrable w.r.t.\ $\mu_1$ and  $\mu_2$, respectively.
Equipped with the inner products
$$
\langle f, g \rangle := \int_\mathcal{T} f(t) g(t) d \mu_i(t) \quad \quad f,g \in H_i ~~~~(i=1,2)
$$
$H_1$ and $H_2$   are Hilbert spaces. Notice that the inner product $\langle f, g \rangle $ depends on the index $i=1,2$, but for the sake of simplicity we do not make this explicit. Accordingly, the norms induced by the inner products are denoted by $\|\cdot \|$. 

This general setup includes many of the standard scenarios treated in the related literature. For instance to retrieve the model \eqref{functions_on_regressors}  with functional regressors and scalar responses  \citep[see][]{hall2007} it suffices to  set $\mathcal{T}=[0,1]$, $\mu_1= \lambda$ (the Lebesgue measure) and $\mu_2= \delta_1$ (the Dirac  measure at the point  $1$). Another typical setting is to choose both measures as the Lebesgue measure, which gives functional inputs and outputs   \citep[see, for example,][among many others]{yao2005}.  Further important non-standard cases such as spatio-temporal functions with continuous time and discrete space components  \citep[see][]{constantinou2015} can be accommodated as well.

Let $(X_1, Y_1),\ldots , (X_N, Y_N)$ denote 
$N$ observations  from a time series
$\{ (X_n, Y_n) \}_{n \in \Z} \subset H_1 \times H_2$, which 
  are generated according to the linear model \eqref{model_1}, that is 
\begin{align} \label{hd1}
        Y_n 
    = 
        S X_n + \varepsilon_n 
    \qquad 
        n=1,...,N,
\end{align}
where $ S \in \mathcal{S}(H_1, H_2)$ denotes  the  (unknown)  slope parameter  and $\varepsilon_n \in H_2$ an  observational error.
 By virtue of the outer product (see Section~\ref{Subsection_Operators}) it is possible to transform this linear model into a version, which is 
 more suitable to inference about the slope parameter. More precisely, ``multiplying'' \eqref{hd1}  by $X_n$ from the right gives
\begin{equation} \label{model_2}
Y_n \otimes X_n = SX_n \otimes X_n + \varepsilon_n \otimes X_n \quad \quad n=1,...,N.
\end{equation}
Under the assumption  $\E \|X_n\|^2, \E \|\varepsilon_n\|^2 <\infty$ 
the operators $ SX_n \otimes X_n , \varepsilon_n \otimes X_n$ are random elements in $\mathcal{S}(H_1, H_2)$. 
%which are moreover genuine 
Moreover, if the random functions $X_n, \varepsilon_n$ are also centered, taking   expectations on both sides of \eqref{model_2} gives
\begin{equation}  \label{equation_expectation}
\mathbb{E} Y_n \otimes X_n = S \Gamma+\E \varepsilon_n \otimes X_n.
\end{equation}
Here $\Gamma := \mathbb{E} X_n \otimes X_n$ is the covariance operator of $X_n$ (recall that the sequence $\{ X_n\}  _{n\in \N} $  of  regressors is stationary). Note that we merely assume  centered  regressors for ease of presentation and adaption to the non-centered case is simple (for details see Remark~\ref{remark_centering}). Under the additional assumption of weak exogeneity, i.e., $\E \varepsilon_n \otimes X_n=0$, equation \eqref{equation_expectation} entails the fundamental identity
\begin{equation}  \label{inverse_characterization}
\mathbb{E} Y_n \otimes X_n = S \Gamma.
\end{equation}

%Next, we discuss how identity~\eqref{inverse_characterization}  can be used in order to estimate $S$ by inversion of $\Gamma$. 

% \subsection{The regression estimator} \label{Subsection_regression_estimator}

The task of recovering the operator $S$ from
equation \eqref{inverse_characterization}  
%can be potentially exploited to estimate the slope parameter $S$. However this task 
is non-trivial, even if we knew the ``true'' expectation $S \Gamma = \mathbb{E} Y_n \otimes X_n$ and the  covariance operator $\Gamma$. 
One obvious condition for a complete recovery of $S$ is identifiability, which is satisfied, if $\Gamma$ is an injective  operator. However, even in this case, as  $\Gamma$ is compact, its inverse must be unbounded and hence can only be defined on a dense linear subspace. We refer the interested reader to \cite{DunSchBookPart1} and \cite{WeidmannBook80} for a detailed discussion of (un)bounded operators.

 A remedy for this problem is given by the application of a regularized inverse, i.e.\ a sequence of continuous operators $\{\Gamma^\dagger_k\}_{k \in \N}$,  converging pointwise to $\Gamma^{-1}$. Of course this means that $\vvvert \Gamma^\dagger_k \vvvert_\mathcal{L} \to \infty = \vvvert \Gamma^{-1} \vvvert_\mathcal{L}$, but for each finite $k$ the operator $S \Gamma \Gamma^\dagger_k$ is well defined on the whole space. Moreover, for sufficiently large $k$ we expect that  $S \approx S \Gamma \Gamma^\dagger_k$ in the sense that  $\vvvert S - S \Gamma \Gamma^\dagger_k \vvvert$ becomes arbitrarily small.
 Let
 $$
\Gamma  := \sum_{i=1}^\infty {\lambda_i} e_i \otimes e_i,
$$
denote the spectral decomposition of the operator $\Gamma$, with eigenvalues $\lambda_1 \ge \lambda_2 \ge ... > 0$ and corresponding eigenfunctions $e_1, e_2,...$. A 
 typical example of a regularized inverse operator is given by the spectral cut-off regularizer
$$
\Gamma^\dagger_k := \sum_{i=1}^k \frac{1}{\lambda_i} e_i \otimes e_i,
$$
which evidently has operator norm $\vvvert \Gamma^\dagger_k \vvvert_\mathcal{L} = \lambda_k^{-1}< \infty$. We also point out that  $ \Gamma \Gamma^\dagger_k =: \Pi_{ k }$, where $\Pi_{ k }$ is the projection on the space spanned by the first $k$ eigenfunctions $e_1,...,e_k$ of $\Gamma$.  Notice that, if this was the whole problem, we could simply choose a large, but finite $k$ and receive an arbitrarily precise approximation of $S$ via $S \Gamma \Gamma^\dagger_k$. However, in 
practice neither the true expectation $\E Y_n \otimes X_n$, nor the true covariance operator $\Gamma$ are known and have to be estimated from the data. For this purpose  we define 
\begin{equation} \label{h2}
\hat \Gamma_N := \frac{1}{N} \sum_{n=1}^N X_n \otimes X_n \quad
\end{equation}
as the standard estimate of the covariance operator $\Gamma$ and $ \frac{1}{N} \sum_{n=1}^N Y_n \otimes X_n$ as  estimate of $\E Y_n \otimes X_n$.
This gives an empirical analogue of equation 
\eqref{inverse_characterization}, that is 
\begin{equation} \label{h1}
 \frac{1}{N} \sum_{n=1}^N Y_n \otimes X_n = S \hat \Gamma_N + U_N ,
\end{equation} 
where  
\begin{equation} \label{def_U}
  \quad U_N := \frac{1}{N} \sum_{n=1}^N \varepsilon_n \otimes X_n   
\end{equation}
is  a remainder term, arising from \eqref{model_2}.  Note that  the  identity~\eqref{h1} provides a way of estimating  $S \Gamma$. 
We define the empirical version of the regularized inverse by

\begin{equation} \label{h3}
\hat \Gamma^\dagger_k := \sum_{i=1}^{k } \frac{1}{ \hat \lambda_i}  \hat e_i \otimes  \hat e_i,
\end{equation}
where $\hat \lambda_1 \ge \hat \lambda_2 \ge ... \ge 0$ are the ordered eigenvalues of $\hat \Gamma_N$ and $\hat e_1, \hat e_2,...$ the corresponding eigenfunctions. An estimator of the operator 
$S$ is now given by
\begin{equation} \label{h4}
        \hat{S}_N 
    :=  
        \frac{1}{N} \sum_{n=1}^N Y_n \otimes X_n \hat \Gamma_k^\dagger
    =
        S \hat \Gamma_N \hat \Gamma^\dagger_k + U_N \hat \Gamma^\dagger_k
    = 
        S \hat \Pi_{ k } + U_N \hat \Gamma^\dagger_k~,
\end{equation}
where $\hat \Pi_{ k }$ is the projection on the subspace spanned by the the first $k$ eigenfunctions of 
the empirical covariance operator $\hat \Gamma_N$. This  equation differs notably from the ideal 
$
S \Gamma \Gamma^\dagger_k 
=
S \Pi_{ k }
$
by the noise term $ U_N \hat \Gamma^\dagger_k$ (which makes it a statistical inverse problem; see the discussion in the introduction). If $k$ is large compared to $N$ this remainder can  potentially spoil  the  estimate, because the noise $U_N$ is amplified by the regularized inverse $\hat \Gamma^\dagger_k$. Consequently the solution of the  inverse problem as described in model~\eqref{model_2},  features a trade-off between regularization parameter $k$ and sample size $N$. 
\begin{comment}
For illustrative purposes, consider the situation  where $k \in \N$  is fixed and  the sample size $N$ diverges to infinity. In this case  the noise term vanishes and we get from \eqref{h4} 
%
$
    S \hat \Gamma_N \hat \Gamma^\dagger_k + U_N \hat \Gamma^\dagger_k
\approx 
    S \Gamma   \Gamma^\dagger_k
=
    S \Pi_{ k }
$. 
Consequently, the accuracy of recovery of $S$ is  limited by the size of $k$ and in general $\hat S_N$  is  not a consistent estimator  for $S$. On the other hand, if  $N$  is fixed  and  $k \to \infty$, the variance of the estimator $\hat S_N$ is amplified ever more. In this situation we encounter nothing but noise, because the inflated error term dominates the  estimate. Hence to guarantee successful estimation (and inference) a careful choice of regularization depending on the sample size is necessary.
\end{comment}

As a corollary of our later discussion we will get a consistency result for $\hat S_N$ under suitable regularity conditions. For works specifically aimed at reconstructing the operator $S$ see, for instance, \cite{hall2007}, \cite{benatia2017} and \cite{masaaki2017}.

\section{Statistical inference for the location of $S$} \label{sec3}

In this section we introduce the concept of relevant hypotheses for the location of $S$ and discuss the assumptions that are made throughout this paper. Furthermore we revisit the problems in deriving a weak convergence result for the estimator $\hat S_N$ as  described in  \cite{cardot2007}, \cite{crambes2013} and suggest a new technique - the smoothness shift - to grapple with them. Based on this idea, we establish an invariance principle for the estimated
distance $\vvvert \hat S_N-S_0 \vvvert^2$, which is used to 
develop pivotal statistics for testing relevant hypotheses.

\subsection{Relevant differences in the slope}
\label{Subsection_relevant_hypothesis}

A typical concern in the context of model \eqref{model_1} is the comparison of the true slope $S$  with 
some hypothesized operator $S_0 \in \mathcal{S}(H_1, H_2)$. This problem is often addressed by constructing statistical tests for the hypotheses   
\begin{align} \label{hd2}
   H_0: S=S_0 \quad  {\rm versus } \quad H_1 : S \not = S_0 ~.
\end{align}
These hypotheses may for instance  be used with $S_0 = 0$, to determine the explanatory power of the model, or with a slope $S_0$ from a theoretical model. Various tests have been devised for these (or related) hypotheses, such as in \cite{cardot2003b, cardot2004}, where the cross covariance operator $\mathbb{E}Y_1 \otimes X_1= S \Gamma$ is used to test the mathematically equivalent null hypothesis $S \Gamma = S_0 \Gamma$. In a similar spirit 
\cite{hilgertetal2010}
propose minimax optimal adaptive tests based on projections of $Y$ onto the  principal components of $\Gamma$ or \cite{kong2016} employ traditional tests (such as the $F$-test) on finite dimensional subspaces, to validate model fit. \\
Although from a decision theoretical perspective  all of these methods define  consistent tests for the hypotheses \eqref{hd2}, 
they have the drawback of telling us little about the actual proximity of the operators $S$ and $S_0$. For example a test for $H_0$, based on the quantity $\vvvert S \Gamma - S_0 \Gamma \vvvert$ %(as in \cite{cardot2003}) 
is difficult to interpret, as $\vvvert S \Gamma - S_0 \Gamma \vvvert$ may be arbitrarily small, while in fact the true difference $\vvvert S- S_0\vvvert$  is  arbitrarily large. 
In particular if a user  decides to perform data analysis  under the assumption $S=S_0$  after a test has not rejected the hypothesis 
$S \Gamma = S_0 \Gamma $, there is no guarantee that $S_0$ is indeed a good approximation of $S$. This insight has motivated some of the contemporary approaches to confidence regions for $S$ (see the discussion in the introduction), where even a slower than parametric convergence rate is accepted, in return for an inference method, based on the original slope operator   $S$.\\
In this paper we take up this insight and  base statistical inference directly 
on the measure  $\vvvert S- S_0\vvvert$. Evidently the point hypothesis in \eqref{hd2} is equivalent to $\vvvert S- S_0\vvvert =0$. However in this work, we want to investigate the ``relevant hypotheses'', given by
\begin{eqnarray} \label{relevant_difference}
&&H_0^\Delta: \vvvert S-S_0\vvvert ^2 \le \Delta \quad {\rm versus }  \quad H_1^\Delta:  \vvvert S-S_0\vvvert ^2 > \Delta, \\
\text{and} ~~~~~~~~~~~~~~\nonumber &&\\ 
\label{equiv}
&& H_0^{\tilde \Delta}: \vvvert S-S_0\vvvert ^2 \ge \tilde \Delta \quad {\rm versus }  \quad H_1^{\tilde \Delta}:  \vvvert S-S_0\vvvert ^2  < \tilde \Delta,
\end{eqnarray}
where $\Delta>0$ and  $\tilde \Delta>0$ are predetermined thresholds. Our suggestion  to replace the ``classical'' 
hypotheses in \eqref{hd2} by hypotheses of the form \eqref{relevant_difference} or  \eqref{equiv} has theoretical as well as practical reasons.\\
From a theoretical perspective, testing  exact equality of $S$ and $S_0$ (both of which are infinite dimensional objects) might be questionable, because it is rarely believed that the hypothesized slope coincides perfectly with the true one. Therefore, testing $H_0$ means testing a hypothesis, which is essentially known to be false. This point is important, because any consistent test will detect any arbitrarily small deviation from $H_0$ if the sample size is sufficiently large \citep[see][]{berkson1938} and thus we  expect any consistent test for $H_0$ to eventually reject the hypothesis. This problem is evaded by the  consideration of relevant hypotheses~\eqref{relevant_difference}, which only refer to  sufficient proximity of $S$ and $S_0$. \\
We also believe that the relevant hypotheses are more congruent with common interests of users, who are less concerned with perfect equality than with the practical issue of comparable performance. Often users are willing to trade  - at least to some extend - statistical precision for a simpler model. In this sense the thresholds $\Delta, \tilde \Delta$ in the relevant hypotheses can be understood as the largest deviation between $S$ and $S_0$, which is still acceptable for the user. This also highlights that the choice of the threshold will depend on the application in hand and is not an a priori question.
%, but  often there is some supposition  that the distance between $S$ and $S_0$ is small, such that it might be reasonable to work with $S_0$ instead of $S$. In this case it is questionable  to test the hypotheses in \eqref{hd2} as it is believed that the null hypothesis does not hold. Therefore we propose to test the more realistic hypotheses in \eqref{relevant_difference},  where
%the threshold $\Delta$  defines a scientifically relevant deviation from the postulated operator $S_0$.
%Second, note that the formulation \eqref{relevant_difference} avoids the consistency problem:   
Finally we point out that a  formulation of the hypotheses in  the form 
 \eqref{equiv} might be preferred if one is interested to work under the assumption $S=S_0$. If  the null hypothesis $H_0^{\tilde \Delta}$ is rejected at level $\alpha$, the risk of erroneously assuming $\vvvert S-S_0\vvvert ^2  \le \tilde \Delta$,
is controlled, which is not possible using the ``classical'' hypotheses in \eqref{hd2}, because there is no symmetry in the problem.\\
Although  the hypotheses~\eqref{relevant_difference} and \eqref{equiv} 
are different with respect to their statistical interpretation it
will become clear later that from a mathematical point of view 
they  are in some sense equivalent. Therefore, and also for the sake of brevity, we restrict ourselves 
to the development of testing procedures for the hypotheses in \eqref{relevant_difference} and 
denote the null hypothesis as ``no relevant deviation from $S_0$''.

\subsection{Assumptions}\label{Subsection_assumptions}

The theoretical results of this paper require several assumptions, which are explained and illustrated in this section. 

Recall that a  
stationary sequence 
$\{Z _j\}_{j\in \Z}$ of random variables is called $\phi$-{mixing}, if $\lim_{k \to \infty }\phi(k) =0$,
where 
\begin{align*}
		\phi(k) :=& \sup_{h \in \Z} \sup\big \{ 
			{|\mathbb{P}(F| E)} - \mathbb{P}(F)|
			\, : \,  E \in \sigma(Z _1,\ldots ,Z _h) , F \in 
		 \sigma(Z_{h+k},Z_{k+h+1}, \ldots ) , ~\mathbb{P}(E)>0 \big \}.
\end{align*}
denotes  the  {  $\phi$-dependence coefficients} and
$\sigma(Z _h,...,Z _k),
$
is  the $\sigma$-algebra   generated by $Z _h,\ldots  ,Z _k$ \citep[see for instance][]{dehling}.
% We now state our assumptions:

\begin{assumption} \label{ass31} 
~~\\
{\rm 
(1) \textit{Smoothness:} For some $\beta \ge 0$ the operators $S$ and $S_0$ are elements of the smoothness class
$$
\mathcal{C}(\beta,\Gamma):=\big\{R \Gamma^{\beta}: R \in \mathcal{S}_2(H_1, H_2) \big\}.
$$ 
(2) \textit{Moments:} There exists some $\kappa>0$, such that $\E \|X_1\|^{4+\kappa}$, $\E \|\varepsilon_1\|^{4+\kappa}<\infty$.
\\
(3) \textit{Dependence:} The sequence of random functions $\{(X_n, \varepsilon_n)\}_{n \in \Z}$ is centered, strictly stationary and  $\phi$-mixing, such that $$\phi(1)<1 ~~
\text{ and  } ~~\sum_{h \ge 1} \sqrt{\phi(h)}<\infty
$$

(4) \textit{Coefficients}: There exists a finite constant $C>0$, s.t.\ the inequality  $\E | \langle X_1, e_j \rangle|^4 \le C (\E | \langle X_1, e_j \rangle|^2)^2$ holds for any $j \in \N$. 

(5)  \textit{Weak exogeneity:} $\E \varepsilon_1 \otimes X_1=0$.

(6) \textit{Decay of eigenvalues and eigengaps:} For some $ \SM >0$ and large enough $C >0$, the eigenvalues of the covariance operator $\Gamma$ satisfy
$$
\lambda_k \le C k^{-{\SM}} \quad \quad \textnormal{and} \quad \quad \lambda_k-\lambda_{k+1} \ge C^{-1} k^{-{\SM}-1} \quad  \forall k \in \N.
$$

(7)  \textit{Rates of regularization:} The regularization parameter $k=k(N)$ is chosen such that for some $\delta>0$
$$
\frac{k^{\SM+1+\delta}}{\sqrt{N}} \to 0 \quad \quad \textnormal{and} \quad \quad \frac{k^{\SM \beta}}{\sqrt{N}} \to \infty.
$$
}
\end{assumption}

\begin{rem} \label{rem1}
~~ \\ 
{\rm
(a) Assumption~(1) is a smoothness condition on the slope operators $S, S_0$, w.r.t.\ 
 the principal components of $\Gamma$. To see this let $S = R\Gamma^\beta$ and $x \in H_1$. It follows that
$ 
    S x
= 
    R \Gamma^\beta x 
=  
    R y,
$
where $y= \sum_{r \in \N} e_r (\lambda_r^\beta \langle x, e_r \rangle) $. 
Evidently the $L^2$-coefficients $(\lambda_r^\beta \langle x, e_r \rangle)$ of $y$  decay faster than those of $x$, as they are weighted by a power of the decaying eigenvalues. In this sense $y$ is smoother than $x$ and a larger value of  $\beta$ translates into lighter coefficients and thus more smoothing.
In this way $Sx$ can be understood as the application of an integral operator $R$ to a smoothed version of $x$.  
 Assumption \ref{ass31}(1) was also considered in \cite{benatia2017} in their study of the Tikhonov regularization, where it was denoted by the common label of {\it source condition}. At the beginning of their Section~3 the smoothing effect of $\Gamma^\beta$ is explored by various examples.\\
In the following calculations we demonstrate that Assumption \ref{ass31}(1) can be translated into {\it fast decaying tails} of
the operator $S$ '', which is another standard way of stating smoothness in the literature. Consider the application of $R$ to a basis function $e_q$ of $\Gamma$
$$
R e_q = R \Gamma^{ \beta} \Gamma^{- \beta} e_q = S \Gamma^{- \beta} e_q.
$$
Notice that $\Gamma^{- \beta} e_q= \lambda_q^{-\beta} e_q$ is indeed well defined. We can now express $R e_q $ as
$$
R e_q=S\Gamma^{-\beta} e_q =\Big[ \sum_{i,j \in \N } s_{i,j} e_i \otimes e_j  \Big] \Big[ \sum_{k\in \N} \lambda_k^{-\beta} e_k \otimes e_k  \Big] e_q = \sum_{i\in \N} s_{i, q} \lambda_q^{-\beta} e_i ,
$$
where $s_{i,j} := \langle S, e_i \otimes e_j \rangle$ (with the inner product on the space of Hilbert--Schmidt operators, see Section ~\ref{Subsection_Operators}).  Now the squared Hilbert--Schmidt norm of $R$ equals
\begin{align} \nonumber 
 \infty>\vvvert R \vvvert^2=&\sum_{q \in \N} \langle Re_q, R e_q \rangle =\sum_{q \in \N} \langle \sum_{i \in \N} s_{i,q} \lambda_q^{-\beta} e_i, \sum_{l \in \N} s_{l,q} \lambda_q^{-\beta} e_l  \rangle =  \sum_{q,l \in \N} \lambda_q^{-2\beta} s_{l,q}^2 \\
 =& \sum_{q,l \in \N} \lambda_q^{-2\beta} \langle S e_q, e_l \rangle^2 = \sum_{q \in \N}  \lambda_q^{-2\beta} \|S e_q\|^2 , \label{hd4} 
\end{align}
where  we have used Parseval's identity in the last step. In the scalar response model \eqref{functions_on_regressors} one has $\|S e_q\|^2 = \langle \varphi_S, e_q \rangle^2$. Thus the summability 
in \eqref{hd4} is a smoothness condition for $\varphi_S$. In this form it has been used by \cite{hall2007} (see  equation~(3.3) in that paper). In the more general model \eqref{model_1} the decay of $\|S e_q\|^2 $ was considered as a smoothness condition  in \cite{crambes2013} (see their Definition~3). 
\medskip 

(b) Assumptions~\ref{ass31}(2) - (5)  are required to derive a weak convergence result stated in Theorem~\ref{theorem_1}. The existence of moments of larger order than $4$ is typical for proving second order, weak invariance principles  \citep[it corresponds to the assumption of more than second moments for the first order; see][]{berkes2013}. The mixing assumption  is weaker than those in the related literature, where almost exclusively  i.i.d.\ observations are considered, \citep[see][among others]{hall2007, crambes2013, benatia2017, kato2019, babii2020}. 
%It allows the application of our methodology to time series data  \citep[see][for the study of functional time series under mixing assumptions]{aston2012}. 
%While we formulate our theory for $\phi$-mixing data, some modifications of our proofs show, that the investigation of $\alpha$-mixing data is also possible (see Remark~\ref{remark_centering}). 
Assumption \ref{ass31}(4)  regarding the moments of
the  coefficients
$\langle X_1,  e_j \rangle $ is standard in the literature \citep[see for example][]{hall2007, crambes2013, kato2019} and is needed for technical reasons. We use it in the proof of Lemma \ref{Lem_Bounds-NonSmoothed-Var-and-Bias-of-S_N}, part ii).
Assumption \ref{ass31}(5) regarding the  exogeneity is again weaker than in most of the literature. Here often strong exogeneity is required
(see the literature cited before), where the work of  \cite{benatia2017} constitutes an important exception.
\medskip 

(c)  Assumption \ref{ass31}(6) guarantees a  polynomial decay rate for the eigenvalues of $\Gamma$, that is   $\lambda_k \sim k^{-\SM}$. More important than the precise rate of decay is the assumption on the eigengaps, which have to be controlled for identifiability reasons. Assumptions  of this type are  standard in the literature, in particular in the analysis of spectral cut-off estimators
\citep[see,][among others]{hall2007,qiao2018functional}, even though they are sometimes made implicitly  \citep[see Lemma~12 in][]{crambes2013}. \\
 The two decay rates  in Assumption \ref{ass31}(7) expose the  trade-off inherent in the choice of $k$. On the one hand $k$ has to increase slowly enough, such that the $k$-th eigenvalue $\lambda_k$ can be distinguished from $\lambda_{k+1}$. This means that the $k$-th eigengap of size $k^{-\SM-1}$ is of larger order than the estimation error of size $1/\sqrt{N}$. Our assumption is almost sharp in the sense that we assume $k^{\SM+1}/\sqrt{N}$ to decay at some arbitrarily slow polynomial rate in $N$. We use this additional leverage to derive not only a CLT but a stronger weak invariance principle, where remainders have to be controlled uniformly; see Lemma~\ref{Lem_Bounds-NonSmoothed-Var-and-Bias-of-S_N}. A sharp version has been used for confidence bands in the scalar response model
 by  \cite{kato2019}. On the other hand, $k$ has to increase fast enough, such that the asymptotic bias is  negligible, more precisely
$$
    \vvvert S-S\Pi_{ k } \vvvert= \vvvert R \Gamma^\beta [Id-\Pi_{ k }]\vvvert \le \vvvert R \vvvert \vvvert  \Gamma^\beta [Id-\Pi_{ k }] \vvvert_\mathcal{L} =   \vvvert R \vvvert \lambda_{k+1}^\beta =\mathcal{O}(k^{-\SM \beta})= o(1/\sqrt{N}).
$$
It can be shown that the above bound is sharp for general operators and hence the bias rate cannot be improved upon. Notice that the two Assumptions on $k$ can be simultaneously fulfilled if and only if
$$
 \beta> 1+1/\SM.
$$
}

\end{rem}

\subsection{Main results}
\label{Subsection_weak_convergence}

In order to develop  a statistical test for the relevant hypotheses   defined in  \eqref{relevant_difference}  it is reasonable 
to estimate   the difference $\vvvert S - S_0 \vvvert^2$. A natural estimator is given by $\vvvert \hat S_N-S_0 \hat \Pi_{ k } \vvvert ^2$.  While it is also possible to replace $S_0 \hat \Pi_{ k }$ by $S_0$ in the subsequent theory, we prefer to work with $S_0 \hat \Pi_{ k }$  as  it does not seem sensible to compare $S_0$ along dimensions to $S$, where no estimate for $S$ exists (this common sense approach is also supported by simulations).  
In order to define a consistent and (asymptotic) 
 level-$\alpha$ test,  we are interested 
in  the weak convergence of the difference
\begin{equation} \label{functional_convergence}
\sqrt{N} \big( \vvvert \hat S_N-S_0 \hat \Pi_{ k } \vvvert ^2 - \vvvert S-S_0\vvvert ^2   \big).
\end{equation}
The standard approach to this problem  would be to: first establish weak convergence of the  difference $\sqrt{N} (\hat S_N - S)$ 
in the space $\mathcal{S}(H_1,H_2)$; then deduce weak convergence of the test statistic in \eqref{functional_convergence} by applying the Delta method (see Section~3.9 in \cite{vandervaart1996})  to the mapping
$ S  \to \vvvert S-S_0\vvvert ^2  $.  Notice that, using the OLS estimator, this method works for finite dimensional linear regression.
However this approach fails  in the context of functional regression problems, as 
it is not possible to find a standardizing  
 sequence, say $\{a_N\}_{N\in \N}$,  such  that 
 the difference $a_N( \hat S_N - S) $   converges weakly to a non-degenerate limit, if  $k$ converges to infinity    with the sample size, which is necessary to obtain  an asymptotically vanishing bias \citep[see, for example,][]{crambes2013}.
 More precisely,  if  $k$  is fixed  one can prove  that $\sqrt{N}(\hat S_N - S)\Pi_k$ converges weakly to a Gaussian random vector. A similar result was derived by \cite{benatia2017} for a different regularization method. However  these authors likewise concluded that for decaying regularization, i.e.\ $k \to \infty$ as $N \to \infty$ the sequence $\sqrt{N}(\hat S_N - S)$ has a degenerate limit 
 caused by an inflation of the error variance.

 Nevertheless the fact that no weak convergence of $\sqrt{N} (\hat S_N - S)$  in the space  $\mathcal{S}(H_1, H_2)$  can be  established does not  necessarily imply that the difference in  \eqref{functional_convergence} cannot converge weakly. Indeed we will demonstrate  that the  mapping
$ S  \to \vvvert S-S_0\vvvert ^2  $  has a smoothing effect on $\hat S_N$.  Therefore the inflation of the observation error $U_N$ (defined in \eqref{def_U}) is compensated  and it is  possible to  establish weak convergence of 
\eqref{functional_convergence} with a normally distributed limit. The precise statement will be given in Proposition~\ref{proposition_1} below. To get an intuition how this smoothing works
note  that by  the third binomial formula in Hilbert spaces we have  
$$
\sqrt{N} \big( \vvvert \hat S_N-S_0 \hat \Pi_{ k } \vvvert ^2 - \vvvert (S-S_0)  \Pi_{ k }\vvvert ^2   \big) =\langle   \sqrt{N}[ \hat S_N-S_0 \hat \Pi_{ k }- (S-S_0)  \Pi_{ k }], \hat S_N-S_0 \hat \Pi_{ k } +  (S-S_0)  \Pi_{ k } \rangle.
$$
After some careful bounding of the error terms (recall that the left side of the inner product asymptotically degenerates), we can show that this  equals
$$
2 \sqrt{N}\langle  \hat S_N-S_0 \hat \Pi_{ k }- (S-S_0)  \Pi_{ k },  S-S_0  \rangle + o_\PP(1).
$$
By  Assumption~\ref{ass31}$(1)$    there exist  operators $R, R_0 \in \mathcal{S}(H_1, H_2)$, such that $S=R \Gamma^\beta, S_0 = R_0 \Gamma^\beta$. Hence we can perform the following {\it smoothness shift}, moving smoothness in the form of $\Gamma^\beta$ from the second to the first component of the inner product, i.e.
\begin{align*}
        2 \sqrt{N}\langle  \hat S_N-S_0 \hat \Pi_{ k }- (S-S_0)  \Pi_{ k },  S-S_0  \rangle  
    = 
    & 
        2 \sqrt{N}\langle  \hat S_N-S_0 \hat \Pi_{ k }- (S-S_0)  \Pi_{ k },  [R-R_0] \Gamma^\beta  \rangle 
    \\
    = & 
        2 \sqrt{N}\langle  [\hat S_N-S_0 \hat \Pi_{ k }- (S-S_0)  \Pi_{ k } ]\Gamma^\beta  ,  R-R_0\rangle.
\end{align*}
It turns out that the smoothing effect of $\Gamma^\beta$ on the left stops the error inflation and thus weak convergence 
to a non-degenerate   and  (with some technical linearization) normally distributed  limit can be proved. Intuitively the smoothing works, because
$$
        \hat S_N \Gamma^\beta 
    = 
        \frac{1}{N} \sum_{n=1}^N Y_n \otimes X_n \hat \Gamma^\dagger_k \Gamma^\beta 
    \approx  
        \frac{1}{N} \sum_{n=1}^N Y_n \otimes X_n \Gamma^{\beta-1},
$$
i.e.\ the regularized inverse $\hat \Gamma^\dagger_k$ and the shifted operator $\Gamma^\beta$ "cancel out" to $\Gamma^{\beta-1}$, thus eliminating the pathology of the asymptotically unbounded operator $\hat \Gamma^\dagger_k$. If the term on the right is centered and standardized by $\sqrt{N}$ it is asymptotically normal.
The price we pay for this non-standard approach is a more elaborate proof, where many  difficult remainders have to be controlled. As announced we now formulate the precise result.

\begin{prop} \label{proposition_1}
Under the Assumptions presented in Section~\ref{Subsection_assumptions}, it holds that
\begin{equation} \label{h5} 
    T_N
= 
    \sqrt{N} \left( 
        \vvvert \hat S_N-S_0 \hat \Pi_{ k } \vvvert ^2 
        - 
        \vvvert (S-S_0)\Pi_{k}  \vvvert ^2   
    \right) 
\stackrel{ d }{ \to }
    \mathcal{N}(0, \tau^2),
\end{equation}
where 
the long-run variance $\tau^2$ is defined by 
\begin{align} \label{def_long_run_variance}
    \tau^2 
:=&
    4 \Bigg\{ 
        \sum_{h \in \mathbb{Z}} 
        \mathbb{E} \Big[ 
            \langle (R - R_0) L[ X_0 \otimes X_0 - \Gamma], 
            R - R_0 \rangle \,
            \langle (R - R_0) L[ X_h \otimes X_h - \Gamma], R - R_0 \rangle \Big]
\\
& 
    \quad \, 
    +  2 \mathbb{E} 
        \Big[ \langle ( R - R_0) L[ X_0 \otimes X_0 - \Gamma], 
        R - R_0 \rangle \,
        \langle \varepsilon_h \otimes X_h \Gamma^{\beta-1}, R - R_0 \rangle \Big] \nonumber
\\
& 
    \quad \, \, \, \, +  
    \mathbb{E} \Big[ 
        \langle \varepsilon_0 \otimes X_0 \Gamma^{\beta-1},
        R - R_0 \rangle \,
        \langle \varepsilon_h \otimes X_h \Gamma^{\beta-1},
        R - R_0 \rangle \Big] \Bigg\}. \nonumber
\end{align}
Here the map $L$ is given in Definition~\ref{Def_LinerazationOperat} of the Appendix.
\end{prop}

\begin{comment}
\begin{align} \label{def_long_run_variance}
    \tau 
:=&
    4 \big\{ \langle L(R-R_0) \Psi_1 L (R-R_0)^{ \ast }, (R-R_0) \otimes (R-R_0)^{ \ast } \rangle
\\
& 
    +  2 
    \langle L(R-R_0) \Psi_2 \Gamma^{\beta-1}, 
        (R-R_0) \otimes (R-R_0)^{ \ast } 
    \rangle \nonumber
\\
& + \langle \Gamma^{\beta-1} \Psi_3 \Gamma^{\beta-1}, (R-R_0)^{ \ast } \otimes (R-R_0) \rangle \big\}. \nonumber
\end{align}
Here $R$ and $R_0$ are the operators in Assumption~\ref{ass31}(1) and  the operators $\Psi_1, \Psi_2, \Psi_3$ are defined by 
$$
\Psi_1 := \sum_{h \in \Z} \E X_0 \otimes X_0 \otimes X_h \otimes X_h, \quad \Psi_2 := \sum_{h \in \Z} \E X_0 \otimes X_0 \otimes \varepsilon_h \otimes X_h, \quad \Psi_3 := \sum_{h \in \Z} \E \varepsilon_0 \otimes X_0 \otimes \varepsilon_h \otimes X_h.
$$
\end{comment}

Using Proposition~\ref{proposition_1}, we could in principle construct a test for the hypothesis of no relevant deviation, presented in \eqref{relevant_difference}, by rejecting 
the null hypothesis, whenever 
\begin{equation} \label{primitive_test}
\sqrt{N} \big( \vvvert \hat S_N-S_0 \hat \Pi_{ k } \vvvert ^2 - \Delta \big)> \tau \Phi^{-1}(1-\alpha),
\end{equation}	
where $\Phi^{-1}$ is the quantile function of a standard normal distribution and $\alpha \in (0,1)$ denotes the nominal level. This decision yields indeed a test which is asymptotically consistent and keeps its nominal level asymptotically. To see  this  we use the  expansion
\begin{align} \label{decomposition}
 \sqrt{N} \big( \vvvert \hat S_N-S_0 \hat \Pi_{ k } \vvvert ^2 - \Delta \big) & = T_{1N} + T_{2N} +T_{3N}~,
 \end{align}
 where  
 \begin{align*} 
 T_{1N} & =   \sqrt{N} \big( \vvvert \hat S_N-S_0 \hat \Pi_{ k } \vvvert ^2 -  \vvvert (S-S_0)  \Pi_{ k }\vvvert^2\big),  \\
 T_{2N} &=    \sqrt{N} \big( \vvvert (S-S_0)  \Pi_{ k } \vvvert ^2 -  \vvvert S-S_0\vvvert^2\big), \\
 T_{3N} &=  \sqrt{N} \big(  \vvvert S-S_0\vvvert^2 - \Delta\big). 
\end{align*}
By Proposition \ref{proposition_1} the first term $T_{1N} $  in \eqref{decomposition} converges weakly to a centered normal distribution with variance $\tau^2 $. The   term  $T_{2N} $ is the bias and asymptotically vanishes (see discussion of Assumption \ref{ass31}(7)). 
The third term $T_{3N} $
 is also deterministic. In the \textit{interior} of the null hypothesis, that is $  \vvvert S-S_0\vvvert^2 < \Delta$, 
  it converges to $-\infty$ and thus asymptotically no rejection occurs for $N \to \infty$. On the \textit{ boundary }  of the hypothesis, that is   $  \vvvert S-S_0\vvvert^2 = \Delta$,
   it vanishes and we get $ \sqrt{N} \big( \vvvert \hat S_N-S_0 \hat \Pi_{ k } \vvvert ^2 - \Delta \big) = T_{1N}  + o_\PP(1)  \to \mathcal{N}(0, \tau^2)$.  Consequently,  the test \eqref{primitive_test} 
has   asymptotic level $\alpha$ in this case.  Notice that the bias $T_{2N} $ is always non-positive which means that small choices of $k$ (resulting in larger bias) invariably make the test more conservative. % and do not cause inflated rejection at the boundary .  
Finally,   under the alternative the   term  $T_{3N} $ diverges to $\infty$ and thus rejection occurs with probability 
converging  to $1$ (asymptotic consistency). In the following remark we briefly explain how the decomposition \eqref{decomposition} can be used for a more
refined analysis with respect to  local alternatives.

\begin{rem} \label{remark_local_alternatives}

{\rm 
Consider the local alternative of $S= \tilde S +c H/\sqrt{N}$, where both $H, \tilde S \in \mathcal{C}(\beta,\Gamma)$ (see Assumption \ref{ass31}(1)) are operators, $c>0$ is a scaling factor and   $\Delta = \vvvert \tilde{S} - S_0 \vvvert ^2$ demarcates the boundary of the hypothesis. Furthermore assume that $\langle \tilde S-S_0, H \rangle =: \Lambda >0$. The last requirement is necessary, such that we are indeed under the alternative ($\Lambda < 0$ corresponds to the hypothesis) and thus
$$
 \vvvert S- S_0 \vvvert ^2 =  \vvvert \tilde{S} - S_0 \vvvert ^2 + 2 c \langle H, \tilde S-S_0 \rangle /\sqrt{N} +c^2\vvvert H\vvvert ^2/N = \Delta + 2 c \Lambda /\sqrt{N} + \mathcal{O}(1/N) >\Delta.
$$
Now suppose that the Assumptions of Proposition~\ref{proposition_1} hold. We apply the test, defined in \eqref{primitive_test} in this situation and let 
$
p_{H_1^\Delta}(c) 
$
denote the probability of rejection. It then follows that
$$ \lim_{N \to \infty}  p_{H_1^\Delta}(c)> \alpha \quad \Leftrightarrow \quad c>0 \quad \quad  \textnormal{and} \quad\quad \lim_{c \to \infty} \lim_{N \to \infty} p_{H_1^\Delta}(c) = 1.$$
Both results follow from the decomposition \eqref{decomposition}. It is not difficult to show that $T_{1N}$ converges to the same normal distribution as in the case of $S=\tilde S$, that $T_{2N}$ asymptotically vanishes and that $T_{3N} = 2 c \Lambda +o(1)>0$ is non-vanishing and (asymptotically) scales linearly with $c$. Consequently the test in \eqref{primitive_test} is able to detect local alternatives converging to the null hypothesis 
at a rate of $1/\sqrt{N}$.
%, which is faster than the alternatives presented in the literature, such as \cite{kato2019} and \cite{babii2020}. 
We also point out that all subsequently presented, self-normalized tests directly inherit this property, as the numerator of the normalized statistic can be decomposed as in \eqref{decomposition}. 
}
\end{rem}

Note that the test  \eqref{primitive_test} provides an attractive decision rule for the hypothesis~\eqref{relevant_difference} supposing that a reliable
estimate of the variance $\tau^2$  is available.  Unfortunately, even in the case of independent observations this quantity  is painfully complex to estimate. It requires not only estimation of $\Gamma$, but also of the fourth order structure of regressors and errors, a linearization map $L$ depending on all eigenvalues and eigenfunctions of the operator $\Gamma$ (an object which depends inversely on the small eigengaps) as well as knowledge about the operators 
$R$ and $R_0$ in Assumption~\ref{ass31}(1). What is difficult for i.i.d.\  data is almost infeasible  in the case of dependent data. In this case $\tau^2$ is a long-run variance, which requires besides the estimation of all the  mentioned  entities the determination of a bandwidth, capturing the sequential dependence of the regressors and errors. \\
Given the impracticality and instability of such an estimate we pursue the different approach of self-normalization in the following  section. The technical prerequisite for this procedure is the derivation of a weak invariance principle, generalizing Proposition~\ref{proposition_1}. For this purpose we introduce a sequential version
of the statistic $\hat S_N$ which  is defined similarly as the original, with the difference that - instead of all $N$ observations - 
  only the observations   $(X_{1},Y_{1}), \ldots , (X_{\lfloor \FP N \rfloor},Y_{\lfloor \FP N \rfloor})$ for $\FP \in (0,1]$ are used for  estimation.
To be precise we define the sequential covariance estimator
\begin{equation} \label{Eq_DefSeqCovOp}
        \hat \Gamma_N[\FP]
    =  
        \frac{1}{N} \sum_{n=1}^{\lfloor \FP N \rfloor} X_n \otimes X_n.
\end{equation}
Furthermore, we define the sequential estimators of the eigenvalues and eigenfunctions of $\Gamma$, denoted by $\hat{\lambda}_i[\FP], \hat e_i[\FP]$ as the eigenvalues and eigenfunctions of the 
operator $\hat \Gamma_N[\FP]$ (where the eigenvalues are again assumed to be in non-increasing order). With these estimators we set
\begin{equation}  \label{Eq_DefSeqPseudoInvs-and-SeqProj}
	    \hat \Gamma^\dagger_k[\FP]
	= 
	    \sum_{i=1}^k \frac{1}{\hat{\lambda}_i[\FP]} 
	    \hat e_i[\FP] \otimes \hat e_i[\FP] 
    \qquad   
        \textnormal{ and }  
    \qquad 
        \hat \Pi_{ k }[\FP]
    = 
        \sum_{i=1}^k \hat e_i[\FP] \otimes \hat e_i[\FP].
\end{equation}
Finally the sequential estimator of $S$ is given by
\begin{equation} \label{App_ref_3}
        \hat{S}_N [\FP]
    :=  
        \frac{1}{N} \sum_{n=1}^{\lfloor \FP N \rfloor} Y_n \otimes X_n \hat \Gamma^\dagger_k[\FP].
\end{equation}
Note that in the case of $\FP=1$ these estimators are identical to their non-sequential counterparts 
$\hat{\Gamma}_N $, $\hat \Gamma^\dagger_k$ and  $ \hat{S}_N $  defined in \eqref{h2}, \eqref{h3}  and \eqref{h4}, respectively, and that we do not adapt $k=k(N)$ to $\FP$. 
Throughout this paper we will  use   the notations $\hat \Gamma_N$  and   $\hat \Gamma_N [1]$  simultaneously.
We can now state the weak invariance principle generalizing Proposition~\ref{proposition_1}.

\begin{theo} \label{theorem_1}
Under our Assumptions~\ref{ass31}, 
%presented in Section~\ref{Subsection_assumptions}, 
it holds for any compact interval  $I \subset (0,1]$ that
$$
    \left\{ 
    \sqrt{N} \FP 
    \left( 
         \vvvert 
         \hat S_N[\FP] - S_0 \hat \Pi_{ k }[\FP] 
         \vvvert ^2 
        - 
        \vvvert ( S - S_0 )\Pi_{k} \vvvert ^2   
    \right)
    \right\}_{\FP \in I} 
\stackrel{ d }{ \to } 
    \{ \tau \mathbb{B} (\FP) \}_{\FP \in I},
$$
as $N \to \infty$, where $\mathbb{B}$ is a standard Brownian motion and 
the long-run variance $\tau^2$ is defined in \eqref{def_long_run_variance}. 
\end{theo}

\subsection{A pivotal test statistic } \label{Subsection_test_one_sample}

%Notice that for $I=\{1\}$ we get Proposition \ref{proposition_1}. Furthermore we want to point out that $I \subset (0,1]$  is bounded away from $0$. This is important, because if $\FP$ was arbitrarily small, $\lfloor N \FP \rfloor$ could not satisfy Assumption $6)$ for all $\FP$. 

In the last section we have derived a weak invariance principle for the estimated deviation measure in \eqref{functional_convergence}. While a central limit theorem is theoretically sufficient to construct a test for the hypothesis \eqref{relevant_difference}, as we have seen in the discussion of Proposition \ref{proposition_1}, the estimation of the long-run variance $\tau^2$ is infeasible in applications. 
% On the one hand it requires knowledge  
%of the unknown operators $R$ and $R_0$ in Assumption \ref{ass31}(1). On the other hand, even if these operators would be known, estimates  of the quantities $\Psi_1,$ $\Psi_2,$ and $\Psi_1,$ in \eqref{def_long_run_variance} would require further regularization, which would make any procedure based on  estimate of $\tau^2$ unstable.
 In this section we circumvent the problem of estimating  $\tau^2$, by a self-normalization approach, based on the weak invariance principle in Theorem \ref{theorem_1}. For this purpose,   define for 
 $0 < a<1$  the interval  $I = [a, 1]$, let  $\nu$ be a probability measure on $I$  and consider the normalizer
\begin{equation} \label{denominator}
    \hat V_N 
:= 
    \left\{ \int_I \FP^4 
        \left(  \vvvert \hat S_N[\FP]-S_0 \hat \Pi_{ k }[\FP] \vvvert ^2 
        -  
        \vvvert \hat S_N-S_0 \hat \Pi_{ k } \vvvert ^2 \right)^2 d\nu(\FP)  
    \right\}^{1/2}.
\end{equation}
% In the literature on self-normalization $\nu$ is typically chosen to be a (standardized) version of the Lebesgue measure \citep{shao2015} 
The next corollary is a consequence of Theorem~\ref{theorem_1} and the continuous mapping Theorem. It can be viewed as a standardized version of Proposition~\ref{proposition_1}.

\begin{cor} \label{corollary_2}
Suppose that the assumptions of Theorem~\ref{theorem_1} hold and that $\tau>0$. Then the weak convergence 
\begin{align}
    \label{hd5}
    \frac{ 
    \left( 
    \vvvert \hat S_N-S_0 \hat \Pi_{ k } \vvvert ^2 
    - 
    \vvvert ( S-S_0 ) \Pi_{ k } \vvvert ^2  
    \right) 
    }{
    \hat V_N
    } 
\stackrel{ d }{ \to } 
    W 
:= 
    \frac{
        \mathbb{B}(1)
    }{
        \left\{\int_I \FP^2 
        ( \mathbb{B}(\FP)-\FP \mathbb{B}(1))^2 d \nu(\FP) \right\}^{1/2}
    }
\end{align}
holds, where $\mathbb{B}$ is a standard Brownian motion  on the interval $[0,1]$. 
%$W$ is a pivotal distribution, only depending on $\nu$. 
\end{cor}

We point out that the quantiles of the distribution of $W$ can be readily simulated using the Fourier representation of the Brownian motion.
% (see, for instance, Table~1 in \cite{DetKokVol20}). 
A typical choice for the measure $\nu $ is  a discrete  uniform measure on the set $\{1/T,2/T,...,(T-1)/T, 1\}$ for some $T \ge 2$. 
Simulations suggest that the choice of $T$ has little impact on 
 the statistical performance of the resulting procedure, while - of course - smaller  values of $T$  yield computational advantages  (see Section \ref{sec6}). 

In view of \eqref{primitive_test} and the subsequent discussion we now define a decision rule for the hypothesis in  \eqref{relevant_difference}
 rejecting the null hypothesis in \eqref{relevant_difference}, whenever 
\begin{equation} \label{test_decision}
\hat W_N(\Delta) :=\frac{\sqrt{N} \big( \vvvert \hat S_N-S_0 \hat \Pi_{ k } \vvvert ^2 - \Delta  \big)}{\hat V_N}> q_{1-\alpha},
\end{equation}
where $q_{1-\alpha}$ is the $1-\alpha$ quantile of the distribution of $W$ in \eqref{hd5}. The next theorem shows the validity of this test decision. 

\begin{theo} \label{theorem_2}
Under the assumptions of Corollary~\ref{corollary_2} the  decision presented in \eqref{test_decision} yields an asymptotic  level-$\alpha$ and  consistent  test 
for the hypothesis \eqref{relevant_difference}.
\end{theo}

% We conclude this Section by a Remark on centeredness and dependence of the observations.

\begin{rem} \label{remark_centering} ~~
\smallskip

{\rm 
(1)  In the theoretical results presented so far it is  assumed that the regressors are centered, that is    $\E X_1 =0$. In reality it may well be that $\E X_1 = \mu \neq 0$ and therefore an empirical centering is necessary. More precisely we can introduce the sequential mean estimates (recall that $\FP \in \left[ a,1 \right]$, $0 < a < 1$) 
$$
        \bar Y_N[\FP] 
    :=  
        \frac{1}{ \lfloor N \FP \rfloor } \sum_{n=1}^{\lfloor N \FP \rfloor} Y_n, 
    \quad 
        \textnormal{and} 
    \quad 
        \bar X_N[\FP] 
    :=  
        \frac{1}{\lfloor N \FP \rfloor} \sum_{n=1}^{\lfloor N \FP \rfloor} X_n,
$$
and consider the modified observations $Y_n-\bar Y_N[\FP]$ and $ X_n- \bar X_N[\FP] $ in any of the sequential statistics introduced at the beginning of this section. It can be shown that all results
presented so far remain correct in this  case (we also employ this empirical centering in the simulation study in Section~\ref{sec6}). 
\smallskip

(2)  It follows from the proof of Theorem~\ref{theorem_2} that 
$
\lim_{N \to \infty }
\mathbb{P}\big( \hat W_N(\Delta)  > q_{1-\alpha}\big) 
= 0
$
if $ \vvvert S-S_0\vvvert ^2 < \Delta $ (interior of the null hypothesis), while for 
$ \vvvert S-S_0\vvvert ^2 = \Delta $ 
(boundary of the null hypothesis)
we have 
$\lim_{N \to \infty }
\mathbb{P}\big( \hat W_N(\Delta)  > q_{1-\alpha}\big) 
=
\alpha $. 
\smallskip

(3) It is easy to see that the test statistic $\hat W_N(\Delta)$ is a  decreasing function of the threshold $\Delta$. This means that rejection for some $\Delta> 0$ also entails rejection for all smaller thresholds and vice versa accepting the hypothesis for some threshold means acceptance for all larger values. Hence the interpretation for multiple values of $\Delta$ - if considered - is internally consistent. 
\smallskip

(4) Similar results can also be obtained for other dependence concepts than 
$\phi-$mixing. 
For example,  consider  $\alpha$-mixing  processes \citep[for a definition, see, for instance][]{dehling} and assume
\begin{align*}
    (3'): &  ~~~~\text{The sequence ~} \{(X_n, \varepsilon_n)\}_{n \in \Z} ~\text{ is strictly stationary and~} \alpha \text{-mixing s.t. ~} \sum_{h \ge 1} h^{4/\kappa}\alpha(h)<\infty\\
(4'):&  ~~~~  \text{If } c \text{ is the smallest even integer } c>4+ \kappa , \text{ then } ~~
\E | \langle X_1, e_j \rangle|^c \le C (\E | \langle X_1, e_j \rangle|^2)^{c/2}   ~~~~.
\end{align*}
respectively, where $\alpha(h)$ denotes the $\alpha$-mixing coefficient.
Then all  statements in this and the subsequent sections remain correct if the conditions $(3)$ and $(4)$ in  Assumption~\ref{ass31} are replaced by $(3')$ and $(4')$, respectively.  For technical details we refer the interested reader to \cite{dehling} (covariance inequalities for $\alpha-$mixing in Hilbert spaces) and to \cite{merlevede} (invariance principles under $\alpha$-mixing).  
}
\end{rem}

\section{Statistical inference for relevant prediction errors} \label{sec4}

In the previous section we have compared the slope operator $S$ to a predetermined operator $S_0$, in terms of the Hilbert--Schmidt norm $\vvvert S-S_0 \vvvert^2$. However, from a statistical perspective other deviation measures are at least equally important.  One vital mode of comparison is, in how far the predictions  of the two operators differ,
which we discuss in this section. Prediction in finite and infinite dimensional linear models is a well-investigated subject. In the work most closely related to our own, \cite{crambes2013} considered the minimax prediction error of the spectral cut-off estimator $\hat S_N$, compared to the true slope $S$. The focus in our work is different, as we want to compare the predictive properties of the true slope $S$, with the hypothesized operator $S_0$.  More specifically we are interested in the quantity $\E \| S X-S_0 X \|^2$, where the expectation is taken with respect to a regressor $X$, distributed  as $X_1$. A simple calculation, using the trace representation of inner products and its properties (see Section~13.5 in \cite{HorKokBook12})  shows that
\begin{equation}\label{h6} 
    \E \| S X-S_0 X \|^2 
= 
    \vvvert S \Gamma^{1/2}-S_0  \Gamma^{1/2}\vvvert^2.
\end{equation} 
Therefore we are comparing smoothed versions of the slope operators.
%and, as we have noticed already in Section \ref{Subsection_relevant_hypothesis}, this is quite different from the measure $\vvvert S-S_0 \vvvert^2$. 
We point out that even though the inequality
$$
\vvvert S \Gamma^{1/2}-S_0  \Gamma^{1/2}\vvvert^2 \le  \vvvert \Gamma^{1/2} \vvvert_\mathcal{L}^2 \vvvert S-S_0 \vvvert^2
$$
implies that small differences between $S$ and $S_0$ result in small prediction errors, the converse is not true. In particular small prediction errors may be found in operators, that vastly differ in  the  Hilbert--Schmidt norm. 
%From a practical perspective the predictive mode of comparison is of interest, because it assesses the error made using the (simpler) operator $S_0$ instead of the true slope $S$. One example of $S_0$ might be a low dimensional reconstruction of $S$ with a fixed set of basis functions. 
%In the first part of this section we formulate the hypothesis of no relevant prediction error and discuss some general features of prediction in linear models. In the second part we derive analogues of  Theorems \ref{theorem_1} and \ref{theorem_2}. Finally in Remark \ref{remark_1}, we 
%provide a comparison of  our findings with a  related weak convergence result deduced  in Theorem 9, \cite{crambes2013}. 
%\subsection{Relevant prediction errors} \label{Subsection_relevant_prediction}

We now formulate the hypothesis of no relevant prediction error  as
\begin{equation} \label{relevant_prediction_error}
{}^{\textnormal{pred}}H_0^\Delta: \E\| SX-S_0X\| ^2 \le \Delta \quad \quad {}^{\textnormal{pred}}H_1^{\Delta}: \E \| SX-S_0X\| ^2 > \Delta,
\end{equation}
where $X$ has the same distribution as $X_1$.
%\change{(independent of the sequence $\{(X_n ,\varepsilon_n)\}_{n\in \Z}$)}. 
Again $\Delta>0$ is a user determined threshold, where a deviation of more than $\Delta$ is considered scientifically relevant. In order to test this hypothesis we recall the identity 
\eqref{h6}
%$\E \vvvert S X-S_0 X \vvvert^2 = \vvvert S \Gamma^{1/2}-S_0  \Gamma^{1/2}\vvvert^2$, from the beginning of this Section
 which suggests the natural estimator $\vvvert \hat S_N \hat \Gamma_N^{1/2}-S_0 \hat \Pi_{ k } \hat  \Gamma_N^{1/2}\vvvert^2$ for the prediction error. Recall that the projection $S_0 \hat \Pi_{ k }$ can be replaced by the  operator $S_0$, but projecting seems more sensible, because otherwise $S_0$ is compared to $S$ along axes, which are not estimated. Compared to the statistic discussed in Section \ref{sec3}, we expect that the multiplication with $\hat \Gamma_N^{1/2}$
leads to an even stronger smoothing effect, which indeed manifests itself in  weaker assumptions on the regularization parameter. 
 % We therefore formulate the altered condition \ref{ass41}) as follows
 
\begin{assumption}  \label{ass41} \textit{Rates of regularization:} 
% \textit{Rates of regularization:} 
 {\rm The regularization parameter $k$  satisfies  for some $\delta>0$
$$
\frac{k^{\SM+1+\delta}}{\sqrt{N}} \to 0 \quad \quad \textnormal{and} \quad \quad \frac{k^{\SM (\beta+1/2)}}{\sqrt{N}} \to \infty.
$$
} 
\end{assumption}

If Assumption~\ref{ass41} holds, the bias of the prediction error vanishes asymptotically, as
\begin{eqnarray*} 
	\vvvert [S-S\Pi_{ k }]\Gamma^{1/2}\vvvert  &= &\vvvert R \Gamma^\beta [I-\Pi_{ k }]\Gamma^{1/2}\vvvert \le \vvvert R\vvvert  \vvvert  \Gamma^\beta [I-\Pi_{ k }]\Gamma^{1/2}\vvvert _\mathcal{L} \\
	&= & \vvvert R\vvvert   \lambda_{k+1}^{\beta+1/2} =\mathcal{O}(k^{-\SM (\beta+1/2)})= o(1/\sqrt{N}).
\end{eqnarray*}
Notice that compared to Assumption \ref{ass31}(7),  Assumption \ref{ass41} translates into weaker smoothness requirements for  the operators $S$ and $S_0$. 
In fact it implies $\beta>1/2 +1/\SM$ (instead of $\beta>1+1/\SM$, because $S \in \mathcal{C}(\beta - 1/2, \Gamma)$ already entails $S \Gamma^{1/2} \in \mathcal{C}(\beta, \Gamma)$).  In applications this effect is reflected by smaller values of  $k$ in the spectral cut-off  estimator  for  prediction compared to reconstruction. 
Nevertheless the representation 
$$
    \hat S_N \hat \Gamma_N^{1/2} 
= 
    S \hat \Pi_{ k } \hat  \Gamma_N^{1/2}
    + 
    U_N (\hat \Gamma^\dagger_k)^{1/2},
$$
suggests, that inference for the prediction error remains a  genuinely inverse problem.
In particular we still observe an amplification of the observation error $U_N$  by the regularized inverse,  but to a weaker extend than in the case of reconstruction. 
%In the next Section we formulate a statistical test for the hypothesis of no relevant prediction error. 

%\subsection{A self-normalized test for relevant prediction errors} \label{Subsection_relevant_prediction_test}

Recalling  the definition of the sequential estimators \eqref{Eq_DefSeqCovOp}, \eqref{Eq_DefSeqPseudoInvs-and-SeqProj} and \eqref{App_ref_3}  in Section~\ref{Subsection_weak_convergence}
% , as well as the convention of omitting the sequential parameter $\FP$, if $\FP = 1$, such that e.g. $\hat \Gamma =  \hat \Gamma[1]$. We f
we obtain  the following invariance principle. 
%, for the subsequent self-normalization.

\begin{theo} \label{theorem_3}
Under the Assumptions~\ref{ass31}(1)-(6) and
Assumption~\ref{ass41}, it holds for any compact interval $I \subset (0,1]$, that
$$
    \left\{ \sqrt{N} \FP 
    \left( 
        \vvvert 
        \hat S_N[\FP] \hat \Gamma_N [\FP]^{1/2} 
        - 
        S_0 \hat \Pi_{ k }[\FP] \hat \Gamma_N [\FP]^{ 1 / 2 } 
        \vvvert ^2 
        - 
        \vvvert ( S - S_0 )\Pi_{k} \Gamma^{1/2} \vvvert ^2
    \right)
    \right\}_{\FP \in I} 
\stackrel{ d }{ \to } 
    \{ \tau^{\rm pred} \mathbb{B} (\FP) \}_{\FP \in I}, 
$$
where the long-run variance $(\tau^{\rm pred})^2 $ 
is defined as follows
\begin{align} \label{def_long_run_variance_pred}
    (\tau^{\rm pred})^2 
:=&
    4 \Big\{ \sum_{h \in \mathbb{Z}} \mathbb{E} \Big[ \langle (R-R_0) \tilde L[X_0 \otimes X_0 -\Gamma], R-R_0 \rangle \langle (R-R_0) \tilde L[X_h \otimes X_h -\Gamma], R-R_0 \rangle \Big]
\\
& 
    \quad \, +  2 \mathbb{E} \Big[\langle (R-R_0) \tilde L[X_0 \otimes X_0 -\Gamma], R-R_0 \rangle \langle \varepsilon_h \otimes X_h \Gamma^{\beta-1}, R-R_0 \rangle \Big] \nonumber
\\
& \quad \, \, \, \, +  \mathbb{E} \Big[ \langle \varepsilon_0 \otimes X_0 \Gamma^{\beta-1}, R-R_0 \rangle\langle \varepsilon_h \otimes X_h \Gamma^{\beta-1}, R-R_0 \rangle \Big] \Big\}. \nonumber
\end{align}
Here the map $ \tilde L$ is given in Definition \ref{Def_LinerazationOperat}.
\end{theo}

\begin{comment}
\begin{align} \label{long_run_prediction}
        \tau^{\rm pred}
=&
    4 \big\{ \langle \tilde L(R-R_0) \Psi_1 \tilde L (R-R_0)^{ \ast }, (R-R_0) \otimes (R-R_0)^{ \ast } \rangle
\\
& 
    +  2 
    \langle \tilde L(R-R_0) \Psi_2 \Gamma^{\beta}, 
        (R-R_0) \otimes (R-R_0)^{ \ast } 
    \rangle \nonumber
\\
& + \langle \Gamma^{\beta} \Psi_3 \Gamma^{\beta}, (R-R_0)^{ \ast } \otimes (R-R_0) \rangle \big\}. \nonumber
\end{align}
where $\mathbb{B}$ is a standard Brownian motion, the map $\tilde L$ is given in Definition \ref{Def_LinerazationOperat} and the operators $\Psi_1, \Psi_2, \Psi_3$ are defined in Proposition \ref{proposition_1}.
\end{comment}

Next we define the adapted denominator 
\begin{equation} \label{denominator_2}
    \hat V_N^{\rm pred} 
:= 
    \Big\{ \int_I \FP^4 \Big(  \vvvert \hat S_N[\FP]\hat \Gamma_N [\FP]^{1/2}-S_0 \hat \Pi_{ k }[\FP]\hat \Gamma_N [\FP]^{1/2} \vvvert ^2 - \vvvert \hat S_N\hat \Gamma_N ^{1/2}-S_0 \hat \Pi_{ k } \hat \Gamma_N ^{1/2}\vvvert ^2 \Big)^2 d\nu(\FP)  \Big\}^{1/2},
\end{equation}
and propose  to reject the null hypothesis in \eqref{relevant_prediction_error}, if
\begin{equation} \label{test_decision_2}
\hat W_N^{\rm pred}(\Delta) :=\frac{\sqrt{N} \big( \vvvert \hat S_N \hat \Gamma_N ^{1/2}-S_0 \hat \Pi_{ k } \hat \Gamma_N ^{1/2} \vvvert ^2 - \Delta  \big)}{\hat V_N^{\rm pred}}> q_{1-\alpha}.
\end{equation}
%As in Section \ref{Subsection_test_one_sample}, the test statistic $\hat W^{\rm pred}(\Delta)$ converges weakly to the pivot $W$ (see Corollary \ref{corollary_2}). 
%The validty of this approach is demonstrated by the following theorem. 

\begin{theo} \label{theorem_4}
Suppose that the Assumptions~\ref{ass31}(1)-(6), Assumption~\ref{ass41} hold and that the long-run variance $\tau^{\rm pred}$ is positive. Then the decision rule in \eqref{test_decision_2} defines  a consistent, asymptotic level-$\alpha$ test for the hypothesis in \eqref{relevant_prediction_error}.
\end{theo}

%Consistency and level-$\alpha$ here are to be understood as in Theorem \ref{theorem_2}. 
We conclude this part by comparing  the weak convergence result of this Section  to those derived in \cite{crambes2013} for prediction. 

\begin{rem} \label{remark_1}
{\rm  %\TK{\bf gibt es noch andere Arbeiten, mit denen man vergleichen kann z.B. \cite{cardot2007}?} 
 \cite{crambes2013}  proved  a weak convergence result in the case of i.i.d.\ observations and somewhat different assumptions than those  used  in this section. In their Theorem~9 (which is a generalization of Theorem~4.2 in \cite{cardot2007}), they showed that for a
 random variable  $X$ distributed as $X_1$ and independent of the sequence $\{(X_n ,\varepsilon_n)\}_{n\in \Z}$ , the weak convergence
\begin{equation} \label{mas_convergence}
    \sqrt{N/k}(\hat S_N X -SX) 
\stackrel{ d }{ \to } 
    G
\end{equation}
holds, where $G$ is a centered Gaussian process on $H_2$, with covariance operator $\E \varepsilon_1 \otimes \varepsilon_1$. 
Notice the standardization of $\sqrt{N/k}$ instead of $\sqrt{N}$, which corresponds to the standard deviation of $U_N  \Gamma^\dagger_k X $. 
This term naturally occurs (as second term) in the decomposition 
$$
    \sqrt{ N / k } (\hat S_N - S\Pi_{ k } ) X 
= 
    \sqrt{ N / k } \left\{ S ( \hat \Pi_{ k } -\Pi_{ k } ) X 
    + 
    U_N \hat \Gamma^\dagger_k X \right\}.
$$
Importantly the first term here is asymptotically negligible, which is not the case in our smoothed statistic. Indeed, in the $L^2$-statistic, after the smoothness shift is performed,  the amplifying effect of  the regularized inverse is eliminated, which yields the convergence rate $1/\sqrt{N}$ for both terms instead of $\sqrt{k/N}$. In view of these technical differences we have  developed a separate asymptotic theory for the proof of Theorem~\ref{theorem_4} tailored to the study of relevant hypotheses
and could not use the result in \eqref{mas_convergence}. 
\begin{comment}
It is also worthwhile to mention that the rate of convergence in \eqref{mas_convergence} is usually slow. To see this recall the decomposition
$$
\sqrt{N/k}(\hat S_N - S\Pi_{ k } )X = \sqrt{N/k} \big[ S [\hat \Pi_{ k } -\Pi_{ k }] X + U_N \hat \Gamma^\dagger_k X \big].
$$
Asymptotically, the second term dominates, such that
$$
\sqrt{N/k} U_N \hat \Gamma^\dagger_k X  \approx \sqrt{N/k} U_N  \Gamma^\dagger_k X = \frac{1}{\sqrt{N}}\sum_{n=1}^N \varepsilon_n \frac{\langle X_n, \Gamma^\dagger_k X \rangle }{\sqrt{k}} \to G.
$$
However, the first term $\sqrt{N} S [\hat \Pi_{ k } -\Pi_{ k }] X $ is also approximately normal, which can be understood in terms of the linearizations presented in this paper (see Lemma \ref{Lem_Projection-Statistics-Smoothed-Right-and-Left} and its proof). Given that in applications
$k$ is usually quite small,  $\sqrt{N / k } S [\hat \Pi_{ k } -\Pi_{ k }] X $ tends to blow up the variance of $\sqrt{N/k}(\hat S_N X -SX)$ for finite samples. 
 \TK{{\bf das ist irgendwie unklar. wuerde ich  weglassen}
 In particular the asymptotic normality presented in \cite{crambes2013} only holds as $k \to \infty$ and not for  fixed $k$ (as in this paper).
In this sense the linearizations  presented in this paper can be understood  as refinement to the existing   theory. }
\end{comment}
}
\end{rem}

\section{Change point analysis and two sample tests} \label{sec5}

In the context of dependent time series, functional data analysis  is usually employed to model  successive observations of a system over an extended time period. In this context it is natural to consider the stability of the data, e.g., by searching for change points in the mean  \citep[see e.g.][]{BerGabHorKok2009}, Chapter~6 in \cite{HorKokBook12}, \cite{aston2012} or \cite{DetKokVol20}) or in the second order structure, i.e.\  covariance operators \citep{jaruskova2013}, principle components \citep{dettekutta2020} or other features 
 \citep{aue2018}. 
For the linear regression model~\eqref{model_1} stability concerns first and foremost the slope operator $S$.
This problem has been addressed by \cite{HorHusKok2010} for AR(1)-processes and by \cite{horvathlin2011} for more general processes by  testing ``classical'' hypotheses (of the type $H_0$ versus $H_1$ described at the beginning of Section \ref{Subsection_relevant_hypothesis}).
  In this Section we discuss how one can adapt the previous techniques to the detection of a relevant change  in  the operator $S$. The related, but easier case of comparing two operators, say $S^{(1)}$ and $S^{(2)}$ from  independent samples is briefly discussed in Remark~\ref{rem2} below.

% \subsection{Notations and hypothesis}

%To be precise, we are interested in  a (potential) relevant change in the slope parameter $S$ in model \eqref{model_1}.
%  Change point analysis for the slope estimator in functional models has been conducted by \cite{HorHusKok2010} for AR(1)-processes and by \cite{horvathlin2011} in more general frameworks. 
%Our  approach is more general than the methods in the above cited works, since we do not need the assumption of i.i.d.\ errors.
%
%In applications where the stability of the regressors is doubtful, a prior change point detection for the covariance structure of the regressors should be performed and the analysis be split. For change point detection in the covariance operator see \cite{aston2012}, \cite{stoehr2020} and for the detection of relevant changes \cite{DetKokVol20}.\\

To be precise consider the following regression model  
\begin{equation} \label{model_3}
Y_n \otimes X_n = S_n X_n \otimes X_n + \varepsilon_n \otimes X_n \quad \quad n=1,...,N,
\end{equation}
where $\SL := S_1 =S_2 =...=S_{\lfloor \theta N \rfloor}$, $\SLL:= S_{\lfloor \theta N \rfloor+1}  =...=S_N$ and $\theta \in (0,1)$ determines the location of the change point  and is unknown.  We assume that $\{(X_n, \varepsilon_n)\}_{n \in \Z}$  is  a stationary time series of regressors and errors, which satisfies the Assumptions~\ref{ass31}(1)-(5) in Section~\ref{Subsection_assumptions}.
The two hypotheses of no relevant change at $\theta$ in the slope operator and of no relevant change in the predictive properties of $S$ are defined by
\begin{align}
 \label{relevant_change}
H_0^\Delta& : \vvvert \SL-\SLL\vvvert ^2 \le \Delta \quad \quad & H_1^\Delta :  \vvvert \SL-\SLL \vvvert ^2 > \Delta, \quad \quad \,\, \\
{}^{\textnormal{pred}}H_0^\Delta& : \E \| \SL X-\SLL X\| ^2 \le \Delta \quad \quad & {}^{\textnormal{pred}}H_1^\Delta: \E \| \SL X-\SLL X \| ^2 > \Delta.  \label{relevant_changepre}
\end{align}
% These hypotheses are the analogues to the one-sample versions presented in \eqref{relevant_difference} and \eqref{relevant_prediction_error} respectively. 
Before continuing we point out an important difference to  change point analysis based on testing classical hypotheses (that is $\Delta=0$): Suppose a change in the slope operator is detected by a traditional change point test, but would be considered irrelevant in the sense of the hypotheses \eqref{relevant_change} for some small $\Delta >0$. In this situation it might be reasonable to reconstruct the slope $S$ using all of the data, instead of considering two 
estimates from the data 
before and after the estimated change point. On the one hand  this would introduce a (small) bias in the estimation, but 
on the other hand this increase could be compensated by a  significant reduction of the variance.

In the following discussion let 
$\hat \theta $ denote  an estimator of the change 
point (see Remark~\ref{rem2}(1) below for a concrete example). We  define the sequential  estimators for the  covariance operator 
\begin{align}
    \label{hd6}
	\hat \Gamma_N^{(1)}[\FP]=  \frac{1}{N \hat \theta} \sum_{n=1}^{\lfloor  \FP \hat \theta N \rfloor} X_n \otimes X_n \quad   \quad\textnormal{and }  \quad \quad \hat \Gamma_N^{(2)}[\FP]=    \frac{1}{N (1- \hat \theta)} \sum_{n= \hat \theta N+1}^{\lfloor  \FP (1- \hat \theta) N \rfloor} X_n \otimes X_n~.
\end{align}
The  eigenvalues (in non-increasing order) and their corresponding eigenfunctions are
denoted by $\hat \lambda_1^{(j)}[\FP] \ge \hat \lambda_2^{(j)}[\FP] \ge \ldots$ and  $\hat e_1^{(j)}[\FP], \hat e_2^{(j)}[\FP], \ldots$,  respectively $(j=1,2)$. 
As before, we consider for $k \in \N$ the regularized  inverse of the operator $\hat \Gamma^{(j)}_N[\FP]$, as well as the projection on the first $k$ empirical eigenfunctions as
$$
	    \hat \Gamma_k^{\dagger, (j)}[\FP]
	= 
	    \sum_{i=1}^{k} \frac{1}{\hat{\lambda}_i^{(j)}[\FP]} \hat e_i^{(j)}[\FP] \otimes \hat e_i^{(j)}[\FP] 
    \qquad
        \textnormal{and }  
    \qquad
        \hat \Pi_{ k }^{(j)}[\FP]
    = 
        \sum_{i=1}^{k} \hat e_i^{(j)}[\FP] \otimes \hat e_i^{(j)}[\FP].
$$
By virtue of the regularized inverse operators, we can now define the slope estimates, as
\begin{align}   \label{hd7}
        \hat S_N^{(1)}[\FP]
    =  
        \frac{1}{N \hat \theta} 
        \sum_{n=1}^{\lfloor  \FP \hat \theta N \rfloor} 
        Y_n \otimes X_n \hat \Gamma^{\dagger, (1)}_k[\FP] 
    \qquad
        \textnormal{and }  
    \qquad
        \hat S^{(2)}_N[\FP]
    =    
        \frac{1}{N (1- \hat \theta)} 
        \sum_{n=N \hat \theta+1}^{\lfloor  \FP (1- \hat \theta) N \rfloor} 
        Y_n \otimes X_n \hat \Gamma^{\dagger, (2)}_k[\FP]
\end{align}
and propose to 
reject the null hypothesis in \eqref{relevant_change}
whenever
\begin{equation} \label{test_decision_3}
        \hat W^{\rm cp}_N(\Delta)
    :=
        \frac{\sqrt{N} 
        \left( 
            \vvvert \hat S_N^{(1)}- \hat S_N^{(2)} \vvvert ^2 - \Delta  
        \right)
        }{
        \hat V_N^{\rm cp} }
    > 
        q_{1-\alpha}.
\end{equation}
where the denominator $\hat V_N^{\rm cp}$ is defined as
\begin{equation} \label{denominator_3}
        \hat V_N^{\rm cp} 
    := 
        \left\{ \int_I \FP^4 
        \left(  
        \vvvert  \hat S_N^{(1)}[\FP]- \hat S_N^{(2)}[\FP] \vvvert ^2 
        - 
        \vvvert  \hat S_N^{(1)}- \hat S_N^{(2)}\vvvert ^2 \right)^2 d\nu(\FP)  
        \right\}^{1/2}, 
\end{equation}
and $q_{1-\alpha}$ is the $(1-\alpha)$-quantile of the distribution of the random variable $W$  defined in  \eqref{hd5}. In order to test for relevant  predictive differences we define $\hat W_N^{\rm cp, pred}(\Delta)$ in the same way as $\hat W_N^{\rm cp}(\Delta)$, where we replace all instances of $S^{(j)}[\FP]$ by $S^{(j)}[\FP] \hat \Gamma_N^{(j)}[\FP]^{1/2}$. This gives us the test decision for a relevant change in prediction 
\begin{equation} \label{test_decision_4}
    \hat W_N^{\rm cp, pred}(\Delta)
> 
    q_{1-\alpha}.
\end{equation}

For the statement of the main results of this section we require the  consistency of the change point estimator $\hat \theta$, such that 
the amount of missclassified data is asymptotically negligible.
\begin{assumption}  \label{ass51}
 {\rm (Consistency of $\hat \theta$): }
    $ \quad
\hat \theta = \theta +o_\PP(1/\sqrt{N}).
$
\end{assumption}

%\subsection{A test for relevant changes in the slope operator}

 The following result shows 
 that the decision rules~\eqref{test_decision_3} and
 \eqref{test_decision_4} define 
 consistent tests for the hypotheses~\eqref{relevant_change} and \eqref{relevant_changepre}, respectively and have asymptotic level $\alpha$. In its formulation we understand that a postulated assumption applies  to each operator before and after the change point.

\begin{theo} \label{theorem_5}
Suppose that the Assumptions~\ref{ass31}(1)-(6) and Assumption~\ref{ass51} hold. 
\begin{itemize}
\item[a)] If additionally  Assumption~\ref{ass31}(7) holds, then the long-run variance $(\tau^{\rm cp})^2$ of 
$$ 
     \sqrt{N} \left( 
     \vvvert  \hat S_N^{(1)}- \hat S_N^{(2)} \vvvert ^2 
     - 
     \vvvert   S^{(1)} -  S^{(2)}\vvvert ^2 
     \right) 
\stackrel{ d }{ \to } 
    \mathcal{N}(0, (\tau^{\rm cp})^2)
$$
exists. If $\tau^{\rm cp}$ is positive, then the decision rule   in \eqref{denominator_3} yields a consistent, asymptotic level-$\alpha$ test for the hypothesis~\eqref{relevant_change} of no relevant change in the slope. 
\item[b)] If additionally   Assumption~\ref{ass41} holds, then the long-run variance $(\tau ^{\rm cp, pred})^2$ of 
$$
     \sqrt{N} 
     \left(
        \vvvert   \hat S_N^{(1)} 
            ( \hat \Gamma_N^{(1)} )^{1/2} 
            -  \hat S_N^{(2)}
            ( \hat \Gamma_N^{(2)} )^{1/2}
        \vvvert ^2 
        - 
        \vvvert  
        ( S^{(1)}-  S^{(2)} ) \Gamma^{1/2} 
        \vvvert ^2 
    \right) 
\stackrel{ d }{ \to } 
    \mathcal{N}(0, (\tau ^{\rm cp, pred})^2)
$$ 
exists. If $\tau ^{\rm cp, pred}$ is positive, then the decision rule  in \eqref{test_decision_4} yields a consistent, asymptotic level-$\alpha$ test for the hypothesis \eqref{relevant_changepre} of no relevant change in the prediction.
\end{itemize}
\end{theo}

It is possible to give explicit formulas for $\tau ^{\rm cp}$ and $\tau ^{\rm cp, pred}$, which are similar to those in Proposition~\ref{proposition_1} and  Theorem~\ref{theorem_3} above, but we omit them to avoid redundancy. 
We conclude this section with a brief remark  concerning the change point estimator $\hat \theta$ and two sample testing.

\begin{rem} \label{rem2} ~~ 

{\rm 

(1) There are many ways of defining an estimator
for  the location $\theta$ of the change point.
As an example  we consider a standard change point estimator from  the observations $Y_1 \otimes X_1 
\ldots , Y_n \otimes X_n$
 based on the  CUSUM-principle 
(note that any change in the slope operator $S_n$ in model~\eqref{model_3} manifests itself in the product $S_n \Gamma$). 
To be precise we define 
\begin{align}   \label{hd8}
        \hat \theta 
    := 
        \frac{1}{N}
    \operatorname{argmax} \left\{ f(M): 1< M < N  \right\}  , 
\end{align}
where 
the function $f$ is given by  
$$
        f(M) 
    := 
        \frac{M}{N} 
        \left( 1 -\frac{M}{N} \right) 
        \left\vvvert 
            \frac{1}{M} \sum_{n=1}^M 
            Y_n \otimes X_n 
            - \frac{1}{N-M} 
            \sum_{n=M+1}^N Y_n \otimes X_n 
        \right\vvvert^2, 
\quad
        M = 2, \ldots ,N-1 
$$ 
It then follows from  Corollary~1 in \cite{hariz2007}
that 
$$
        |\hat \theta-\theta|
    =
        \mathcal{O}_\PP(1/N)
$$
if the condition
$$
        \sum_{q,r} \sqrt{\E \langle S^{(j)} X \otimes X, f_r \otimes  e_q \rangle^2 }
    < 
        \infty
$$
        holds for some orthonormal basis $\{f_r\}_{r \ge 1}$ of $H_2$. 
        %This condition can be understood as a requirement of decaying tails for higher moments. 
       In  this  case  Assumption~\ref{ass51}  
        is satisfied for the estimator in \eqref{hd8}.
        \smallskip 

  (2) It is easy to see that the test formulated in this section can be applied (without the change point estimation) to the case of two independent samples of size $N_1$ and $N_2$. In this case we set $N=N_1+N_2$  and replace $\hat \theta  N $ and $(1-\hat \theta ) N $ in the definitions~\eqref{hd6} and \eqref{hd7} by $N_1$ and $N_2$. The details are omitted for the sake of brevity.
Tests for relevant differences between the slopes of two functional linear regression models  
 may be of interest e.g.\ in cases where the behavior of contemporary individuals at different geographical locations is compared.
    }
\end{rem}

\section{Finite sample properties} \label{sec6}

In this section we investigate the finite sample properties of the spectral reconstructions and the self-normalized tests by means of a simulation study. We restrict ourselves to the one-sample cases presented in Sections~\ref{sec3} and \ref{sec4}, even though non-reported simulations suggest similar performance for two sample cases and  change point scenarios.  Following \cite{benatia2017} we consider  $(H_1, \mu_1) = (H_2, \mu_2) =(L^2([0,1], \mu)$, where  $\mu$ is the uniform distribution on the points $\{0, 1/50, 2/50..., 1\}$,
which may be regarded as a discretized version of the Lebesgue measure. 
We consider two scenarios of dependence:  i.i.d.\ observations and dependent observations, where regressors and errors are generated by $AR(1)$ processes. 

\subsection{Inference for the location of $S$}
\label{sec61}

Recall the regression model in \eqref{model_1}.
% as
% $$
%Y_n = S X_n + \varepsilon_n \quad \quad n=1,...,N.
%$$
In the case of   i.i.d.\ observations, we generate the regressors as 
\begin{equation} \label{regressors}
X_n(t) = \frac{\Gamma(A_n + B_n)}{\Gamma(A_n)\Gamma(B_n)} t^{A_n}(1-t)^{B_n} + Z_n,
\end{equation}
(shifted $\beta$-densities) where $A_n, B_n$ are independent, uniformly distributed on the interval $[2,5]$ and $Z_n$ is an independent, standard normal shift. Notice that the regressor functions $X_1, X_2, \ldots$ are 
%-as in most real scenarios - 
not centred, and hence we include an empirical centering in all statistics (see Remark~\ref{remark_centering}). The error functions $\varepsilon_n$ are i.i.d.\ realizations of an Ornstein--Uhlenbeck process, with zero mean, variance parameter and mean
reversion rate equal to one. Note that $\varepsilon_n$ is a stationary, centered, Gaussian process, which is the solution of the stochastic differential equation $d \varepsilon(t) = -\varepsilon(t) dt +  \sigma d \mathbb{B}(t) $, where $ \mathbb{B}$ is a standard Brownian motion and $\sigma=1$. 
Some typical paths of these regressors and errors are depicted in Figure~\ref{fig:sub21}.

 \begin{figure}[H]
\begin{subfigure}{0.5\textwidth}
\centering
\includegraphics[width=1\linewidth,height=220pt]{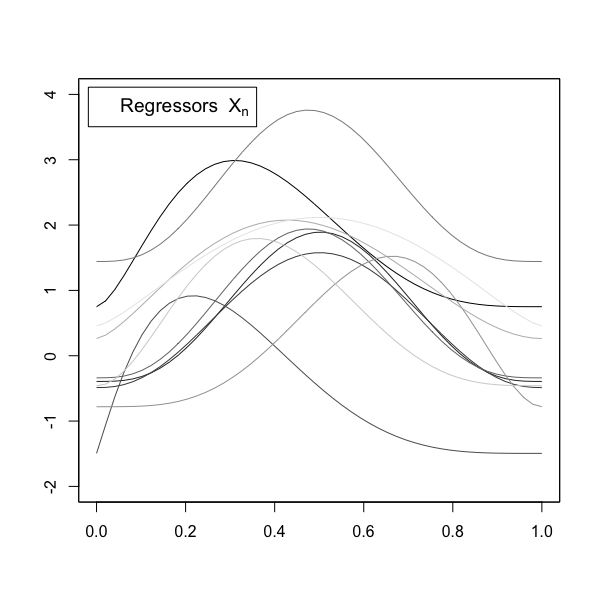}\\[-1ex]
 \end{subfigure}%
 \begin{subfigure}{0.5\textwidth}
 \centering
  \includegraphics[width=1\linewidth,height=220pt]{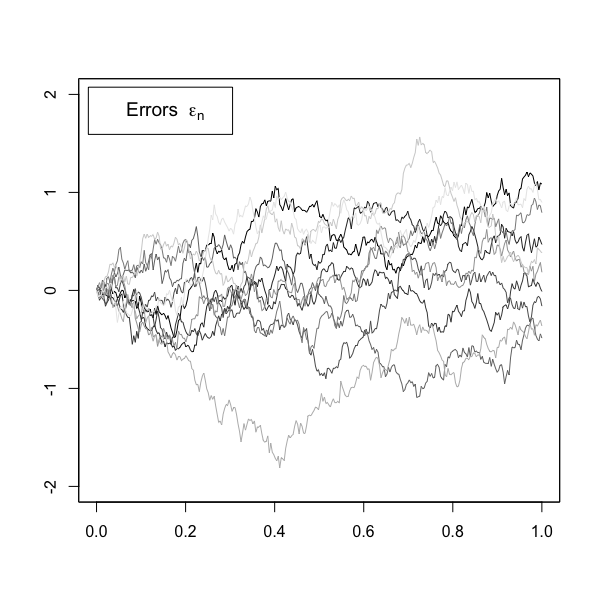}\\[-2ex]
 \end{subfigure}
  \caption{\textit{Realizations of regressors (left) and errors (right), both in the i.i.d.\ case.} \label{fig:sub21}}
 \end{figure}

In the case of dependent observations we generate  both regressors and errors by $AR(1)$ processes, with parameter $\rho =0.6$, that is
%. More precisely we define the dependent regressors $\tilde X_n$ and errors $\tilde \varepsilon_n$ as 
$$
\tilde X_n = \rho \tilde X_{n-1} +  X_{n}, \quad \quad \tilde \varepsilon_n = \rho \tilde \varepsilon_{n-1} + \varepsilon_{n},
$$
where the random variables $X_n$ and $ \varepsilon_n$ are i.i.d.\ random functions, generated as in the independent case  (see equation \eqref{regressors} and following discussion). 
In all  simulations a burn in period of $200$ repetitions is used. 
%By construction the errors and regressors are mutually independent, and thus in particular our weak exogeneity condition $4)$ holds. 
Finally we turn to the operators $S_0$ and $S$, both of which are integral operators,  defined as
$$
S f \mapsto \int_0^1 \varphi_S(s,\cdot ) f(s) \mu(s), \quad \quad \textnormal{and} \quad \quad S_0 f \mapsto \int_0^1 \varphi_{S_0}(s,\cdot ) f(s) d\mu(s),
$$
pointwise for a function $f \in (L^2[0,1], \mu)$, where $\varphi_{S}$ and $\varphi_{S_0}$ are the integral kernels corresponding to $S$ and $S_0$. We first define the benchmark kernel $\varphi_{S_0}$ as in \cite{benatia2017}, by
$$
\varphi_{S_0}(s,t) = 1 -|s-t|^2 
$$
and then the slightly more complex regression kernel  $\varphi_S$ as
$$
\varphi_S(s,t) = 1 -4/5 |s-t|^2 + 1/5 \cos( |s-t|/5).
$$
In Figure \ref{fig2} we plot the two kernel functions, to illustrate their shape differences. The difference between the kernels can be asessed by the relative deviation measure
$$
1 - \frac{\vvvert S- S_0 \vvvert^2}{\vvvert S_0 \vvvert^2} = 1 - \frac{\int_0^1 \big[ \varphi_{S}(s,t) -\varphi_{S_0}(s,t) \big]^2 ds \, dt }{\int_0^1 \varphi_{S_0}(s,t)^2 ds \, dt }  \approx 0.032 
$$
(since $\vvvert S \vvvert^2 \approx  \vvvert S_0 \vvvert^2$ it does not matter by which norm we normalize). 

 \begin{figure}[H]  
\begin{subfigure}{0.5\textwidth}
\centering
\includegraphics[width=1\linewidth,height=200pt]{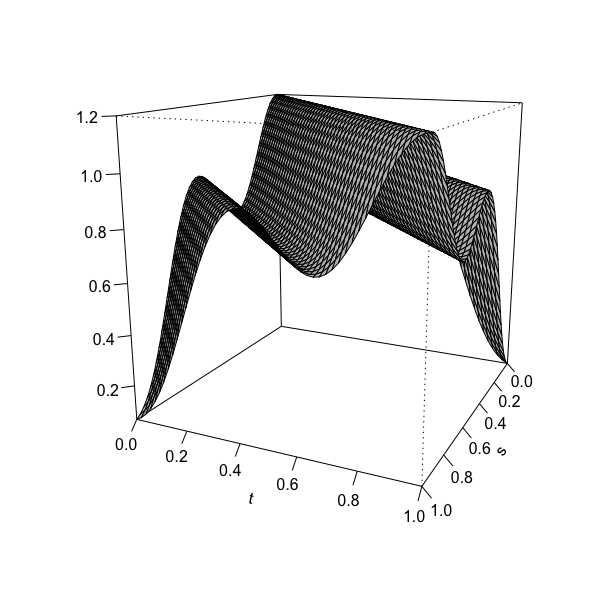}\\[-1ex]
 \end{subfigure}%
 \begin{subfigure}{0.5\textwidth}
 \centering
  \includegraphics[width=1\linewidth,height=200pt]{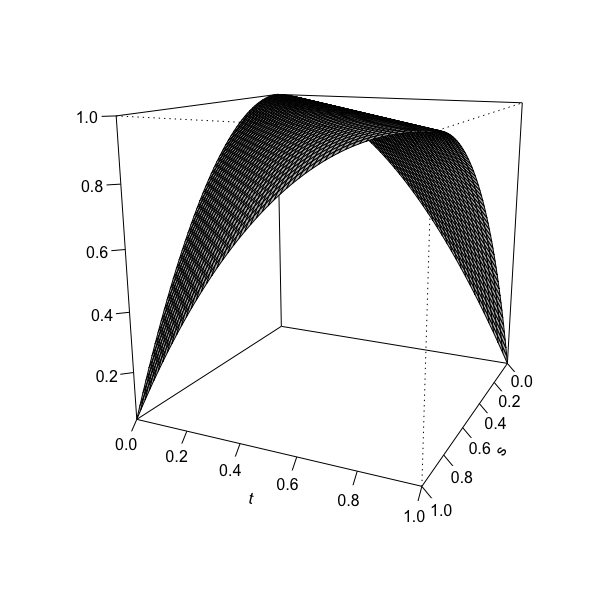}\\[-2ex]
 \end{subfigure}
  \caption{\label{fig2} \textit{ Image of the two integration kernels, plotted as surfaces. Left: $\varphi_S$. Right: $\varphi_{S_0}$. }}
 \end{figure}

Before we consider the estimation problem, it is reasonable to investigate the complexity of the two slopes $S$ and $S_0$, relative to the principal components of the operator 
$\Gamma$. For this purpose we consider the measure of relative explanation
\begin{equation}  \label{h7}
    \frac{
    \vvvert S \Pi_{ k }- S_0 \Pi_{ k }\vvvert^2 
    }{
    \vvvert S- S_0 \vvvert^2
    },
\end{equation}
 which varies in the interval $[0,1]$ and is increasing in $k$. A value of $1$ means that $S-S_0$ acts exclusively on $\operatorname{span}\{e_1, \ldots ,e_k\}$, whereas a value of $0$ implies that $S-S_0$ lives on the orthogonal complement.  A rapid increase in $k$ corresponds to a high degree of smoothness in the sense of Assumption~\ref{ass31}(1) and hence to a small bias.
However, smoothness of the slopes is not enough, since one also has to be able to estimate the principal components of $\Gamma$ properly. This corresponds to eigenvalues $\lambda_1, \lambda_2, \ldots$ of $\Gamma$ (and eigengaps), which are not too small. 

\begin{figure}[H]  
\begin{subfigure}{0.5\textwidth}
\centering
\includegraphics[width=0.9\linewidth,height=160pt]{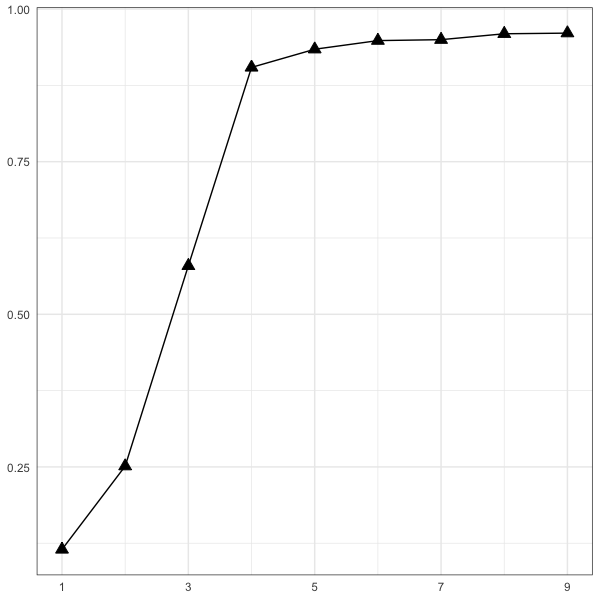}\\[-1ex]
 \end{subfigure}%
 \begin{subfigure}{0.5\textwidth}
 \centering
  \includegraphics[width=0.9\linewidth,height=160pt]{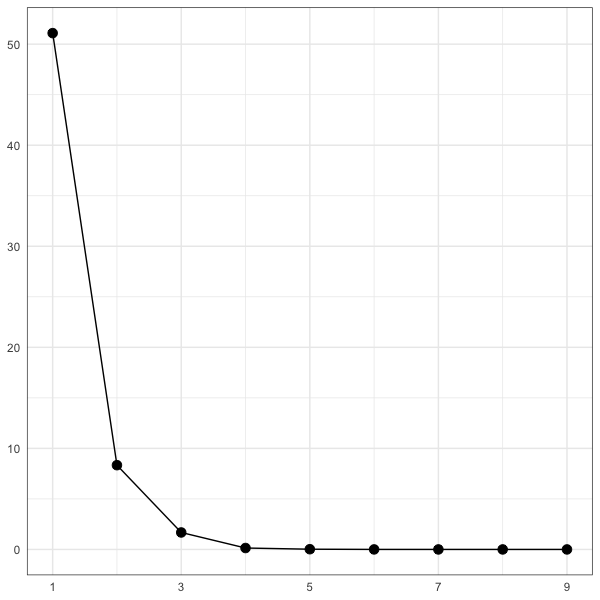}\\[-2ex]
 \end{subfigure}
  \caption{\label{fig3} \textit{Left: The measure \eqref{h7} as a function of $k$. Right: First $9$ eigenvalues of the operator $\Gamma$.  }}
 \end{figure}

In left part of Figure \ref{fig3} we display  the measure \eqref{h7}  as a function of $k$ and  observe  a 
quick increase for $k \le 5$ up to $95\%$. In the right part of the figure we present 
the decaying eigenvalues of the operator  $\Gamma$ (in the case of independent variables).  Here we observe  a fast decay followed by a sharp elbow.  
The first three eigengaps are rather large, but afterwards the distinction becomes increasingly difficult. Indeed even for a large sample size of $N=1000$, the recovery of more than $5$ eigenfunctions is somewhat unstable, resulting in inflated rejection probabilities at the boundary  of the hypothesis. Together these considerations suggest that choices of $k$ between $4$ and $5$ are optimal, depending on the sample size $N$.
\medskip

Throughout this section  all empirical results are based on  $1000$ simulation runs.
In order to investigate the power   of the  test \eqref{test_decision}  for the relevant  hypotheses \eqref{relevant_difference} we consider the  sample sizes of $N= 50, 200, 500$  and $k=3,4,5$ (note that $k=3$ is rather small for practical inference and only included to illustrate aspects of the bias-variance trade-off). The measure $\nu$ in the definition of the normalizer~\eqref{denominator} is  the uniform distribution  on the set  $\{1/5, 2/5, 3/5, 4/5\}$.  Simulations for other measures, which are not reported for the sake of brevity, 
 suggest that the number of points does not have a large or systematic impact  on the results.  
 % in any significant or systematic manner. 
% As mentioned in the beginning of this Section we consider the scenarios of i.i.d.\ regressors and errors, as well as $AR(1)$ processes. 
In 
Figure~\ref{fig4}   we display the 
 rejection probability of the self-normalized test \eqref{test_decision} as a function
of the threshold  $\Delta$ in the  hypothesis~\eqref{relevant_difference}.  A vertical grey line indicates the true value of $\vvvert S- S_0\vvvert^2 \approx 0.023$ and 
corresponds to the boundary of the hypotheses, while the  grey horizontal line shows the nominal level $\alpha$, which is chosen as  $\alpha = 0.05$. The left column shows the results for the i.i.d.\  case, while the results  for the dependent case can be found in the right column. 
The plots can be evaluated as follows: 

\begin{figure}[H]  \ContinuedFloat
\begin{subfigure}{0.5\textwidth}
\centering
\includegraphics[width=0.85\linewidth,height=180pt]{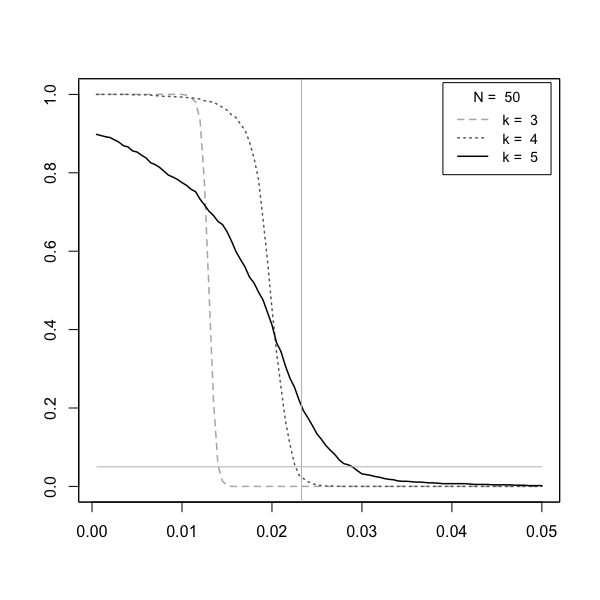}\\[-1ex]
 \end{subfigure}%
 \begin{subfigure}{0.5\textwidth}
 \centering
  \includegraphics[width=0.85\linewidth,height=180pt]{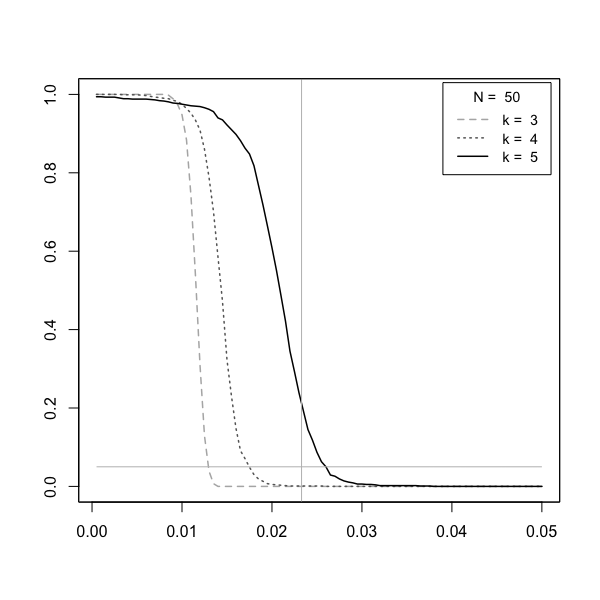}
 \end{subfigure}
  \end{figure}
 $ $\\[-15ex]
 \begin{comment}
  \begin{figure}[H]  
 \begin{subfigure}{0.5\textwidth}
\centering
\includegraphics[width=0.85\linewidth,height=180pt]{dep0N100.png}\\[-1ex]
 \end{subfigure}%
 \begin{subfigure}{0.5\textwidth}
 \centering
  \includegraphics[width=0.85\linewidth,height=180pt]{dep1N100.png}\\[-2ex]
 \end{subfigure}
   \end{figure}
\end{comment} 
 %
  \begin{figure}[H]  \ContinuedFloat
 \begin{subfigure}{0.5\textwidth}
\centering
\includegraphics[width=0.85\linewidth,height=180pt]{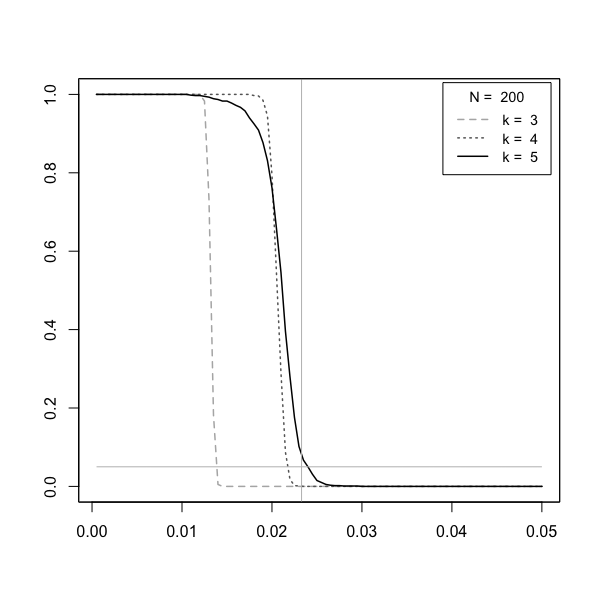}\\[-1ex]
 \end{subfigure}%
 \begin{subfigure}{0.5\textwidth}
 \centering
  \includegraphics[width=0.85\linewidth,height=180pt]{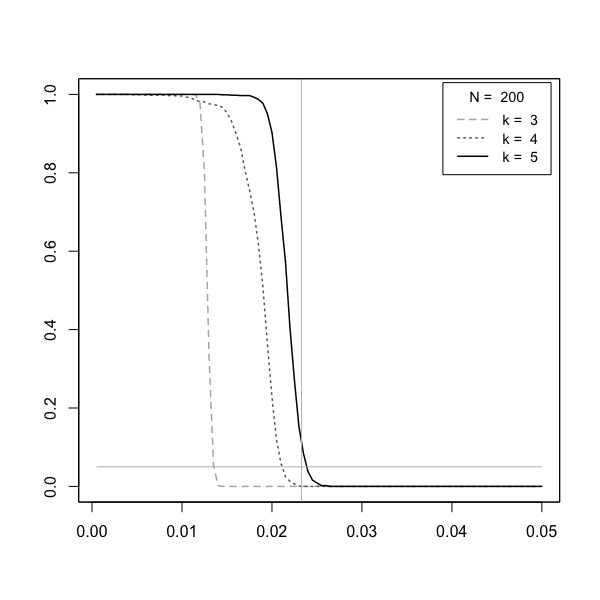}\\[-2ex]
 \end{subfigure}
  \end{figure}
  $ $\\[-15ex]
  \begin{figure}[H]  
 \begin{subfigure}{0.5\textwidth}
\centering
\includegraphics[width=0.85\linewidth,height=180pt]{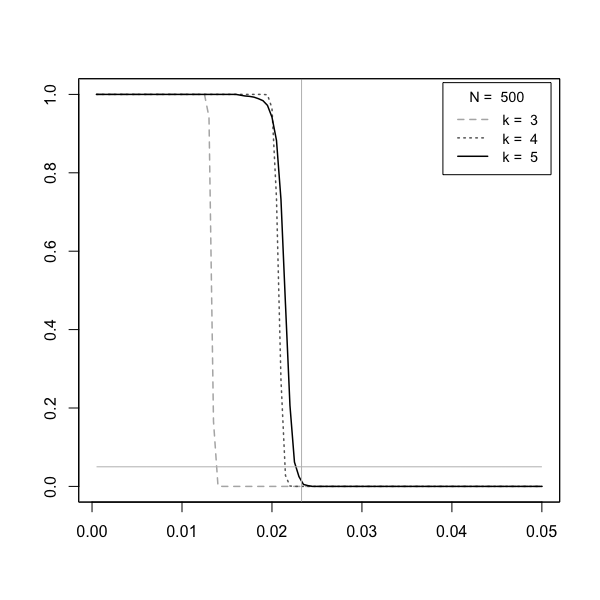}\\[-1ex]
 \end{subfigure}%
 \begin{subfigure}{0.5\textwidth}
 \centering
  \includegraphics[width=0.85\linewidth,height=180pt]{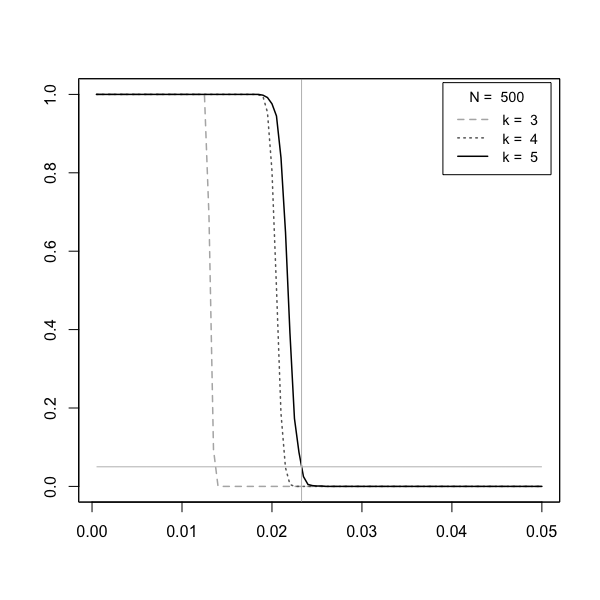}\\[-2ex]
 \end{subfigure}
   \caption{\label{fig4}
   % \HD{die Nummerierungen fuer die Figures stimmen nicht  ! } 
     \textit{Rejection probabilities ($y$-axis) of the test \eqref{test_decision} for different choices of $\Delta$ ($x$-axis). Different sizes of $N$ and $k$ are displayed, for i.i.d.\ observations (left) and dependent observations (right). The level $\alpha=0.05$ is indicated by the horizontal grey line, and the true size of $\vvvert S- S_0\vvvert^2$ by the vertical line.}}
 \end{figure}
 
If $\Delta <\vvvert S- S_0\vvvert^2\approx 0.023 $ (left of the vertical line) we operate under the alternative (see \eqref{relevant_difference}) and expect high rejection probabilities. At the boundary  of the hypotheses corresponding to  the vertical line we expect to approximate $\alpha$, whereas for larger values of $\Delta$ the asymptotic theory developed in  Section~\ref{sec3} suggests 
that the rejection probability tends to $0$; see Remark~\ref{remark_centering}(2). Because the  test statistic is a monotone function of  $\Delta$, the rejection probability decreases monotonically in $\Delta$; see Remark~\ref{remark_centering}(3). 
 
Before we evaluate the specific performance of the  test \eqref{test_decision}, we illustrate in Figure~\ref{fig4} some general features of the linear inverse problem. Evidently the rejection curves exhibit the bias variance trade-off, which occurs at the level of $k$. For $k = 3$ we observe a substantial bias, which we would expect, in view of Figure~\ref{fig3} (left). It diminishes rapidly for $k=4$ and $k=5$. In view of our discussion of \eqref{decomposition} we should understand the left shifts of the rejection curves as a result of the bias, which makes the test somewhat conservative. The upside of smaller values of $k$ is the accompanying small variance, manifest in the rapid decay of the rejection curves. For larger values of  $k$ the variance increases and this may result in inflated rejection probabilities at the boundary (e.g. for $N=50$ and the too large choice $k=5$). 

With regard to the statistical properties of the test presented in \eqref{test_decision}, we observe a reasonable level approximation at  the boundary, even for sample sizes as small as $50$ in the independent scenario. Dependence in the observations leads to worse performance, particularly for samples as small as $N=50$, whereas the effect for $N= 200, 500$ is minute.  The power of the test is for independent observations even for $N=50$ acceptable  and for larger samples, we observe rapid improvements, even for greater values of  $k$. 
%In the dependent case for $N=100$  the test is less  powerful  for $k=4$  (whereas $k=5$ leads to inflated rejection at the boundary ).
In the dependent case for $N=200, 500$ we observe a good level approximation at the boundary  and high rejection probabilities under the alternative. Interestingly the bias-variance trade-off sometimes implies that while some $k$ leads to the optimal level approximation at the boundary  and thus high power close to the hypothesis, for larger distances smaller values of $k$ perform better, because the variance is smaller. This effect is reflected by crossing rejection curves. 
Finally, we notice that in view of the sometimes abrupt change in variance and bias for two successive values of $k$, in practice a soft threshold might be considered, for a more nuanced trade-off.

\subsection{Inference for relevant prediction errors}

We consider the set-up described in the previous section to investigate deviation in the predictive performance of $S$ and $S_0$. We begin by considering the smoothed kernels $S \Gamma^{1/2}$ and $S_0 \Gamma^{1/2}$, which are depicted in Figure~\ref{fig5} (for $\Gamma$ corresponding to the i.i.d.\ case). Even though they bear some resemblance to their originals (see Figure~\ref{fig2}), we observe a high smoothing effect caused  by the application of $\Gamma^{1/2}$.

 \begin{figure}[H]  
\begin{subfigure}{0.5\textwidth}
\centering
\includegraphics[width=1\linewidth,height=200pt]{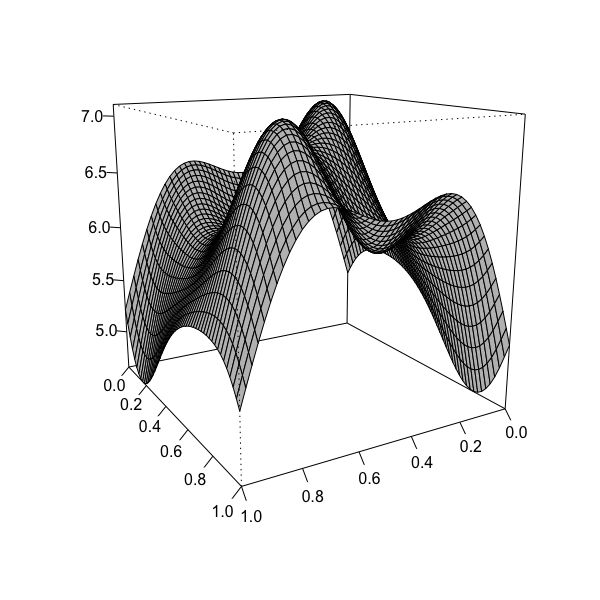}\\[-1ex]
 \end{subfigure}%
 \begin{subfigure}{0.5\textwidth}
 \centering
  \includegraphics[width=1\linewidth,height=200pt]{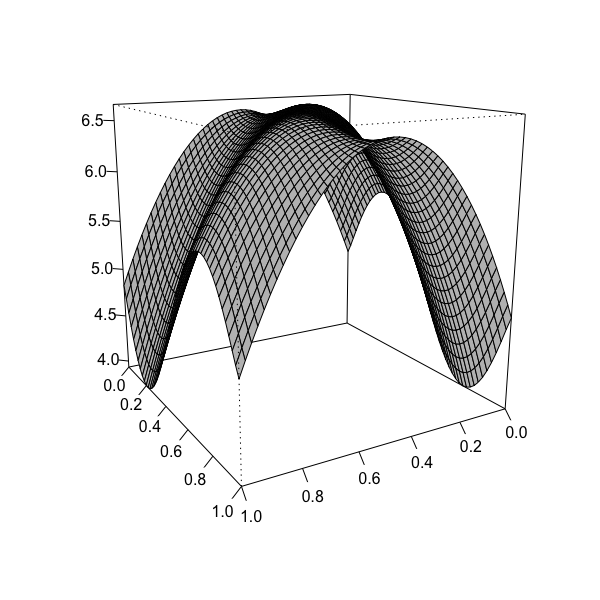}\\[-2ex]
 \end{subfigure}
  \caption{\label{fig5} \textit{ Image of the two integration kernels, plotted as surfaces. Left: $\varphi_{S \Gamma^{1/2}}$. Right: $\varphi_{S_0\Gamma^{1/2}}$. }}
 \end{figure}
As a consequence of the smoothing effect, we expect that the relative explanation should be higher than for the non-smoothed operators. This is in fact what we see in Figure~\ref{fig6} (left), where we have plotted the relative explanation measure
\begin{equation} \label{h8}
 \frac{\vvvert S\Gamma^{1/2}\Pi_{ k }- S_0\Gamma^{1/2} \Pi_{ k }\vvvert^2}{\vvvert S\Gamma^{1/2}-\ S_0\Gamma^{1/2} \vvvert^2}. 
 \end{equation}
We see that the first principal component already covers more than $ 75 \%$ of the norm, for $k = 3$ the relative explanation is about $99 \%$ and for $k =4$ it has reached almost $100\%$. Compared to the explanation for the non-smoothed kernel in Figure~\ref{fig3} (left) this is a very rapid increase and it suggests the use of smaller values for $k$. Notice that this matches our
theoretical results in Sections~\ref{sec3} and \ref{sec4} (compare Assumptions~\eqref{ass31}(7) and \ref{ass41})), which suggest higher $k$ for the recovery of the slope and smaller $k$ for the purpose of prediction. 
On the right side of Figure~\ref{fig6} we display  the smoothing kernel corresponding to the operator  $\Gamma^{1/2}$ in the case of i.i.d.\ observations.\\

  \begin{figure}[H]  
\begin{subfigure}{0.5\textwidth}
\centering
\includegraphics[width=0.9\linewidth,height=160pt]{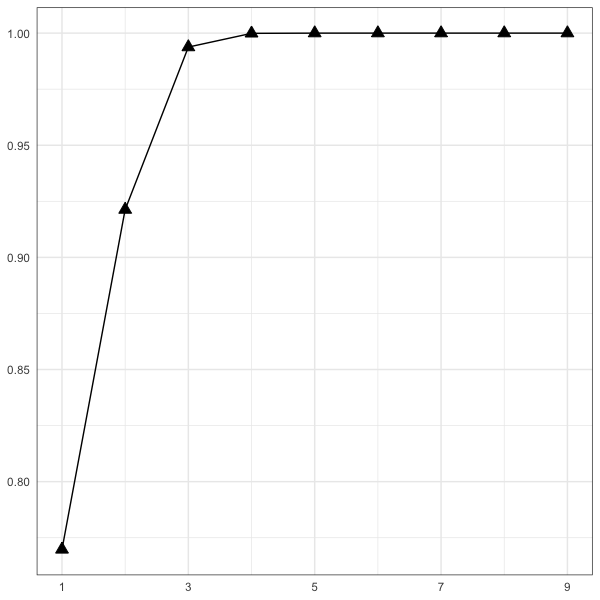}\\[-1ex]
 \end{subfigure}%
 \begin{subfigure}{0.5\textwidth}
 \centering
  \includegraphics[width=1\linewidth,height=200pt]{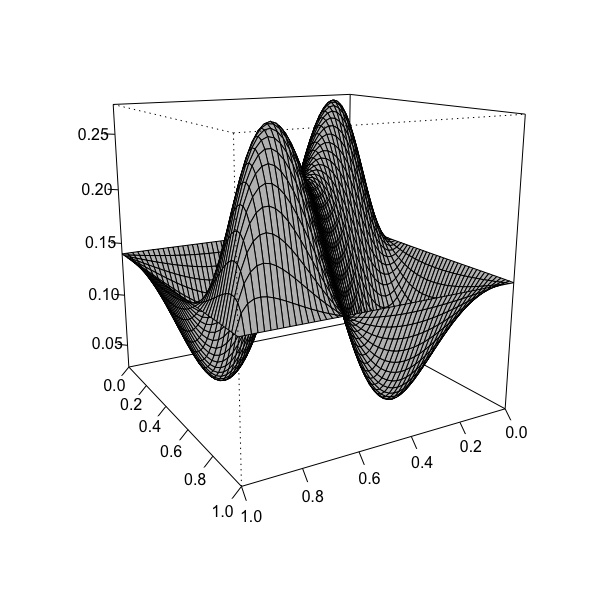}\\[-2ex]
 \end{subfigure}
  \caption{\label{fig6} \textit{ Left: Relative explanation measure defined  in \eqref{h8}  as a function of $k$. Right: Kernel of the operator $\Gamma^{1/2}$ (for i.i.d.\ observations). }}
 \end{figure}

We now proceed to the application of the statistical test \eqref{test_decision_2}, presented in Section~\ref{sec4} for the hypothesis~\eqref{relevant_prediction_error}.  As in  Section~\ref{sec61} we consider sample sizes $N = 50,  200, 500$ and parameter choices $k=1,2,3$, both for i.i.d.\ samples (left part of the figures) and dependent samples (right part of the figures). For details on the model as well as the dependence we refer to   Section~\ref{sec61}. The measure $\nu$ in the normalizer $\hat V_N^{\rm pred}$ (see \eqref{denominator_2}) is again chosen to be the uniform measure on $\{1/5, 2/5, 3/5, 4/5\}$ and the level of the test is $\alpha = 0.05$. All simulations are based on $1000$ repetitions. Notice that the absolute deviation (vertical grey line) differs in the case of independent and dependent data, since  the operator $\Gamma$  is different in the dependent and independent case.

\begin{figure}[H]  \ContinuedFloat
\begin{subfigure}{0.5\textwidth}
\centering
\includegraphics[width=0.85\linewidth,height=180pt]{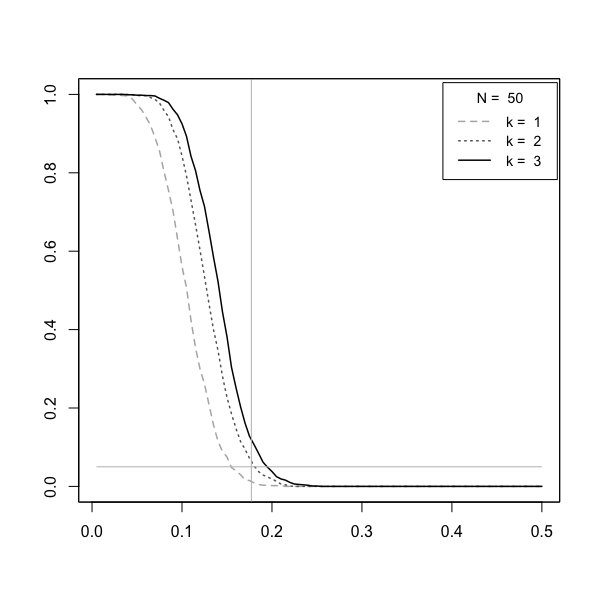}\\[-1ex]
 \end{subfigure}%
 \begin{subfigure}{0.5\textwidth}
 \centering
  \includegraphics[width=0.85\linewidth,height=180pt]{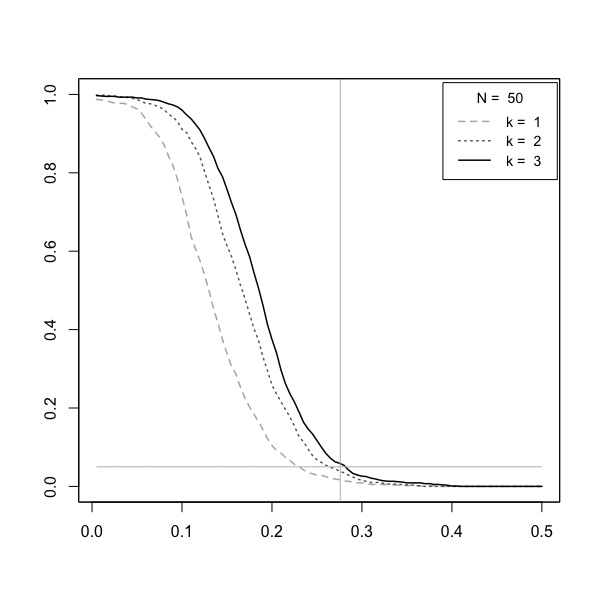}
 \end{subfigure}
  \end{figure}
   $ $\\[-15ex]
  %
 \begin{comment}
  \begin{figure}[H]  
 \begin{subfigure}{0.5\textwidth}
\centering
\includegraphics[width=0.85\linewidth,height=180pt]{preddep0N100.png}\\[-1ex]
 \end{subfigure}%
 \begin{subfigure}{0.5\textwidth}
 \centering
  \includegraphics[width=0.85\linewidth,height=180pt]{preddep1N100.png}\\[-2ex]
 \end{subfigure}
   \end{figure}
\end{comment}  
  \begin{figure}[H]  \ContinuedFloat
 \begin{subfigure}{0.5\textwidth}
\centering
\includegraphics[width=0.85\linewidth,height=180pt]{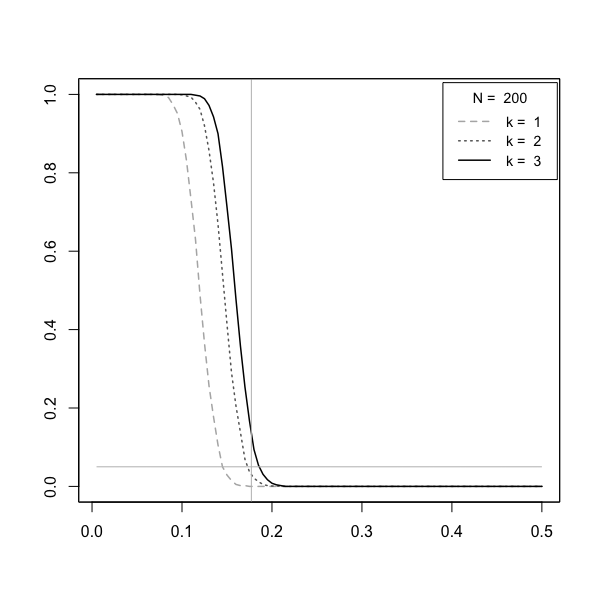}\\[-1ex]
 \end{subfigure}%
 \begin{subfigure}{0.5\textwidth}
 \centering
  \includegraphics[width=0.85\linewidth,height=180pt]{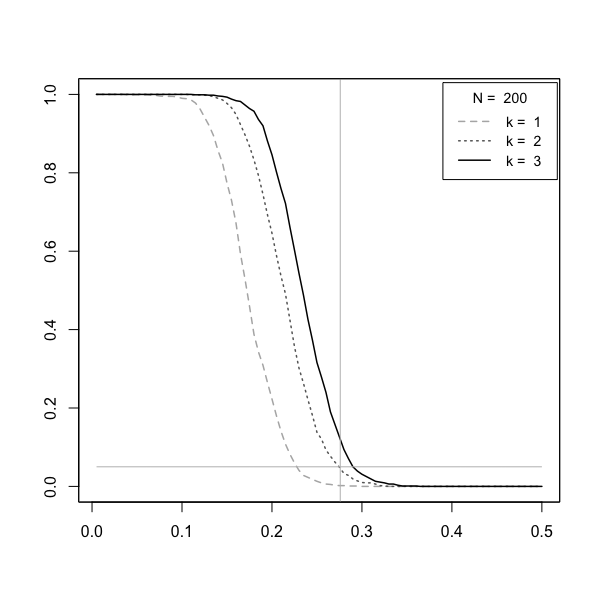}\\[-2ex]
 \end{subfigure}
  \end{figure}
 $ $\\[-15ex]
  \begin{figure}[H]  
 \begin{subfigure}{0.5\textwidth}
\centering
\includegraphics[width=0.85\linewidth,height=180pt]{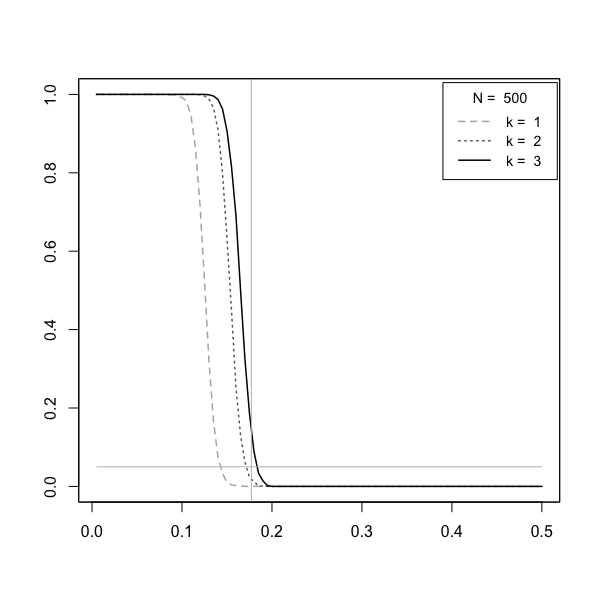}\\[-1ex]
 \end{subfigure}%
 \begin{subfigure}{0.5\textwidth}
 \centering
  \includegraphics[width=0.85\linewidth,height=180pt]{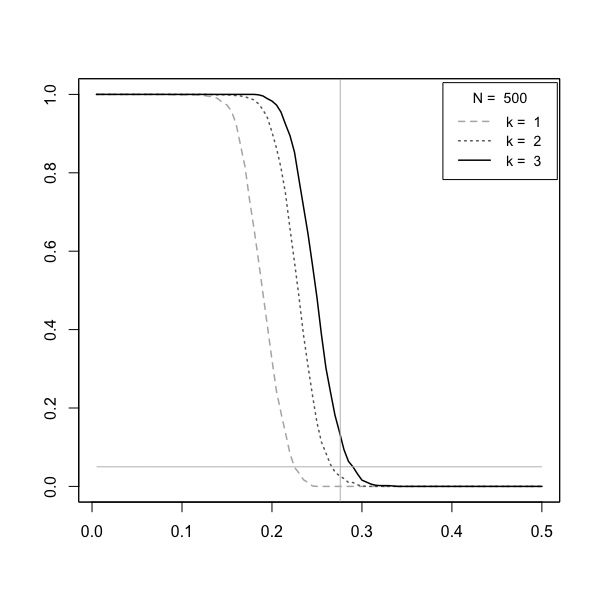}\\[-2ex]
 \end{subfigure}
   \caption{\label{fig7} \textit{Rejection probabilities ($y$-axis) of the test \eqref{test_decision_2} for different choices of $\Delta$ ($x$-axis). Different sizes of $N$ and $k$ are displayed, for i.i.d.\ observations (left) and dependent observations (right). The level $\alpha=0.05$ is indicated by the horizontal grey line, and the true size of $\vvvert S \Gamma^{1/2} - S_0 \Gamma^{1/2}\vvvert^2$ by a vertical line.}}
 \end{figure}

The numerical results confirm the theoretical findings in Section~\ref{sec4}. We observe a good approximation of the level  at the boundary  
of the hypothesis, both for dependent and independent data and accordingly high rejection probabilities under the alternative. The smoothing parameter is chosen smaller than in the case of prediction,  which corresponds to the smaller  bias in the case of prediction. In contrast to   Section~\ref{sec61}  we do not see pronounced crossing of the power curves for different $k$, such that better level approximation automatically translates into higher overall power. This also is an effect of the relatively small bias in the case of  prediction.

\bigskip
\bigskip

%{\bf Acknowledgements}
%This research was partially supported by the Collaborative Research Center `Statistical modeling
%of nonlinear dynamic processes' ({\it Sonderforschungsbereich 823, Teilprojekt A1, C1}).\\
%
%We are also very grateful to the reviewers   for their constructive comments on an earlier version of this paper.

 %While the absolute difference between the smoothed kernels is increased by the application of $\Gamma^{1/2}$ to 
 %$$\E \vvvert S X- S_0 X\vvvert^2 = \vvvert S \Gamma^{1/2}- S_0 \Gamma^{1/2}\vvvert^2 \approx 0.17 ,$$

 \nocite{*}
                                                                                                                                                                                                                                                                                                                            
                                                                                                                                                                                                                                                                                                                            %\addcontentsline{toc}{chapter}{\bibname}

%
%
%--------------------------------------------------------------------------------------------------------------
%
%
%
% 	APPENDIX
%
%
%
%--------------------------------------------------------------------------------------------------------------
%
%

\appendix

%
%
%--------------------------------------------------------------------------------------------------------------
%
%
%
% 	SECTION : APPENDIX A
%
%
%
%--------------------------------------------------------------------------------------------------------------
%
%

\section{Proofs and technical details}

The Appendix is dedicated to the proofs of the theoretical results from Sections~\ref{sec3} to \ref{sec5}. We only show the weak convergence results from Section~\ref{sec3} explicitly, as those in Sections~\ref{sec4} and  \ref{sec5}  are straightforward modifications.  In the Appendix~B we have collected some results concerning operators and their eigensystems. There (in Remark~\ref{Rem_eigenfunctions}) we also address the problem of the non-uniqueness of eigenfunctions, which is a technical issue, concerning the comparisons of eigenfunctions. Roughly speaking we always assume that the eigenfunctions $e_i$ and their empirical counterparts  (both unique up to sign) have "the same sign", in the sense that for all $\FP$ the inequality $\|e_i-\hat e_i[\FP]\| \le \|e_i+\hat e_i[\FP]\| $ holds. Notice that this technicality is of no concern in applications, as our test statistics always involve outer products of the form $\hat e_i[\FP] \otimes \hat e_i[\FP]$, which cancel the sign out. \\
Finally, we assume for notational simplicity that the sequential parameter $\FP$ is contained in the interval $[1/2, 1]$ (any interval $\left[ a, 1 \right]$,  $a>0$, can be dealt with in the same way). In the remainder of this introduction we recall some useful identities for sequential operators and introduce a suitable sup-norm for them. \\

Let us revisit the sequential statistics defined in Section~\ref{Subsection_weak_convergence}: Recall the definition of the sequential covariance estimate $\hat{ \Gamma }_N[ \FP ]$ (see \eqref{Eq_DefSeqCovOp}), its eigenvalues and eigenfunctions $\hat{\lambda}_i[ \FP ], \hat e_i[ \FP ]$, $i \ge 1$, its regularized inverse $\hat \Gamma_k^\dagger[ \FP ]$ (see \eqref{Eq_DefSeqPseudoInvs-and-SeqProj}), the projection 
$
\hat \Pi_{ k }[\FP],
$ 
 (see \eqref{Eq_DefSeqPseudoInvs-and-SeqProj})
 on $\operatorname{span}\{\hat e_1[ \FP ],...,\hat e_{ k }[ \FP ]\}$, which can be expressed as $\hat \Pi_{ k }[\FP]=\hat{ \Gamma}_N [ \FP ]  \hat{ \Gamma}^{ \dagger }_k [ \FP ]$ and finally the sequential estimate of $S$, denoted by $\hat S_N[ \FP ]$ (see \eqref{App_ref_3}). Notice that for $\hat S_N[\FP]$ an analogue decomposition to \eqref{h4} holds
\begin{equation*} %\label{h4_seq}
      \hat S_N[\FP] 
    = 
        S \hat \Pi_{ k } [\FP ] + U_N [\FP]\hat{ \Gamma}_k^\dagger[\FP],
\end{equation*}
where the sequential residual term is defined as
\begin{equation*} % \label{def_UN_seq}
        U_N[\FP] 
    :=
        \frac{1}{N} \sum_{n=1}^{\lfloor N \FP \rfloor} 
        \varepsilon_n \otimes X_n.
\end{equation*}

For fixed $\FP$, each of these statistics is defined as an element of a suitable Hilbert space ($\hat{\lambda}_i[\FP] \in \R$, $\hat e_i[\FP]\in H_1$, $\hat{ \Gamma }_N[\FP], \hat{ \Gamma}_k^\dagger[\FP], \hat{ \Pi }_{ k }[\FP]\in \mathcal{S}(H_1, H_1)$ and $\hat{ S }_N[\FP]\in \mathcal{S}(H_1, H_2)$). Alternatively, we may also view each of these statistics as a bounded function in $\FP$, mapping from $\left[ 1/2, 1 \right]$ into the respective Hilbert space. We make this notion more precise by defining the space of bounded functions:
\begin{defi} \label{definition_sequential_space}
Let $\mathcal{B}$ be a Banach space with norm $\vvvert \cdot \vvvert_\mathcal{B}$. Then we denote by 
$$
\ell^\infty(\mathcal{B}) := \big\{ f: [1/2,1] \to \mathcal{B}: \sup_{\FP \in [1/2, 1 ] } \vvvert f(\FP)\vvvert_\mathcal{B}  < \infty \big\},
$$
the space of all bounded functions with range $\mathcal{B}$. Endowed with the sup-norm $\ell^\infty(\mathcal{B})$ is itself a Banach space.
\end{defi}

In the sense of Definition~\ref{definition_sequential_space}, we see that e.g.\ $\hat{ \Gamma }_N [ \cdot ] \in \ell^\infty(\mathcal{S}(H_1, H_1))$. We conclude this part with the observation that the sequential covariance estimator $\hat{ \Gamma}_N[ \FP ]$ is asymptotically close to the true one $\FP \Gamma$. 

\begin{theo} (\cite{dette2021}) \label{Theo_StochBoundFuncCovOp}
	Under the Assumptions \eqref{ass31}(2) and (3)  it holds that
$$
		\sup_{\FP \in [1/2, 1 ] } \vvvert  \sqrt{ N }  (\hat{ \Gamma}_N [ \FP ] - \FP \Gamma ) \vvvert
	=
		\mathcal{ O }_{ \PP } (1 ).
%	\stackrel{ d }{ \to }
%		\mathcal{ G }
$$
\end{theo}
    We now give the proof of the main Theorem~\ref{theorem_1}. To make the proof 
    easier to comprehend, the discussion of various remainders is bundled in later lemmata.

%
%
%--------------------------------------------------------------------------------------------------------------
%
%
% 	SUBSECTION : PROOF OF MAIN THEOREM
%
%
%
%
%

\subsection{Proof of Theorem~\ref{theorem_1} } \label{Subsection_App_A_3}

The proof consists of two steps: First we derive an asymptotic linearization of the test statistic. Subsequently we show weak convergence to a Brownian motion.\\

	Using the identity $ \| a \|^2 - \| b\|^2 = 2 \langle a - b,  b \rangle + \| a - b \|^2$ 
	(which is a version of the third binomial formula for inner products), it follows that
\begin{align}\label{Eq_FuncWeakConv-InProd-Eq-ThirdBinForm}
	& \FP\sqrt{ N }  \left\{ 
			\vvvert \hat{ S }_N [ \FP ] - S_0  \hat{ \Pi }_{ k } [ \FP ] \vvvert^2  
			-
			\vvvert  S \Pi_{ k } - S_0 \Pi_{ k } \vvvert^2 
		\right\} \\
=& 	2 \FP \sqrt{ N } 
		 \langle  \hat{ S }_N [ \FP ] - S_0  \hat{ \Pi }_{ k } [ \FP ] -  S \Pi_{ k } + S_0 \Pi_{ k },
		   S \Pi_{ k } - S_0 \Pi_{ k }  \rangle 
	+
	\FP\sqrt{ N } \mathcal{R}_1^2 [ \FP ]. \nonumber
\end{align}
Here the term $\mathcal{R}_1[ \FP]$ is defined as 
\begin{align} \label{remainder_R_1}
	\mathcal{R}_1 [ \FP ] 
:=&
		 \vvvert \hat{ S }_N [ \FP ] - S_0  \hat{ \Pi }_{ k } [ \FP ] -  S \Pi_{ k } + S_0 \Pi_{ k } \vvvert
	\\
=& 
		\vvvert 
			( S  - S_0 ) ( \hat{ \Pi }_{ k } [ \FP ] -  \Pi_{ k } ) 
			+ 
			U_N [ \FP ]  \hat{ \Gamma}^{ \dagger }_k [ \FP ]
		\vvvert. \nonumber
\end{align}
	
	Using the triangle inequality we see that
	$
	\mathcal{R}_1 [ \FP ] \le
	\vvvert 
			( S  - S_0 ) ( \hat{ \Pi }_{ k } [ \FP ] -  \Pi_{ k } ) 
	\vvvert 
			+ 
	\vvvert 
			U_N [ \FP ]  \hat{ \Gamma}^{ \dagger }_k [ \FP ]
	\vvvert.
	$
Recalling that $S-S_0$ can be rewritten as $(R-R_0) \Gamma^\beta$ for suitable Hilbert--Schmidt operators $R, R_0$ (see Assumption~\ref{ass31}(1)), implies that the first term on the right is bounded by
\begin{equation}\label{Eq_R_1_bound_op_norm}
\vvvert 
			( S  - S_0 ) ( \hat{ \Pi }_{ k } [ \FP ] -  \Pi_{ k } )  \vvvert  \le \vvvert 
			 R - R_0 \vvvert \vvvert  \Gamma^\beta( \hat{ \Pi }_{ k } [ \FP ] -  \Pi_{ k } )  \vvvert_\mathcal{L}.
\end{equation}
	Now Lemma~\ref{Lem_Bounds-NonSmoothed-Var-and-Bias-of-S_N}, together with the bound \eqref{Eq_R_1_bound_op_norm}
	implies for an arbitrarily small $\epsilon>0$, that
\begin{equation} \label{rate_R_1}
         \mathcal{R}_1 [ \FP ] 
	 = 
%			\mathcal{O}_{ \PP } ( 1 / \sqrt{N} )
%			+
			 \mathcal{ O }_{ \PP } ( k^{  (\SM + 1)/2 +  \epsilon}/   \sqrt{N} ).
%	=
%		o_{ \PP } ( 1 ).
\end{equation}
%for an arbitrarily small $\epsilon>0$, which by assumption is $o_{ \PP } ( 1 )$ for small enough $\epsilon$.
Consequently  $\FP \sqrt{ N } \mathcal{R}_1^2 [ \FP ]=	o_{ \PP } ( 1   )$	whenever $\epsilon$ is chosen small enough, i.e., $\epsilon \le 2 \delta$ (see Assumption~\ref{ass31}(7)). 
We now focus on the non-vanishing term on the right of \eqref{Eq_FuncWeakConv-InProd-Eq-ThirdBinForm}. 
It can be further decomposed into the sum 
$T[ \FP ] +  \mathcal{R}_2[ \FP ]$ where
\be \label{Eq_MainTheo-Def-dummyVariable-T}
	    T[ \FP ]
	:=
	    2  \FP \sqrt{ N } \,
	    \left\langle  
	        ( S  - S_0 ) 
	        ( \hat{ \Pi }_{ k } [ \FP ] -  \Pi_{ k } ) 
	        + 
	        U_N [ \FP ] \hat{ \Gamma }^{ \dagger }_k [ \FP ], \;
	     S -   S_0  
	    \right\rangle
\ee
	and $ \mathcal{R}_2  [ \FP ]$ is another remainder term, defined as
$$
	\mathcal{R}_2  [ \FP ]:= 2  \FP\sqrt{ N } 
	\left\langle  
		( S  - S_0 ) ( \hat{ \Pi }_{ k } [ \FP ] -  \Pi_{ k } ) + U_N  [ \FP ]  \hat{ \Gamma}^{ \dagger }_k [ \FP ], \;
		  ( S -   S_0 ) ( \Pi_{ k }   - \operatorname{Id} ) 
	\right\rangle.
$$

	By the Cauchy--Schwarz inequality one has
$$
		| \mathcal{R}_2[\FP] |
	\le
		 2   \FP\sqrt{ N } \mathcal{R}_1  [ \FP ]  
		  \vvvert ( S -   S_0 ) ( \Pi_{ k }    - \operatorname{Id} ) \vvvert.
$$
From \eqref{rate_R_1} we know that 
$
\sqrt{N} \mathcal{R}_1[\FP]
=
\mathcal{O}_{\PP}(k^{\SM/2+1/2+\epsilon})
$  and according to our  discussion of Assumption \ref{ass31}(7)  (see Remark~\ref{rem1} (c)) 
$
\vvvert ( S -   S_0 ) ( \Pi_{ k }    - \operatorname{Id} ) \vvvert 
= 
\mathcal{O}(k^{-\SM\beta})$. Hence it follows that 
$$
    \mathcal{R}_2[\FP]
 =
    \mathcal{O}_{ \PP } 
    \left( k^{ \SM/2 + 1/2 + \epsilon - \SM \beta } \right),
$$
which is $o_{\PP}(1)$ for a sufficiently small choice of $\epsilon$, namely $\epsilon < ( \SM + 1 ) /2$ where we used $\beta>1+1/\SM$ (again see discussion of Assumption~\ref{ass31}(7) in Remark~\ref{rem1} (c)).
We now analyze the non-vanishing term $T[\FP]$ defined in equation \eqref{Eq_MainTheo-Def-dummyVariable-T}. Recall that by Assumption~\ref{ass31}(1) two Hilbert--Schmidt operators $R, R_0$ exist, such that $S=R \Gamma^\beta $ and $S_0=R_0 \Gamma^\beta$.  It thus follows that
\begin{align*}
	T[\FP]%= &2  \FP\sqrt{ N } \,
	 %\left\langle  ( S  - S_0 ) ( \hat{ \Pi }_{ k } [ \FP ] -  \Pi_{ k } ) + U_N[\FP]  \hat{ \Gamma}^{ \dagger }_N [ \FP ], \;
	 %  S -   S_0  \right\rangle 
	 %\\
	= &2  \FP\sqrt{ N } \,
	 \left\langle  ( R -   R_0 ) \Gamma^\beta  ( \hat{ \Pi }_{ k } [ \FP ] -  \Pi_{ k } ) + U_N[\FP]  \hat{ \Gamma}^{ \dagger }_k [ \FP ], \;
	   ( R -   R_0 ) \Gamma^\beta \right\rangle 
	 \\
=&
	2  \FP\sqrt{ N } \,
	 \left\langle \left\{  ( R -   R_0 ) \Gamma^\beta   ( \hat{ \Pi }_{ k } [ \FP ] -  \Pi_{ k } ) + U_N[\FP]  \hat{ \Gamma}^{ \dagger }_k [ \FP ] \right\}
		\Gamma^{ \beta }, \;
	   R  - R_0   \right\rangle   	
\end{align*}
	Notice that we have "shifted" the smoothing operator $\Gamma^\beta$ from the right side of the inner product to the left. This shift can be validated by basic calculations. However it can be seen more easily as an application of the cyclical property for the trace of operators (see \cite{HorKokBook12}, Section~13.5 for details). 
	Finally, we use  
	Lemma~\ref{Lem_Projection-Statistics-Smoothed-Right-and-Left}
to replace 
$ \FP \Gamma^{ \beta }  ( \hat{ \Pi }_{ k }  -  \Pi_{ k } )\Gamma^{ \beta }$ 
and 
$\FP U_N[\FP]  \hat{ \Gamma}^{ \dagger }_k [ \FP ]\Gamma^{ \beta }$ on the right, by their asymptotic linearizations, which yields 
$ 	
	T[ \FP ]
=
    T'[ \FP ]+o_{ \PP }( 1 )$, 
    where
\begin{equation} \label{Eq_def_T_prime}
    T'[ \FP ]=2  \sqrt{ N } \,
	 \left\langle  ( R  - R_0 )  L (\hat \Gamma_N[\FP]-\FP \Gamma)  + U_N[\FP]   \Gamma^{ \beta -1} , \;
	   R  - R_0   \right\rangle.
\end{equation}	
Here $L$ is a linear map defined in Definition~\ref{Def_LinerazationOperat}. Our proof up to this point implies the (asymptotic) stochastic linearization 
$$
\sqrt{N} \left\{ 
			\vvvert \hat{ S }_N [ \FP ] - S_0  \hat{ \Pi }_{ k } [ \FP ] \vvvert^2  
			-
			\vvvert  S \Pi_{ k } - S_0 \Pi_{ k } \vvvert^2 
		\right\}  = T'[ \FP ]+o_{ \PP }( 1 ).
$$
Now, having completed the linearization we still have to show that $T'[ \FP ]$ converges weakly to $\tau \mathbb{B}[ \FP ]$. By definition of $T'[ \FP ]$ (in \eqref{Eq_def_T_prime}), it can be written as a real valued, sequential sum process of $\phi$-mixing random variables. Therefore, we can  apply the invariance principle from Corollary~2.6 in \cite{herrndorf1983} (where Condition~(E) can be verified using our Assumption~\ref{ass31}(2), with $\delta:= \kappa/2$). This directly implies
$T'[ \FP ] \stackrel{d}{ \to } \tau \mathbb{B}[ \FP ]$,
where $\tau^2$ is the long-run  variance of the statistic. It is now easy to calculate that  $\tau^2$ has the form postulated in \eqref{def_long_run_variance}, which
 concludes the proof. \hfill $\square$\\

\begin{comment}
Similarly as for $\tau$, we can express it as 
Here $\tilde L$ is a linearization map, defined as $\tilde L = \sum_{i \ge 1} \tilde L_i + (\tilde L_i)^{ \ast }$ and $\tilde L_i$ is defined as $L_i$ in Definition~\ref{Def_LinerazationOperat}, where $\beta$ has to be replaced by $\beta+1/2$. Similarly one can derive explicit formulas for the long-run variances of the change point statistics
%
\begin{align} \label{long_run_cp}
\tau^{\rm cp} := & \lim_{N \to \infty} Var \Big(
     \sqrt{N} \big(\vvvert  \hat S_N^{(1)}- \hat S_N^{(2)} \vvvert ^2 - \vvvert   S^{(1)}-  S^{(2)}\vvvert ^2 \big) \Big)\\
\tau^{\rm cp, pred} := & \lim_{N \to \infty} Var \Big(
     \sqrt{N} \big(\vvvert   \hat S_N^{(1)} (\hat \Gamma_N^{(1)})^{1/2} -  S^{(2)} (\Gamma^{(2)})^{1/2}\vvvert ^2 - \vvvert  [ S^{(1)}-  S^{(2)}]\Gamma^{1/2}\vvvert ^2 \big) \Big), \label{long_run_cp_pred}
\end{align}
in Theorem~\ref{theorem_5}. However we do not calculate them explicitly to avoid redundancy. 
\\
\end{comment}

In the subsequent sections we show several bounds for the remainder terms, we have used throughout the proof of Theorem~\ref{theorem_1}.

%
%
%--------------------------------------------------------------------------------------------------------------
%
%
% 	SUBSECTION : NON-SMOOTHED REMAINDERS
%
%
%
%
%

\subsection{Bounds for $\mathcal{R}_1$}

    In the next Lemma~\ref{Lem_Bounds-NonSmoothed-Var-and-Bias-of-S_N} we give orders of magnitude for 
$
    \vvvert 
    U_N [ \FP ]  \hat{ \Gamma}^{ \dagger }_k [ \FP ] 
    \vvvert
    $ 
    and 
    $  
    \vvvert 
    \Gamma^\beta( \hat{ \Pi }_{ k } [ \FP ] -  \Pi_{ k } )  \vvvert_\mathcal{L}.
    $
Together with \eqref{Eq_R_1_bound_op_norm} (and the following part up to \eqref{rate_R_1}) these imply that $\mathcal{R}_1$ satisfies the decay rate \eqref{rate_R_1}.

\begin{lem}\label{Lem_Bounds-NonSmoothed-Var-and-Bias-of-S_N}
	Under the conditions of Theorem~\ref{theorem_1}, it holds that

\begin{itemize} 
	\item[ i) ] 
$$
		\sup_{ \FP \in \left[1/2, 1 \right] } \vvvert \Gamma^{ \beta } ( \hat{ \Pi }_{ k } [ \FP ] -  \Pi_{ k } ) \vvvert_{ \mathcal{ L } }
	=
		\mathcal{ O }_{ \PP } ( 1 / \sqrt{ N } ).
$$
	\item[ ii) ] For any $\epsilon>0$
$$
		\sup_{\FP \in [1/2, 1 ] } \vvvert  U_N[\FP ]  \hat{ \Gamma}^{ \dagger}_k[\FP ]  \vvvert
	=
		\mathcal{ O }_{ \PP } ( k^{ \SM/2+1/2+\epsilon }  / \sqrt{N} )	.
$$

%
%	\item[ iii) ]

\end{itemize}
	
\end{lem}

\begin{proof} \mbox{}
	\begin{itemize}
	\item[ i) ]
		%The proof runs along the same lines as the proof of 
		%Theorem~\ref{Lem_Projection-Statistics-Smoothed-Right-and-Left} 
		%except that one has less smoothness. Hence, we merely indicate the new bounds one has here.

We use \eqref{Eq_DefSeqPseudoInvs-and-SeqProj} and the decomposition 
$$
\Gamma^{ \beta } ( \hat{ \Pi }_{ k } [ \FP ] -  \Pi_{ k }) =A_1[ \FP ]+A_2[ \FP ]+A_3[ \FP ],
$$ 
where
\begin{align} 
	    A_1[ \FP ]
	&:=\sum_{ i = 1 }^{ k } 
    	\Gamma^{ \beta } 		
		 (  \hat{ e }_i [ \FP ]  - e_i )  \otimes e_i 
\label{Eq_UnsmoothRemProj-A1}
\\	
		A_2[ \FP ]
	&:=\sum_{ i = 1 }^{ k } 
	    \Gamma^{ \beta }   e_i    \otimes ( \hat{ e }_i [ \FP ]  - e_i )
\label{Eq_UnsmoothRemProj-A2}
\\	   
		A_3[ \FP ]
	&:=\sum_{ i = 1 }^{ k } 
	    \Gamma^{ \beta }  
	    ( \hat{ e }_i [ \FP ]  - e_i )  \otimes  ( \hat{ e }_i [ \FP ]  - e_i )   
\label{Eq_UnsmoothRemProj-A3}		.
\end{align}
	 We now show the desired rate for each term separately. 
	\begin{enumerate}
	\item[$A_1$) ]
	   Recall the spectral decomposition of the covariance operator $\Gamma= \sum_{q\ge 1} \lambda_q e_q \otimes e_q$, which yields
	   $$
	   A_1[ \FP ] 
	  = 
	   \sum_{ i = 1 }^{ k }  \sum_{ q \ge 1 } \lambda_q^{ \beta }
	   \langle  \hat{ e }_i [ \FP ]  - e_i , e_q \rangle
	   e_q \otimes e_i.
	   $$
	   Separating the terms where $i=q$ from the ones where $i \neq q   $, we decompose  
	   $$ 
	   A_1 [ \FP ] = A_{1,1}[ \FP ] + A_{1,2}[ \FP ],
	   $$ with
\begin{align}
        A_{1,1}[ \FP ]
    &:=
        \sum_{ i = 1 }^{ k } \lambda_i^{ \beta }
        \langle  \hat{ e }_i [ \FP ]  - e_i , e_i \rangle
	   (e_i \otimes e_i)
	 =
	    \sum_{ i = 1 }^{ k }  
    	 \frac{ - \lambda^{  \beta}_i }{ 2 }  
    	 \left\|  \hat{ e }_i [ \FP ]  - e_i  \right\|^2 
    	 ( e_i \otimes e_i ),
\label{Eq_UnsmoothRemProj-A11}
\\
        A_{1,2}[ \FP ]
    &:=
        \sum_{ i = 1 }^{ k } 
	    \sum_{ q \neq i } \lambda_q^{ \beta }
	    \langle  \hat{ e }_i [ \FP ]  - e_i , e_q \rangle    ( e_q \otimes e_i ).
 \label{Eq_UnsmoothRemProj-A12}       
\end{align}
     Notice that we used identity 
    $
    \langle  \hat{ e }_i [ \FP ]  - e_i , e_i \rangle
    =
    - \left\|  \hat{ e }_i [ \FP ]  - e_i  \right\|^2/2
    $
      in \eqref{Eq_UnsmoothRemProj-A11} (see Appendix \ref{misc}, Lemma \ref{Lem_bounds_op}).
     We bound the operator norm of \eqref{Eq_UnsmoothRemProj-A11} and \eqref{Eq_UnsmoothRemProj-A12} individually. \\
     \fbox{$A_{1,1}$:} Recall that the operator norm of a diagonal operator equals its largest, absolute diagonal entry, i.e.
\begin{equation} \label{Eq_deal_with_ev}
        \vvvert A_{1,1}[\FP]\vvvert_\mathcal{L} 
    = 
        \max_{1 \le i \le k}
    	 \frac{  \lambda^{  \beta}_i }{ 2 }  
    	 \left\|  \hat{ e }_i [ \FP ]  - e_i  \right\|^2.
\end{equation}
Further, using the inequality
    $$
    \left\|  \hat{ e }_i [ \FP ]  - e_i  \right\|
    \leq \frac{2 \sqrt{2}
    \vvvert \hat{\Gamma}_N[ \FP ] - \FP \Gamma\vvvert_\mathcal{ L }^2
    }{ 
     \FP \min( \lambda_{ i - 1 }- \lambda_i, \lambda_i - \lambda_{ i + 1 } )},
    $$
    (see Appendix \ref{misc} Lemma \ref{Lem_bounds_op}) we have 
$$
        \max_{1 \le i \le k}
	     \frac{  \lambda^{  \beta}_i }{ 2 }  
	     \left\|  \hat{ e }_i [ \FP ]  - e_i  \right\|^2 
	 \le 
	    \max_{1 \le i \le k}
	    \frac{  2 \sqrt{2} 
	        \lambda^{  \beta}_i 
	        \vvvert \hat{\Gamma}_N[ \FP ] - \FP \Gamma\vvvert_\mathcal{ L } 
	     }{
	          \min( \lambda_{ i - 1 }- \lambda_i, \lambda_i - \lambda_{ i + 1 } )
	    },
$$
    where we have used that $1/\FP \le 2$.
    We now simplify the right side by Lemma~\ref{Lem_Bound-Ratio-Eigenvalues} part $iii)$ and the fact that $\vvvert\hat \Gamma_N[\FP]- \FP \Gamma\vvvert_\mathcal{L} =  \mathcal{O}_{\PP}(1/\sqrt{N})$ 
    (Theorem~\ref{Theo_StochBoundFuncCovOp}). Together these show that
$$
	\sup_{ \FP \in \left[1/2, 1 \right] }\vvvert A_{1,1}[\FP]\vvvert_\mathcal{L} 
	=	
		\mathcal{ O }_{ \PP } ( 1/ N ).
$$
	
	\fbox{$A_{1,2}$:}
	Since  the norm inequality
	$\vvvert \cdot \vvvert_{ \mathcal{ L } } \le  \vvvert \cdot \vvvert$ holds,
	 it is enough to show that  
	 $$
	 \sup_{\FP \in [1/2,1]}\vvvert A_{1,2}[\FP]  \vvvert = \mathcal{O}_\PP(1/\sqrt{N}).
	 $$
	 The identity~\eqref{Eq_IdentInprodDiff2Eigvect-with-other-Eigvect} implies that, 
	 $
	 \langle  \hat{ e }_i [ \FP ]  - e_i , e_q \rangle 
	 = 
	 \langle  \hat{ \Gamma }_N[ \FP ]  - \FP \Gamma,
	        \hat{ e }_i [ \FP ] \otimes e_q 
	  \rangle
	  /
	  ( \hat{ \lambda }_i [ \FP ] - \FP \lambda_q ).
	 $ 
    We can now upper bound the Hilbert--Schidt norm of $A_{1,2} [ \FP ]$ as follows:
\begin{align} \label{bound_A_12}
    	\vvvert  A_{1,2} [ \FP ] \vvvert 
	= &
	    \left(
		\sum_{ i = 1 }^{ k } \sum_{ q \neq i }
		\lambda_q^{ 2 \beta } 
		 \frac{ 
		 \langle  \hat{ \Gamma }_N [ \FP ]  - \FP \Gamma , \hat{ e }_i [ \FP ]  \otimes e_q \rangle^2
		}{ 
		( \hat{ \lambda }_i [ \FP ] - \FP \lambda_q )^2
		}
		\right)^{ 1 / 2 } \\
    \le &
        \max_{ i \le k, q \neq i }  \frac{ \lambda_q^{  \beta }  }{ | \hat{ \lambda }_i [ \FP ] - \FP \lambda_q |	}
        \left(
		\sum_{ i = 1 }^{ k } \sum_{ q \neq i }
		 \langle  \hat{ \Gamma }_N [ \FP ]  - \FP \Gamma , \hat{ e }_i [ \FP ]  \otimes e_q \rangle^2
		\right)^{ 1 / 2 } \nonumber\\
	\le &
		\max_{ i \le k, q \neq i }  \frac{ \lambda_q^{  \beta }  }{ | \hat{ \lambda }_i [ \FP ] - \FP \lambda_q |	}
		\vvvert  \hat{ \Gamma }_N [ \FP ]  - \FP \Gamma \vvvert. \nonumber	
\end{align}
    Here we have used that $\{\hat{ e }_i [ \FP ]\}_{i \in \N}, \{e_q\}_{q \in \N}$ are ONBs and thus their products $\{\hat{ e }_i [ \FP ] \otimes e_q \}_{i, q \in \N}$ form an ONB of the Hilbert--Schmidt operators (see Section \ref{Subsection_Operators}). 
    The fraction of the eigenvalues is uniformly of order $\mathcal{ O }_{ \PP }( 1 )$ by Lemma~\ref{Lem_Bound-Ratio-Eigenvalues}, part $ii)$ whereas $\vvvert  \hat{ \Gamma }_N[ \FP ]  - \FP \Gamma \vvvert = \mathcal{ O }_{ \PP }( 1 / \sqrt{ N } )$ by Theorem~\ref{Theo_StochBoundFuncCovOp}. 
    
    Combining both estimates  gives the (uniform) order $\mathcal{ O }_{ \PP }( 1/\sqrt{N} )$ for the term $A_1$.

	\item[ $A_2$) ]
	As $\Gamma^\beta e_i = \lambda_i^\beta e_i$ we can rewrite the operator $A_2[ \FP ]$ (defined in \eqref{Eq_UnsmoothRemProj-A2}) as follows
$$
	 A_2[ \FP ]=\sum_{ i = 1 }^{ k }  \lambda_i^{ \beta } e_i \otimes ( \hat{ e }_i [ \FP ]  - e_i ).
$$
	Then applying the Fourier expansion   
	$
	\hat{ e }_i [ \FP ]  - e_i 
	=
	\sum_{q \ge 1}  \langle  \hat{ e }_i [ \FP ]  - e_i , e_q \rangle e_q
	$, we can expand
	$A_2 [ \FP ] $  into 
\begin{align*}
    & \sum_{ i = 1 }^{ k }\sum_{q \ge 1}   \lambda_i^{ \beta } \langle  \hat{ e }_i [ \FP ]  - e_i , e_q \rangle e_i \otimes e_q \\ 
    = &
	 \sum_{ i = 1 }^{ k }   \frac{ - \lambda^{  \beta}_i }{ 2 }  \left\|  \hat{ e }_i [ \FP ]  - e_i  \right\|^2 e_i \otimes e_i
	+
	\sum_{ i = 1 }^{ k } \sum_{ q \neq i }  
	\lambda_i^{ \beta } \langle \hat{ e }_i [ \FP ]  - e_i, e_q \rangle  e_i \otimes e_q.
\end{align*}
    In the second line we have split up the terms for $q =i$ and $q \neq i$ and have employed identity~\eqref{Eq_IdentSqauNorm-of-Normal-Vect} for the first term.
	Proceeding as for $A_1 [ \FP ]$ now yields the  (uniform) order of $\mathcal{ O }_{ \PP }( 1 / \sqrt{ N } )$ for the term $A_2 [ \FP ]$.

	\item[ $A_3$) ] Finally we turn to $A_3[ \FP ]$ in \eqref{Eq_UnsmoothRemProj-A3}. Again we use the spectral decomposition $\Gamma= \sum_{q\ge 1} \lambda_q e_q \otimes e_q$ to rewrite this term as 
$$
    A_3[ \FP ] 
    = \sum_{ i = 1 }^{ k } \sum_{  q \ge 1 } 
    \lambda_q^{ \beta }
    \langle  \hat{ e }_i [ \FP ]  - e_i, e_q \rangle
    e_q \otimes (\hat{ e }_i [ \FP ]  - e_i ).
$$
    Employing the Fourier expansion of 
	$ \hat{ e }_i [ \FP ]  - e_i$ for the right factor of the outer product, gives the further expansion
$$
     \sum_{ i = 1 }^{ k } \sum_{  q \ge 1 } \sum_{  l \ge 1 } 
    \lambda_q^{ \beta }
    \langle  \hat{ e }_i [ \FP ]  - e_i, e_q \rangle
    \langle  \hat{ e }_i [ \FP ]  - e_i, e_l \rangle
    e_q \otimes e_l.
$$
    Though superficially more complicated, this expansion in terms of the product basis $\{e_q \otimes e_l\}_{ q, l \in \N}$ (see Section \ref{Subsection_Operators}) is handy, to decompose $A_3[ \FP ] $ into more simple parts. More precisely we set
	$$
	A_3[ \FP] = \sum_{m =1 }^4 A_{3,m }[ \FP ] ,
	$$ where
\begin{align*}
        A_{3,1}[ \FP ]
    &:=
        \sum_{ i = 1 }^{ k } \lambda_i^{  \beta } 
	     \langle  \hat{ e }_i [ \FP ]  - e_i, e_i \rangle^2
	    e_i \otimes e_i,	
%\label{Eq_UnsmoothRemProj-A31} 
\\
	    A_{3,2}[ \FP ]
	&:=
    	\sum_{ i = 1 }^{ k } 
    	\sum_{  q \neq i } \lambda_q^{ \beta }
	   	\langle  \hat{ e }_i [ \FP ]  - e_i, e_i \rangle 
		\langle  \hat{ e }_i [ \FP ]  - e_i, e_q \rangle 
		e_q \otimes e_i
		%\frac{ \| \hat{ \Gamma } [ \FP ]  - \Gamma\|^2 }{ ( \hat{ \lambda }_i [ \FP ]  - \FP \lambda_{ j } )^2 }
		%\frac{ \lambda_q^{ 2 \beta } }{ 4 } \|  \hat{ e }_i [ \FP ]  - e_i \|^4
%\label{Eq_UnsmoothRemProj-A32}	
\\
        A_{3,3}[ \FP ]
    &:=
	    \sum_{ i = 1 }^{ k } 
	    \sum_{  q \neq i } \lambda_i^{ \beta }
		\langle  \hat{ e }_i [ \FP ]  - e_i, e_i \rangle 
		\langle  \hat{ e }_i [ \FP ]  - e_i, e_q \rangle 
		e_i \otimes e_q, 
%\label{Eq_UnsmoothRemProj-A33}
\\
	    A_{3,4}[ \FP ]
	&:=
	    \sum_{ i = 1 }^{ k }
	    \sum_{  l, q \neq i } \lambda_q^{ \beta }
	 	\langle  \hat{ e }_i [ \FP ]  - e_i, e_q \rangle 
		\langle  \hat{ e }_i [ \FP ]  - e_i, e_l \rangle 
		e_q \otimes e_l
%\label{Eq_UnsmoothRemProj-A34}
\end{align*}
    
     We can now prove the uniform rate of $\mathcal{O}_{\PP}(1/\sqrt{N})$ for each of these terms individually.
     
     \fbox{$A_{3,1}$:}	
	Identity~\eqref{Eq_IdentSqauNorm-of-Normal-Vect} in Appendix \ref{misc} implies that
	$
	\langle  \hat{ e }_i [ \FP ]  - e_i, e_i \rangle^2
	= 
	\|  \hat{ e }_i [ \FP ]  - e_i \|^4 /4.
	$
	Now notice that $A_{3,1}[ \FP ]$ is a positive definite,  diagonal operator, such that its operator norm equals its largest diagonal value, which implies
$$
	    \sup_{ \FP \in \left[1/2, 1 \right] } \vvvert A_{3,1}[ \FP ] \vvvert_\mathcal{L} = 
		\sup_{ \FP \in \left[1/2, 1 \right] }
		\max_{ i \le k }  \frac{  \lambda^{  \beta}_i }{ 4 }  \left\|  \hat{ e }_i [ \FP ]  - e_i  \right\|^4.
$$
	As before (in the analysis of $A_{1,1}$) we can use the inequality~\eqref{Eq_Inequalty-Norm-Diff-Eigenvect} and Lemma~\ref{Lem_Bound-Ratio-Eigenvalues}, to show that the right side is of order
	$
	\mathcal{ O }_{ \PP }( k^{ 3 ( \SM +1 ) } / N^2 )
=
	\mathcal{ O }_{ \PP }( 1 / \sqrt{ N } )$.\\
	
	We bound the Hilbert--Schmidt norm of the remaining terms (since 
	$
	\vvvert \cdot \vvvert_{ \mathcal{ L } } 
	\le  
	\vvvert \cdot \vvvert
	$
	) 	starting with the two middle ones.
	
	\fbox{$A_{3,2}$:}
    We begin by noticing that
\begin{align*}
    \vvvert A_{3,2}[ \FP ] \vvvert^2 = &
    	\sum_{ i = 1 }^{ k } \sum_{  q \neq i } \lambda_q^{ 2 \beta }
	\frac{ \|  \hat{ e }_i [ \FP ]  - e_i \|^4 }{ 4}
	 \langle  \hat{ e }_i [ \FP ]  - e_i, e_q  \rangle^2
	  \\
\leq &
     4
	\sum_{ i = 1 }^{ k } \sum_{  q \neq i } \lambda_q^{ 2 \beta }
	 \langle  \hat{ e }_i [ \FP ]  - e_i, e_q  \rangle^2
	 .
\end{align*}
	Here we have used that $ \|  \hat{ e }_i [ \FP ]  - e_i \| \le  \|  \hat{ e }_i [ \FP ]  \|+\| e_i \| = 2$. Proceeding as for the term $A_{1, 2}[ \FP ]$, we see that $\vvvert A_{3,2}[\FP] \vvvert$ is uniformly of order $\mathcal{ O }_{ \PP }( 1 / \sqrt{ N } )$.
	
	\fbox{$A_{3,3}$:}
	The proof runs among exactly the same lines as for  $A_{3, 2}[ \FP ]$: We first observe that
$$
    \vvvert A_{3,3}[ \FP ] \vvvert^2 \le 4
	\sum_{ i = 1 }^{ k } \sum_{  q \neq i } \lambda_i^{ 2 \beta }
	 \langle  \hat{ e }_i [ \FP ]  - e_i, e_q  \rangle^2
$$
	and subsequently proceed as for $A_{1,2}[\FP]$
    to show that $\vvvert A_{3,3}[\FP] \vvvert$ is uniformly of order $\mathcal{ O }_{ \PP }( 1 / \sqrt{ N } )$.

	\fbox{$A_{3,4}$:}
	A standard calculation shows that the Hilbert--Schmidt norm of $A_{3,4}[\FP]$ is bounded as follows 
% should be $\sum_{ q > k } \sum_{ i = 1 }^{ k }$ and $\sum_{ q = 1 }^{ k } \sum_{ i \neq q }^{ k }$
%
\begin{align*}
    	   \vvvert A_{3,4}[ \FP ]  \vvvert  =
		&
		\bigg\{ \sum_{ q, l = 1 }
		\Big( \lambda_q^{  \beta } 
			\sum_{ i =1, l \neq i \neq q }^{ k }
			 \langle  \hat{ e }_i [ \FP ]  - e_i, e_q  \rangle
			  \langle  \hat{ e }_i [ \FP ]  - e_i, e_l  \rangle
		 \Big)^2 \bigg\}^{1/2}
\\
\le
&
    \bigg\{ \Big(
	 	\sum_{ i = 1 }^{ k }
	 	\sum_{ q \neq i } 
		\lambda_q^{ 2 \beta }
		\langle  \hat{ e }_i [ \FP ]  - e_i, e_q  \rangle^2
	\Big)
	\Big(
		\sum_{ i = 1 }^{ k } \sum_{ l \neq i } 
		\langle  \hat{ e }_i [ \FP ]  - e_i, e_l  \rangle^2
	\Big)\bigg\}^{1/2},
\end{align*}
	where we have applied Cauchy--Schwarz to the inner part.
	Next, bounding each factor by the same arguments as in the discussion of $A_{1,2}$ (see \eqref{bound_A_12}), we have
\begin{align} \label{Eq_UpperBound-of-TermA34-NonSmoothedRemProject}
        \vvvert A_{3,4} [ \FP ] \vvvert
    \leq \Big(
    \max_{ i \le k, q\neq i } 
       \frac{ \lambda_q^{ \beta } 
            }{
            | \hat{ \lambda }_i [ \FP ]  - \FP \lambda_q |
            }
    \vvvert  \hat{ \Gamma }_N [ \FP ]  - \Gamma\vvvert
    \Big)
    \,
    \Big(
    \max_{ i \le k, l\neq i } 
       \frac{ 1 }{ | \hat{ \lambda }_i [ \FP ]  - \FP \lambda_l | }
    \vvvert  \hat{ \Gamma }_N [ \FP ]  - \Gamma \vvvert \Big).
\end{align}
%	where we used identity \eqref{Eq_IdentInprodDiff2Eigvect-with-other-Eigvect} for the second inequality.
	By Lemma~\ref{Lem_Bound-Ratio-Eigenvalues} part $ii)$ it follows that
$$
        \max_{ i \le k, q\neq i } 
            \frac{\lambda_q^\beta }{ \hat{ \lambda }_i [ \FP ]  - \FP \lambda_q }
    =
       \mathcal{ O }_{ \PP }( 1  )
$$
and by part $i)$ of the same Lemma, that
$$
        \max_{ i \le k, q\neq i } 
        \frac{ 1 }{ \hat{ \lambda }_i [ \FP ]  - \FP \lambda_q }
    =
            \max_{ i \le k, q\neq i } 
        \frac{ 1 }{  \FP (\lambda_i   -  \lambda_q)} \mathcal{O}_\PP(1)
    =
         \mathcal{ O }_{ \PP }(  \sqrt{ N }  ).
$$
	Here we have used Assumption \ref{ass31}(7) in the second step (as the difference $\lambda_i   -  \lambda_q$ is lower bounded by $\min(\lambda_{i-1}   -  \lambda_i, \lambda_i   -  \lambda_{i+1})$). Combining this with 
	$
	\vvvert  \hat{ \Gamma }_N [ \FP ]  - \Gamma\vvvert 
	=
	\mathcal{ O }_{ \PP }( 1 / \sqrt{ N }  )$ (Theorem~\ref{Theo_StochBoundFuncCovOp}), it follows from \eqref{Eq_UpperBound-of-TermA34-NonSmoothedRemProject}, that
$$
        \sup_{\FP \in [1/2,1]} \vvvert A_{3,4} [ \FP ] \vvvert
	=
	   \mathcal{ O }_{ \PP }(  1 / \sqrt{ N }  ).
$$	
	
	\end{enumerate}

	\item[ ii) ]
		We now want to upper bound $\vvvert U_N [ \FP ]   \hat{ \Gamma}^{ \dagger}_k [ \FP ] \vvvert$. For this purpose consider the following decomposition:
\begin{equation} \label{Eq_UN_decomposition}
	\vvvert U_N [ \FP ]   \hat{ \Gamma}^{ \dagger}_k [ \FP ] \vvvert\le \vvvert U_N [ \FP ] \Gamma^{ - \zeta } \vvvert
	\left( 
	 \vvvert \Gamma^{ \zeta } ( \hat{ \Gamma}^{ \dagger}_k [ \FP ] - \Gamma^{ \dagger}_k/\FP ) \vvvert_{ \mathcal { L } }
	 +
	 \vvvert  \Gamma^{  \zeta }  \Gamma^{ \dagger}_k/\FP \vvvert_{ \mathcal { L } }
	 \right),
\end{equation}
    where $\zeta=1/2-1/(2\SM)-\epsilon$ and $\epsilon>0$ is a positive number specified later. For the inequality we have used Lemma \ref{Lem_NormIneqProd-of-HS-operat}.
    We now upper bound the factors on the right side of \eqref{Eq_UN_decomposition}, beginning with the norm $\vvvert U_N [ \FP ] \Gamma^{ - \zeta } \vvvert$, of which we show:
 \be \label{Eq_LemNonSmoothedBiasFirstTerm}
         \sup_{\FP \in [1/2,1]}\vvvert U_N [ \FP ] \Gamma^{ - \zeta } \vvvert
    =
        \mathcal{O}_{\PP}( k^\epsilon/\sqrt{N}).
\ee
    For this purpose we employ a result from \cite{moricz1982} (Theorem~3.1). The theorem is in its original form only formulated for real valued random variables, but the proof can be carried over mutatis mutandis to Hilbert space valued variables. It implies that the inequality
   \begin{equation} \label{Eq_Moricz}
   \E \sup_{\FP \in [1/2,1]}\vvvert \sqrt{N} U_N [ \FP ] \Gamma^{ - \zeta } \vvvert^2 \le \tilde C k^{2\epsilon}
   \end{equation}
   for some $\tilde C$, depending on $\epsilon$, but independent of $N$, if
   \begin{equation} \label{m_1}
      \E \bigg\vvvert \frac{1}{\sqrt{N}}\sum_{i=L}^H  \varepsilon_i \otimes X_i \Gamma^{-\zeta}  \bigg\vvvert^2 \le  C\frac{H-L}{N}
   \end{equation}
	holds for all $1 \le L \le H \le N$ and some $ C$, independent of $L$, $H$ and $N$. To verify \eqref{m_1} we observe that 
	\begin{align*}
	    & \E \bigg\vvvert \frac{1}{\sqrt{N}}\sum_{i=L}^H  \varepsilon_i \otimes X_i \Gamma^{-\zeta}  \bigg\vvvert^2 
	    =  \frac{1}{N}\sum_{i,j=L}^H \E  \langle  \varepsilon_i \otimes X_i \Gamma^{-\zeta}, \varepsilon_j \otimes  X_j \Gamma^{-\zeta}\rangle \\
	    =  & \frac{H-L}{N}\sum_{|h|=0}^{H-L-1} \left( 1-\frac{|h|}{H-L}\right) \E  \langle  \varepsilon_0 \otimes X_0 \Gamma^{-\zeta}, \varepsilon_h \otimes  X_h \Gamma^{-\zeta}\rangle \\
	    \le & \frac{H-L}{N}\sum_{h \in \Z}  \sqrt{\phi(h)}\E \|  \varepsilon_0 \otimes X_0 \Gamma^{-\zeta} \|^2,
	\end{align*}
	where we have used stationarity  for the second equality and $\phi-$mixing for the final inequality (for the covariance inequality for mixing we refer to \cite{dehling1983} equation~(3.17) with $s=r=2$).
	By Assumption~\ref{ass31}(3) the sum $\sum_{h \in \Z}  \sqrt{\phi(h)}$ is finite. Thus we only have to prove that $\E \|  \varepsilon_0 \otimes X_0 \Gamma^{-\zeta} \|<\infty$ to get \eqref{m_1}. By the Cauchy--Schwarz inequality it suffices to show $\E \|  \varepsilon_0\|^4, \E \| X_0 \Gamma^{-\zeta} \|^4<\infty$ separately, where  $\E \|  \varepsilon_0\|^4<\infty$ by assumption. For the remaining term note that 
\begin{align*}
    & \E \| X_0 \Gamma^{-\zeta} \|^4 
    = \E \left\|\sum_{n \in \N} \lambda^{-\zeta}_n \langle X_0, e_n \rangle e_n \right\|^4 = \E \left(\sum_{n \in \N} \lambda^{-2\zeta}_n\langle X_0, e_n \rangle^2 \right)^2 \\
    = & \sum_{m, n \in \N}\lambda^{-2\zeta}_n\lambda^{-2\zeta}_m \E \langle X_0, e_n \rangle^2  \langle X_0, e_m \rangle^2 \le  \sum_{m, n \in \N} \lambda^{-2\zeta}_n\lambda^{-2\zeta}_m \sqrt{\E \langle X_0, e_n \rangle^4 \E  \langle X_0, e_m \rangle^4 } \\
    \le  & \sum_{m, n \in \N} C \lambda^{-2\zeta}_n\lambda^{-2\zeta}_m \E \langle X_0, e_n \rangle^2 \E  \langle X_0, e_m \rangle^2 
    = C \left(  \sum_{n \in \N} \lambda_n^{- 2 \zeta + 1}\right)^2 \le  C \left(  \sum_{n \in \N} n^{\SM(2 \zeta-1)}\right)^2.
\end{align*}
	The last sum is finite, as by choice of $\zeta=1/2-1/(2\SM)-\epsilon$ we have $\SM(2 \zeta-1)<-1$. In the above calculations we have used the Cauchy--Schwarz inequality in the first,  Assumption~\ref{ass31}(4) in the second and  Assumption~\ref{ass31}(6) in the third inequality. 
	We have hence shown \eqref{m_1}, which -as we have argued- implies \eqref{Eq_Moricz}, which again implies \eqref{Eq_LemNonSmoothedBiasFirstTerm}.
	
	We now bound the second factor in \eqref{Eq_UN_decomposition} analyzing the term
$
	 \vvvert \Gamma^{ \zeta } ( \hat{ \Gamma}^{ \dagger}_k [ \FP ] - \Gamma^{ \dagger}_k/ \FP ) \vvvert_{ \mathcal { L } }.
$
	Notice that 
\begin{equation}\label{Eq_Bj_decomposition}
    	\hat{ \Gamma}^{ \dagger}_k [ \FP ] - \Gamma^{ \dagger}_k / \FP 
	= \sum_{i=1}^{k} \frac{\hat e_i \otimes \hat e_i}{\hat \lambda_i[ \FP ]}
	-\sum_{i=1}^{k} \frac{ e_i \otimes  e_i}{ \FP \lambda_i}
	=\sum_{ j = 1 }^4 B_{j} [ \FP ],
\end{equation}

	where
\begin{align*}
        B_{1} [ \FP ]
&
:=
		\sum_{ i =1 }^{ k } 
			( \hat{ e }_i [ \FP ]  - e_i ) \otimes \frac{ \hat{ e }_i [ \FP ] - e_i }{ \hat{ \lambda}_i [ \FP ] }
%\label{Eq_UnsmoothRem-UN-BiasTerm-B1}
\\
        B_{2} [ \FP ]
&
:=
		\sum_{ i =1 }^{ k }  ( \hat{ e }_i [ \FP ]  - e_i ) \otimes \frac{ e_i }{ \hat{ \lambda}_i  [ \FP ] }
%\label{Eq_UnsmoothRem-UN-BiasTerm-B2}
\\
        B_{3} [ \FP ]
&
:=
		\sum_{ i =1 }^{ k }  e_i \otimes \frac{  \hat{ e }_i [ \FP ]  - e_i  }{ \hat{ \lambda}_i  [ \FP ] }
%\label{Eq_UnsmoothRem-UN-BiasTerm-B3}
\\
        B_{4} [ \FP ]
&
:=
		\sum_{ i =1 }^{ k } e_i \otimes e_i 
			\frac{ \FP  \lambda_i - \hat{ \lambda}_i  [ \FP ]  
			}{
			  \hat{ \lambda}_i  [ \FP ]   \lambda_i \FP }
%\label{Eq_UnsmoothRem-UN-BiasTerm-B4}
\end{align*}

In the next step we have to show that 
\begin{equation} \label{Eq_rate_Bj}
\sup_{\FP \in [1/2,1]} \vvvert \Gamma^\zeta B_{j} [ \FP ] \vvvert_\mathcal{L}= \mathcal{O}_\PP (k^{(\SM+1)/2+\SM  \epsilon})
\end{equation}
for each $1 \le j \le 4$. For the sake of brevity, we only present the proofs for $B_2[ \FP ]$ and $B_4[ \FP ]$, as $B_1[ \FP ]$ and $B_3[ \FP ]$ can be treated by similar techniques.

	\fbox{$B_2$:} Using a Fourier expansion of the difference
	$
	\hat{ e }_i [ \FP ]  - e_i 
	$ gives
	$$
	\Gamma^{ \zeta } (\hat{ e }_i [ \FP ]  - e_i )
	=
	\sum_{q \ge 1}  \langle  \hat{ e }_i [ \FP ]  - e_i , e_q \rangle \Gamma^{ \zeta } e_q = \sum_{q \ge 1} \lambda_q^\zeta \langle  \hat{ e }_i [ \FP ]  - e_i , e_q \rangle  e_q.
	$$
Separating the terms where $i=q$ and $i \neq q$ yields	
$$
\Gamma^{ \zeta } B_{2} [ \FP ] = B_{2,1}[ \FP ]+ B_{2,2}[ \FP ],
$$  where
\be \label{Eq_UnsmoothRem-UN-BiasTerm-B2-q=i-q_neq_i}
        B_{2,1}[ \FP ] := \sum_{ i =1 }^{ k } 
        \frac{ \lambda_i^{ \zeta } }{ \hat{ \lambda}_i  [ \FP ] }
         \langle  \hat{ e }_i [ \FP ]  - e_i , e_i \rangle 
         e_i \otimes e_i \quad \quad 
    B_{2,2}[ \FP ] :=
        \sum_{ i =1 }^{ k }  \sum_{q \neq i } 
        \frac{ \lambda_q^{ \zeta } }{ \hat{ \lambda}_i  [ \FP ] }
         \langle  \hat{ e }_i [ \FP ]  - e_i , e_q \rangle 
         e_q \otimes e_i.
\ee
    Now, proceeding as for $A_{1} [ \FP ]$ we have 
\begin{align*}
    & \vvvert B_{2,1}[ \FP ] \vvvert_\mathcal{L} =
       \Big\vvvert
        \sum_{ i =1 }^{ k } 
        \frac{ \lambda_i^{ \zeta } }{ \hat{ \lambda}_i  [ \FP ] }
         \langle  \hat{ e }_i [ \FP ]  - e_i , e_i \rangle 
         e_i \otimes e_i
        \Big\vvvert_{ \mathcal{ L }  } \\[1ex]
    \leq & 
        \max_{ 1 \le i \le k }
        \frac{ \lambda_{ i }^{ \zeta } }{  \hat{ 2 \lambda}_{ i } [ \FP ]  }
        \|  \hat{ e }_i [ \FP ]  - e_i \|^2
    \leq
        2 \max_{ 1 \le i \le k }
        \frac{ \lambda_{ i }^{ \zeta } }{  \hat{  \lambda}_{ i } [ \FP ]  }
        \frac{  \vvvert \hat{\Gamma}_N[ \FP ] - \FP \Gamma\vvvert_\mathcal{ L }^2 
	     }{
	        \min( \lambda_{ i - 1 }- \lambda_i, \lambda_i - \lambda_{ i + 1 } )^2
	    }
\end{align*}
    (where we have used the  identity~\eqref{Eq_IdentSqauNorm-of-Normal-Vect} in the first and the bound ~\eqref{Eq_Inequalty-Norm-Diff-Eigenvect} in the second inequality).
    By part $i)$ of Lemma~\ref{Lem_Bound-Ratio-Eigenvalues} we see that
$$
         \max_{ i \le k }
        \frac{ \lambda_{ i }^{ \zeta } 
            }{
            \min( \lambda_{ i - 1 }- \lambda_i, \lambda_i - \lambda_{ i + 1 } )^2
             \hat{ 2 \lambda}_{ i } [ \FP ]  }
    =
        \mathcal{ O }_{ \PP }
            \left( k^{ 2 ( \SM + 1 ) + \SM ( 1 - \zeta ) } 
            \right).
$$
    Here we have replaced $\hat \lambda_i[\FP]$ by $\FP \lambda_i$ in the denominator of the Lemma and then cancelled $\lambda_i^\zeta$. Recalling that $\vvvert \hat{\Gamma}_N[ \FP ] - \FP \Gamma\vvvert_\mathcal{ L }^2 = \mathcal{O}_{\PP}(1/N)$ and using Assumption \ref{ass31}(6) shows that 
$$
   \sup_{\FP \in [1/2,1]}\vvvert B_{2,1}[ \FP ] \vvvert_\mathcal{L} = \mathcal{ O }_{ \PP } \left( \frac{ k^{ 2 ( \SM + 1 ) + \SM (1 - \zeta ) } }{ N } \right)  =  \mathcal{ O }_{ \PP } \left( \frac{ k^{   \SM + 1  + \SM (1 - \zeta ) } }{ \sqrt{N} } \right).
$$
    
    Next we consider $B_{2,2}[ \FP ] $ (defined in \eqref{Eq_UnsmoothRem-UN-BiasTerm-B2-q=i-q_neq_i}).  By arguments similar to those used in the discussion of the term $A_{1,2}$ (see \eqref{bound_A_12}) we have
$$
        \vvvert
        B_{2,2}[ \FP ]
      \vvvert
    \leq    
       \max_{ i \le k, q \neq i }  
       \frac{ \lambda_q^{  \zeta }  
            }{
                | \hat{ \lambda }_i [ \FP ] - \FP \lambda_q | 
                \, 
                 \hat{ \lambda}_i  [ \FP ]	
            }
		\vvvert  \hat{ \Gamma }_N [ \FP ]  - \FP \Gamma \vvvert.
$$
    Again, part $i)$ of Lemma~\ref{Lem_Bound-Ratio-Eigenvalues} can be used to replace the empirical eigenvalue by its population counterpart, which (by Assumption \ref{ass31}(6)) shows that
$$
       \max_{ i \le k, q \neq i }  
       \frac{ \lambda_q^{  \zeta }  
            }{
                | \hat{ \lambda }_i [ \FP ] - \FP \lambda_q | 
                \, 
                 \hat{ \lambda}_i  [ \FP ]	
            }
    =
      \mathcal{ O }_{ \PP }
            \left( k^{  ( \SM + 1 ) + \SM ( 1 - \zeta ) } 
            \right)  
$$
    Thus 
    $$ 
    \sup_{\FP \in [1/2,1]} \vvvert
        B_{2,2}[ \FP ]
      \vvvert =   \mathcal{ O }_{ \PP } \left( \frac{  k^{ ( \SM + 1 ) +  \SM ( 1 -  \zeta ) } }{ \sqrt{  N } } \right).
      $$
    Putting the estimates for both terms  together we see that $\vvvert B_2 [ \FP ] \vvvert_\mathcal{L}$ is uniformly of order
\be \label{Eq_UnsmoothRem-UN-BiasTerm-B2-order}
		\mathcal{ O }_{ \PP } \left( \frac{  k^{ ( \SM + 1 ) +  \SM ( 1 -  \zeta ) } }{ \sqrt{  N } } \right) 
	=
		\mathcal{ O }_{ \PP } \left(   k^{  ( \SM + 1 ) /2 + \SM  \epsilon }  \right),
\ee
	where the last equality holds by Assumption \ref{ass31}(7) and our choice
	$\zeta =  1/2 -  1 / ( 2 \SM ) - \epsilon$.

\fbox{$B_4 :$}  We can upper bound the operator norm of $\Gamma^{ \zeta } B_4 [ \FP ]$ as follows:
$$
        \vvvert \Gamma^{ \zeta } B_4 [ \FP ] \vvvert_\mathcal{L}
    =  
		\Big\vvvert \sum_{ i =1 }^{ k } e_i \otimes e_i 
	    \frac{ \FP  \lambda_i - \hat{ \lambda}_i  [ \FP ] 
			}{
	    \hat{ \lambda}_i  [ \FP ]   \lambda_i^{1-\zeta} \FP } \Big\vvvert_\mathcal{L} 
	    \le \max_{1 \le i \le k} \Big| \frac{ \FP  \lambda_i - \hat{ \lambda}_i  [ \FP ]
			}{
	    \hat{ \lambda}_i  [ \FP ]   \lambda_i^{1-\zeta} \FP }\Big|.
$$
    Note that
$$
    \sup_{\FP \in [1/2,1]}|\FP  \lambda_i - \hat{ \lambda}_i  [ \FP ]| \le \sup_{\FP \in [1/2,1]} \vvvert  \hat{ \Gamma}_N [ \FP ] - \FP \Gamma  \vvvert_\mathcal{L}
	=
		\mathcal{ O }_{ \PP } (1/\sqrt{ N }  ),
$$
which gives a bound for the numerator. For the denominator we can replace $\hat{ \lambda}_i  [ \FP ]$ by the
original $ \FP \lambda_i  $ (part $i)$ of Lemma \ref{Lem_Bound-Ratio-Eigenvalues}), which gives a rate 
$$
\frac{1}{\hat{ \lambda}_i  [ \FP ]   \lambda_i^{1-\zeta} \FP }= \mathcal{O}_{\PP}\bigg(\frac{1}{   \lambda_i^{2-\zeta} } \bigg) = \mathcal{O}_{\PP}\big( k^{\SM(2-\zeta)} \big).
$$
Combining these estimates we see that 
$$
    \sup_{\FP \in [1/2,1]}\vvvert \Gamma^{ \zeta } B_4 [ \FP ] \vvvert_\mathcal{L}
    =  \mathcal{O}_{\PP}(k^{\SM/2-1/2+\SM  \epsilon}) = \mathcal{O}_{\PP}(k^{\SM/2+1/2+\SM  \epsilon})~,
$$
 where we have used the definition  $\zeta=1/2-1/(2\SM)-\epsilon$ as well as Assumption \ref{ass31}(7).     We have now shown that indeed \eqref{Eq_rate_Bj} holds and therefore (see \eqref{Eq_Bj_decomposition})
\be \label{Eq_LemNonSmoothedBiasSecondTerm}
		\sup_{\FP \in [1/2,1]} \vvvert 
		\Gamma^{ \zeta } 
		( \hat{ \Gamma}^{ \dagger}_k[ \FP ] -  \Gamma^{ \dagger}_k/\FP   ) 
		\vvvert_{ \mathcal { L } }
	=
		\mathcal{ O }_{ \PP } 
    	\left( k^{\SM/2+1/2+ \SM  \epsilon}
    	\right).
\ee
	Finally, we derive an upper bound in the second term on the right of \eqref{Eq_UN_decomposition}, noting that
$$
		\Gamma^{  \zeta } \Gamma^{ \dagger}_k /\FP 
	=
		\sum_{ i \le k } \frac{ \lambda_i^{ \zeta } }{ \FP \lambda_i } e_i \otimes e_i,
$$
	which is positive definite and diagonal. Therefore, we obtain for the operator norm 
\be \label{Eq_LemNonSmoothedBiasThirdTerm}
	\sup_{\FP \in [1/2,1]} \vvvert \Gamma^{  \zeta } \Gamma^{ \dagger}_k/\FP \vvvert_\mathcal{L} = 
	\sup_{\FP \in [1/2,1]}  \max_{1 \le i \le k} \frac{ \lambda_i^{ \zeta } }{ \FP \lambda_i } = \mathcal{ O } ( k^{ \SM ( 1 - \zeta ) } ) = \mathcal{ O } ( k^{ (\SM +1)/2+\SM  \epsilon } ).
\ee
	Now combing equations~\eqref{Eq_LemNonSmoothedBiasFirstTerm}, 
	\eqref{Eq_LemNonSmoothedBiasSecondTerm} and \eqref{Eq_LemNonSmoothedBiasThirdTerm} 
	we find, that for any $\epsilon>0$ 
\begin{align*}
		\vvvert U_N [ \FP ]   \hat{ \Gamma}^{ \dagger}_k [ \FP ] \vvvert
	= &
		\mathcal{ O }_{ \PP }\left( \frac{k^\epsilon}{\sqrt{N}}  \right)
		\left(  
			\mathcal{ O }_{ \PP } \left( \  k^{ ( \SM + 1 )/2 + \SM  \epsilon}  \right) 
			+
			\mathcal{ O }_{ \PP } \left( \  k^{ ( \SM + 1 )/2 + \SM  \epsilon}  \right) 
		\right)\\
	=&
		 \mathcal{ O }_{ \PP } ( k^{ ( \SM + 1 ) /2 +   (1+ \SM ) \epsilon   } / \sqrt{ N } )  .
\end{align*}
Finally, replacing $\epsilon$ by $\epsilon/(1+ \SM ) $ proves the assertion $ii)$ of  Lemma \eqref{Eq_def_T_prime}.
%
%	\item[ iii) ]

\end{itemize}

\end{proof}

\begin{comment}

\begin{cor}\label{Cor_BoundRemainderMainTheo}
	Under the  assumptions of Theorem \ref{theorem_1}, it holds that
%
$$
		\vvvert ( S  - S_0 ) ( \hat{ \Pi }_{ k } [ \FP ] -  \Pi_{ k } ) \vvvert
	=
		\mathcal{ O }_{ \PP } ( 1 / \sqrt{ N } )
$$
\end{cor}	
\begin{proof}
	By Assumption \ref{ass31}(1) we can rewrite the slope operators $S, S_0$ as  $S = R \Gamma^{ \beta }$, $S_0 = R_0 \Gamma^{ \beta }$, where $R, R_0$ are Hilbert Schmidt operators, and
	we thus infer 
%
$$
		\vvvert ( S  - S_0 ) ( \hat{ \Pi }_{ k } [ \FP ] -  \Pi_{ k } ) \vvvert
	\le
		\vvvert R  - R_0  \vvvert
		\vvvert \Gamma^{ \beta } ( \hat{ \Pi }_{ k } [ \FP ] -  \Pi_{ k } ) \vvvert_{ \mathcal { L } } .
$$
	Using the bound from part~i) of Lemma~\ref{Lem_Bounds-NonSmoothed-Var-and-Bias-of-S_N} the claim follows.
\end{proof}	
	
\end{comment}

%
%
%--------------------------------------------------------------------------------------------------------------
%
%
% 	SUBSECTION : LINEARIZATION
%
%
%
%
%

\subsection{Linearization of the test statistic}

    In the proof of Theorem \ref{theorem_1} in Section~\ref{Subsection_App_A_3} we have used the stochastic linearization
    $T[\FP] = T'[\FP] + o_\PP(1)$, 
    where the objects $T[\FP] $ and $T'[\FP] $ are defined in \eqref{Eq_MainTheo-Def-dummyVariable-T} and \eqref{Eq_def_T_prime} respectively ($T'$ is the linearization of $T$). That this replacement is valid is a direct consequence of the subsequent Lemma~\ref{Lem_Projection-Statistics-Smoothed-Right-and-Left}. Before we state our Lemma, we define the linearization function $L$, which acts on the space of sequential Hilbert--Schmidt operators $\ell^\infty(\mathcal{S}(H_1, H_1))$ (see Definition~\ref{definition_sequential_space}). For convenience we also define the map $\tilde L$, which is used to state the long-run variance in Theorem~\ref{theorem_3}.

	%In the lemma below we provide an asymptotically equivalent expression of the variance term of our estimator, i.e., 	
	%$\Gamma^{ \beta } (\hat{ \Pi }_k [ \FP ]  - \Pi_{ k }  )  \Gamma^{ \beta }$.
	%We first introduce the linearization of our statistics, which corresponds to a first order Taylor approximation w.r.t.\ Fr\'{e}chet derivative. Our case is more complicate, since here the remainders do depend on the dimension parameter $k$ which grows with $N$.

\begin{defi}\label{Def_LinerazationOperat}
	%Let $H_1$ be the Hilbert space, as in Assumption~2)
	Let $L_i $ be the linear functional acting on the space $\ell^\infty(\mathcal{S}(H_1, H_1))$ defined pointwise in $F[\FP ]$ as
$$
		L_i (F[\FP ])
	:=
		\sum_{ q \neq i }  
		\frac{ \lambda_q^{  \beta } \lambda_i^{  \beta } 
		}{ 
		 \lambda_i  -  \lambda_q 
		}
		\langle  F[\FP ] ,  e_i   \otimes e_q \rangle ( e_q \otimes e_i ).
$$
    Therewith we define
	$
	    L (F[\FP ]) 
	:= 
	    \sum_{ i \ge 1 } L_i (F[\FP ])+ L_i (F[\FP ])^{ \ast }$. 
	 Moreover we define the map $\tilde L$ as  $\tilde L(F[\FP ]) :=  (\Gamma^{-1/2} \otimes \Gamma^{-1/2}) L (F[\FP ])$.
\end{defi}

\begin{lem}\label{Lem_Projection-Statistics-Smoothed-Right-and-Left}
	Under the assumptions of Theorem~\ref{theorem_1} it holds that
\begin{itemize}
    \item[i)] $\quad
  		\sup_{\FP \in [1/2,1]}\vvvert \FP \Gamma^{ \beta } (\hat{ \Pi }_k [ \FP ]  - \Pi_{ k }  )  \Gamma^{ \beta }
	-
		L (\hat{\Gamma}_N [ \FP ] - \FP \Gamma) \vvvert_\mathcal{L}  = o_{ \PP } ( 1 / \sqrt{  N } )
    $
    \item[ii)] 
    $\quad
    		\sup_{\FP \in [1/2,1]}\vvvert U_N[\FP] [ \FP\hat{ \Gamma}^{ \dagger }_k [ \FP ] 
		\Gamma^{ \beta } - \Gamma^{ \beta-1 }] \vvvert  = o_\PP(1/\sqrt{N}).
    $
\end{itemize}

\end{lem}
\begin{proof}
    We first prove $i)$: 
	Plugging in the definition of the projections gives the following expansion:
$$
		\Gamma^{ \beta } (\hat{ \Pi }_k [ \FP ]  - \Pi_{ k }  )  \Gamma^{ \beta }
	=
		\sum_{ i = 1 }^{ k } 
		\Gamma^{ \beta } 
			\left\{ 
			(  \hat{ e }_i [ \FP ]   \otimes \hat{ e }_i [ \FP ] )
			-
			( e_i \otimes e_i ) 
			\right\}
		\Gamma^{ \beta } = D_1[\FP] + D_2[\FP] + D_3[\FP],
$$
	where
\begin{align*}
    D_1[\FP] :=&	\sum_{ i = 1 }^{ k } 
	\Gamma^{ \beta } 
		( (  \hat{ e }_i [ \FP ]  - e_i )  \otimes e_i  )\Gamma^{ \beta }\\
		D_2[\FP] :=& \sum_{ i = 1 }^{ k } 
		\Gamma^{ \beta }(  e_i    \otimes ( \hat{ e }_i [ \FP ]  - e_i )   )\Gamma^{ \beta }\\
		D_3[\FP] :=&	\sum_{ i = 1 }^{ k } 
		\Gamma^{ \beta }( ( \hat{ e }_i [ \FP ]  - e_i )  \otimes  ( \hat{ e }_i [ \FP ]  - e_i )   )
	\Gamma^{ \beta }.
\end{align*}
 The proof now consists of two steps: In the first step we show that
 \begin{equation} \label{D_1_Linearization}
     \FP D_1[\FP]
     =
     \sum_{ i \ge 1 } L_i (\hat{\Gamma}_N [ \FP ] - \FP \Gamma). 
 \end{equation}
 As $D_2[\FP] =D_1[\FP]^{ \ast }$, this implies $\FP(D_1[\FP]+D_2[\FP]) = L (\hat{\Gamma}_N [ \FP ] - \FP \Gamma)$. In the second step we establish that $D_3[\FP]$ is uniformly of order $o_\PP(1)$.

\fbox{Step 1:}

	Using the diagonal representation $\Gamma^\beta=\sum_{q \ge 1} \lambda_q^\beta e_q \otimes e_q $ and the identity $\Gamma^\beta e_i = \lambda_i^\beta e_i$, we can rewrite $D_1[\FP]$ as follows: 
$$
	D_1[\FP] 
	= 
	 \sum_{ i = 1 }^{ k }\sum_{ q \neq i } \lambda_q^{ \beta } \lambda_i^{ \beta }
	%\sum_{ r \ge 1 } \langle  \hat{ e }_i [ \FP ]  - e_i , e_q \rangle   \langle e_i , e_r \rangle   ( e_q \otimes e_r )
	\langle  \hat{ e }_i [ \FP ]  - e_i , e_q \rangle    ( e_q \otimes e_i ) 
	+ 
	\sum_{ i = 1 }^{ k } \lambda_i^{ 2\beta }
	\langle  \hat{ e }_i [ \FP ]  - e_i , e_i \rangle    ( e_i \otimes e_i ) =:  D_{1,1}[\FP] + D_{1,2}[\FP].
$$
    Here $D_{1,1}[\FP], D_{1,2}[\FP]$ are defined in the obvious way.
 	We first show that $D_{1,2}[\FP]$ is negligible. For this purpose we use \eqref{Eq_IdentSqauNorm-of-Normal-Vect}, to see that
$$  \vvvert D_{1,2}[\FP]     \vvvert_\mathcal{L}
    =
	\Big \vvvert \sum_{ i =1 }^{ k } \frac{ - \lambda^{ 2 \beta}_i }{ 2 } \left\|  \hat{ e }_i [ \FP ]  - e_i  \right\|^2 ( e_i \otimes e_i ) \Big\vvvert_\mathcal{L}
	\le 
	\max_{1 \le i \le k}
	  \frac{ \lambda^{ 2 \beta}_i }{ 2 }  \left\|  \hat{ e }_i [ \FP ]  - e_i  \right\|^2.
$$
The maximum is smaller than a multiple of the right side of  \eqref{Eq_deal_with_ev}, which is uniformly of order $o_\PP(1/\sqrt{N})$. 

Next we turn our attention to $D_{1,1}[\FP]$. Applying identity~\eqref{Eq_IdentInprodDiff2Eigvect-with-other-Eigvect}, $D_{1,1}[\FP]$ can be rewritten as
$$  
    %D_{1,1}[\FP]
    %=
	 \sum_{ i = 1 }^{ k }\sum_{ q \neq i } \lambda_q^{ \beta } \lambda_i^{ \beta }
	%\sum_{ r \ge 1 } \langle  \hat{ e }_i [ \FP ]  - e_i , e_q \rangle   \langle e_i , e_r \rangle   ( e_q \otimes e_r )
			 \frac{ 
		 \langle  \hat{ \Gamma }_N [ \FP ]  - \FP \Gamma , \hat{ e }_i [ \FP ]  \otimes e_q \rangle
		}{ 
		( \hat{ \lambda }_i [ \FP ] - \FP \lambda_q )
		}    ( e_q \otimes e_i ).
$$
    We can now show two things: Firstly that in the above representation we can replace the empirical eigenfunction $\hat e_i[\FP]$ and eigenvalue $\hat \lambda_i[\FP]$ by their respective population  counterparts $e_i$ and $\FP \lambda_i$.
    % the empirical eigenfunction $\hat e_i[\FP]$ by its population counterpart $e_i$ and the empirical eigenvalue $\hat \lambda_i[\FP]$ by its population counterpart $\FP \lambda_i$. 
    Secondly, we can let the outer sum over $i$ run from $1$ to $\infty$, all of this while incurring only an error of size $o_\PP(1/\sqrt{N})$. More precisely we get with 
$$
    \tilde D_{1,1}[\FP]
    =
     \sum_{ i = 1 }^{ k }\sum_{ q \neq i } \lambda_q^{ \beta } \lambda_i^{ \beta }
			 \frac{ 
		 \langle  \hat{ \Gamma }_N [ \FP ]  - \FP \Gamma , e_i  \otimes e_q \rangle
		}{ 
		  \FP (\lambda_i  -  \lambda_q ) 
		}    ( e_q \otimes e_i )
$$
    that 
\begin{equation} \label{Eq_tilde_D1}
    \sup_{\FP \in [1/2,1]}
    \vvvert D_{1,1}[\FP] - \tilde D_{1,1}[\FP] \vvvert_\mathcal{L} =o_\PP(1/\sqrt{N}).
\end{equation}
	The proof of this fact requires the estimation of some further remainder terms, which we defer to the below Lemma~\ref{Lem_FirstTechLem-for-Step1-AsympFormVar}. We have now established  that $\tilde D_{1,1} [\FP] = D_1[\FP]+o_\PP(1/\sqrt{N})$ and since $\FP \tilde D_{1,1} [\FP] = \sum_{ i \ge 1 } L_i (\hat{\Gamma}_N [ \FP ] - \FP \Gamma)$ we have shown \eqref{D_1_Linearization}, which concludes the first step of this proof. \\

\fbox{Step 2:}

	Next, we show that $\vvvert D_3[\FP]\vvvert$ is uniformly of order $o_\PP(1/\sqrt{N})$.
	Using the Fourier expansion  $\hat{ e }_i [ \FP ]  - e_i = \sum_{q \ge 1} \langle \hat{ e }_i [ \FP ]  - e_i, e_q \rangle e_q$ two times, we can rewrite $D_3[\FP]$ as follows:
\begin{align*}
    D_3[\FP] =&	\sum_{ i = 1 }^{ k } \sum_{ r,q  \ge 1} \langle \hat{ e }_i [ \FP ]  - e_i, e_q \rangle \langle \hat{ e }_i [ \FP ]  - e_i, e_r \rangle 
		\Gamma^{ \beta } (e_q \otimes e_r) \Gamma^{ \beta } \\
		=& 	\sum_{ i = 1 }^{ k } \sum_{ r,q  \ge 1} \lambda_q^\beta \lambda_r^\beta \langle \hat{ e }_i [ \FP ]  - e_i, e_q \rangle \langle \hat{ e }_i [ \FP ]  - e_i, e_r \rangle 
		 (e_q \otimes e_r) 
\end{align*}
	In the second equality we have used the fact that $\Gamma^\beta e_q = \lambda_q^\beta e_q$. Consequently 
\begin{align*}
&
    \vvvert D_3[\FP] \vvvert
    =
    \bigg\{
	\sum_{ r, q \ge 1 } 
	\Big( \sum_{ i = 1 }^{ k } 
		\lambda_q^{ \beta } \lambda_r^{ \beta }
		 \langle e_q , \hat{ e }_i [ \FP ]  - e_i \rangle
		 \langle  \hat{ e }_i [ \FP ]  - e_i , e_r \rangle
	\Big)^2
	\bigg\}^{1/2}.
\end{align*}
	By applying the Cauchy--Schwarz inequality to the squared sum, 
	we get
$$
    \vvvert D_3[\FP] \vvvert
    \le 
	\sum_{  r \ge 1 } 
	\sum_{ i = 1 }^{ k } 
		\lambda_r^{ 2 \beta } \langle  \hat{ e }_i [ \FP ]  - e_i , e_r \rangle^2.
$$
The proof that
$$
\sup_{\FP \in [1/2,1]}\vvvert D_3[\FP] \vvvert =o_\PP(1/\sqrt{N})
$$
is now conducted by similar techniques as for the term $A_{1}$ in the proof of Lemma~\ref{Lem_Bounds-NonSmoothed-Var-and-Bias-of-S_N} and therefore omitted.\\
%	Hence, again relying on the identities
%	\eqref{Eq_IdentSqauNorm-of-Normal-Vect} 
%	and \eqref{Eq_IdentInprodDiff2Eigvect-with-other-Eigvect}, its H-S norm is bounded by
%
%\begin{align*}
%&
%	\sum_{ i = 1 }^{ k }
%	\sum_{ r \neq i} \lambda_r^{ 2 \beta }  
%	\frac{    
%		\langle  \hat{ \Gamma } [ \FP ]  - \FP \Gamma , \hat{ e }_i [ \FP ]  \otimes e_r \rangle^2
%	}{ 
%		  ( \hat{ \lambda }_i [ \FP ] - \FP \lambda_r  )^2
%	}
%+
%	\frac{ 1 }{ 2 }
%	\sum_{ i = 1 }^{ k }
%	 \lambda_i^{ 2 \beta }  \|   \hat{ e }_i [ \FP ]  - e_i  \|^2.
%\end{align*}
%	The second term can be bounded as in our Step~1 and is $o_{ \PP }( 1 / \sqrt{ N } )$.
%	For the first term one can use a similar strategy as in the proof of 
%	part~iii) of Lemma~\ref{Lem_FirstTechLem-for-Step1-AsympFormVar} to find that
%	it is $\mathcal{ O }_{ \PP }( 1 / N )$.

Finally we turn to the proof of part $ii)$ of this Lemma. Since this proof is technically very similar to part $ii)$ of Lemma~\ref{Lem_Bounds-NonSmoothed-Var-and-Bias-of-S_N} we only sketch the idea: We begin by the simple upper bound
$$
        \vvvert  U_N[\FP] [\FP \hat{ \Gamma}^{ \dagger }_k [ \FP ] 
		\Gamma^{ \beta } - \Gamma^{ \beta-1 }] \vvvert  
		\le 
		\vvvert U_N[\FP] \Gamma^{-\zeta}\vvvert \vvvert \Gamma^\zeta [ \FP\hat{ \Gamma}^{ \dagger }_k [ \FP ] 
		\Gamma^{ \beta } - \Gamma^{ \beta-1 }] \vvvert_\mathcal{L}.
$$
Here  $\zeta=1/2-1/(2\SM)-\epsilon$ and $\epsilon>0$ is a positive number, which can be chosen arbitrarily small. By \eqref{Eq_LemNonSmoothedBiasFirstTerm} we know that uniformly
$\vvvert U_N [ \FP ] \Gamma^{ - \zeta } \vvvert
    =
        \mathcal{O}_{\PP}( k^\epsilon/\sqrt{N}).
$
Thus it suffices to show that the factor
$
\vvvert \Gamma^\zeta [ \FP \hat{ \Gamma}^{ \dagger }_k [ \FP ] 
		\Gamma^{ \beta } - \Gamma^{ \beta-1 }/\FP] \vvvert_\mathcal{L}
$
decays at some arbitrarily small, polynomial speed in $N$, to get the assertion. We upper bound it by the sum
$$
    \vvvert \Gamma^\zeta [ \FP \hat{ \Gamma}^{ \dagger }_k [ \FP ] 
	- \Gamma^{ \dagger }_k] \Gamma^{ \beta }] \vvvert_\mathcal{L} 
	+
	\vvvert  \Gamma^{\beta-1+\zeta} \Pi_{k}-\Gamma^{\beta-1+\zeta} \vvvert_\mathcal{L}  =: F_1[\FP] + F_2,
$$
where $F_1[\FP], F_2$ are defined in the obvious way, and analyse the terms separately. For $F_1[\FP]$, we use similar techniques as in the proof of Lemma \ref{Lem_Bounds-NonSmoothed-Var-and-Bias-of-S_N}  (after equation \eqref{Eq_Bj_decomposition}). Notice that we can indeed show convergence to $0$ as  additional smoothing is applied (by $\Gamma^\beta$ from the right). The proof for $F_2$ is rather simple: $\Gamma^{\beta-1+\zeta}-\Gamma^{\beta-1+\zeta} \Pi_{k}$ is symmetric, positive definite and can be expressed (by the spectral theorem in Section~\ref{Subsection_Operators}) as $\sum_{q >k} \lambda_q^{\beta-1 + \zeta} e_q \otimes e_q$. Thus
$$
\vvvert  \Gamma^{\beta-1+\zeta} \Pi_{k}-\Gamma^{\beta-1+\zeta} \vvvert_\mathcal{L} = \lambda_{k}^{\beta-1+\zeta}= \mathcal{O}(k^{-\SM(\beta-1+\zeta)}).
$$
This concludes the proof.
\end{proof}

In the proof of Lemma~\ref{Lem_Projection-Statistics-Smoothed-Right-and-Left}, we have used the identity~\eqref{Eq_tilde_D1} for the final step in the linearization. In the Lemma below we give three upper bounds, which combined directly imply \eqref{Eq_tilde_D1}.

\begin{lem}\label{Lem_FirstTechLem-for-Step1-AsympFormVar}
	Under the assumptions of Theorem~\ref{theorem_1} the following identities hold:
\begin{itemize}
	\item[ i) ]  $$
	\sup_{\FP \in [1/2, 1]} \Big\vvvert \sum_{ i > k } L_i(\hat \Gamma_N [\FP]- \FP \Gamma) \Big\vvvert  =  o_{ \PP } ( 1 / \sqrt{ N } ).
	$$

	\item[ ii) ]
$$
	\sup_{\FP \in [1/2, 1]} \Big\vvvert \sum_{ i =1 }^{ k } \sum_{ q \neq i } \lambda_q^{ \beta } \lambda_i^{ \beta }
	\frac{ 
		 \langle  \hat{ \Gamma }_N [ \FP ]  - \FP \Gamma , ( \hat{ e }_i [ \FP ] - e_i ) \otimes e_q \rangle
		}{ 
		( \hat{ \lambda }_i [ \FP ] - \FP \lambda_q )
		}
	( e_q \otimes e_i )\Big\vvvert
=
	o_{ \PP }( 1 / \sqrt{ N } ).
\quad
$$

	\item[ iii) ]
$$
	\sup_{\FP \in [1/2, 1]} \Big\vvvert \sum_{ i =1 }^{ k } \sum_{ q \neq i } \lambda_q^{ \beta } \lambda_i^{ \beta }
	\frac{ 
		 \langle  \hat{ \Gamma }_N [ \FP ]  - \FP \Gamma , e_i  \otimes e_q \rangle
		 ( \FP \lambda_i -  \hat{ \lambda }_i [ \FP ] )
		}{ 
		( \hat{ \lambda }_i [ \FP ] - \FP \lambda_q ) ( \FP ( \lambda_i -  \lambda_q ) )
		}
	( e_q \otimes e_i )\Big\vvvert
=
	o_{ \PP }( 1 / \sqrt{ N } ).
$$
	
\end{itemize}

\end{lem}
\begin{proof} \mbox{}

\begin{itemize}
	\item[ i) ] 
	It follows by standard calculations, that 
$$
	\sup_{\FP \in [1/2, 1]} 
	\vvvert \sum_{ i > k } L_i(\hat \Gamma_N [\FP]- \FP \Gamma) \vvvert 
\le 
    C_N \sup_{\FP \in [1/2, 1]} 
    \vvvert  \hat{ \Gamma }_N [ \FP ]  - \FP \Gamma \vvvert,
$$
	with (recall that $k := k ( N )$)
$$
		C_N 
	:=
	 \max_{i > k, q \neq i} \frac{ \lambda_i^{  \beta } \lambda_q^{  \beta } }{ (  \lambda_i  -  \lambda_q )}  \,  .
$$
     Lemma~\ref{Lem_Bound-Ratio-Eigenvalues} implies that 
$$	
		\max_{ q \neq i } 
		\frac{ \lambda_q^{  \beta }  
		}{ 
		| \lambda_i  -  \lambda_q |
		}
	=
		\mathcal{ O } (1 ).
$$
	 Hence $C_N = o ( 1 )$, as $N \to \infty$. 
	 Since uniformly (in $\FP$) $ \vvvert  \hat{ \Gamma }_N [ \FP ]  - \FP \Gamma \vvvert = \mathcal{O}_\PP(1/\sqrt{N})$ (see Theorem~\ref{Theo_StochBoundFuncCovOp}), the assertion follows.
	\item[ ii) ] 
	The squared Hilbert--Schmidt norm of the term in part $ii)$ for some fixed $\FP$ equals
\be \label{Eq_FirstTechLem_HS-norm-Op}
	\sum_{ i =1 }^{ k } \sum_{ q \neq i } \lambda_q^{ 2 \beta } \lambda_i^{ 2 \beta }
	\frac{ 
		 \langle  \hat{ \Gamma }_N [ \FP ]  - \FP \Gamma , ( \hat{ e }_i [ \FP ] - e_i ) \otimes e_q \rangle^2
		}{ 
		( \hat{ \lambda }_i [ \FP ] - \FP \lambda_q )^2 
		}
\ee
	 and can be upper bounded by
\begin{align*}
	&	\max_{ i \le k, q \neq i }
		\frac{  \lambda_q^{ 2 \beta }  }{ ( \hat{ \lambda }_i [ \FP ] - \FP \lambda_q )^2 }
		\sum_{ i =1 }^{ k } \lambda_i^{ 2 \beta } \sum_{ q \neq i }
			 \langle  \hat{ \Gamma }_N [ \FP ]  - \FP \Gamma , ( \hat{ e }_i [ \FP ] - e_i ) \otimes e_q \rangle^2 \\
	\le &
		\mathcal{ O }_{ \PP }( 1 ) 
		\sum_{ i =1 }^{ k } \lambda_i^{ 2 \beta }
		\| ( \hat{ \Gamma }_N [ \FP ]  - \FP \Gamma ) ( \hat{ e }_i [ \FP ] - e_i ) \|^2,
\end{align*}
	where we have used Lemma~\ref{Lem_Bound-Ratio-Eigenvalues}, part $ii)$.  We further bound the right factor
\begin{align*}
		& \sum_{ i =1 }^{ k } \lambda_i^{ 2 \beta }
		\| ( \hat{ \Gamma }_N [ \FP ]  - \FP \Gamma ) ( \hat{ e }_i [ \FP ] - e_i ) \|^2 
		\le C \sum_{ i =1 }^{ k } \frac{ \lambda_i^{ 2 \beta } \vvvert
		 \hat{ \Gamma }_N [ \FP ]  - \FP \Gamma \vvvert_\mathcal{L}^4}{\FP^2 \min(\lambda_{i-1}-\lambda_i, \lambda_i-\lambda_{i+1})^2} \\
	\le & C k  \vvvert
		 \hat{ \Gamma }_N [ \FP ]  - \FP \Gamma \vvvert_\mathcal{L}^4 \max_{1 \le i \le k} \frac{ \lambda_i^{ 2 \beta }}{ \min(\lambda_{i-1}-\lambda_i, \lambda_i-\lambda_{i+1})^2}
%		 \mathcal{ O }_{ \PP }( k^{1+2 \SM +2-2\SM  \beta} / N^2 )=  \mathcal{ O }_{ \PP }( k / N^2 ).
\end{align*}
	Here we have used  inequality~\eqref{Eq_Inequalty-Norm-Diff-Eigenvect} in the first step and the fact that $1/\FP \le 2$ in the second. Recall that $\vvvert
		 \hat{ \Gamma }_N [ \FP ]  - \FP \Gamma \vvvert_\mathcal{L}^4$ is uniformly of order $\mathcal{O}_\PP(1/N^2)$ (see Theorem \ref{Theo_StochBoundFuncCovOp}) and that the ratio of eigenvalues on the right is bounded by part $iii)$ of Lemma \ref{Lem_Bound-Ratio-Eigenvalues}.
	Therefore, the term in \eqref{Eq_FirstTechLem_HS-norm-Op} is of order
	 $\mathcal{ O }_{ \PP } ( k / N^2 ) = o_{ \PP } ( N^{ - 3 / 2 } )$ (by Assumption~\ref{ass31}(7)).

	\item[ iii) ]
	With similar arguments as in $ii)$,
	one sees that the squared Hilbert--Schmidt norm of the term in $iii)$ is bounded by
$$
		\mathcal{ O }_{ \PP } ( 1 )
		\sum_{ i =1 }^{ k } 
		 \vvvert   \hat{ \Gamma }_N [ \FP ]  - \FP \Gamma \vvvert^2
		|  \hat{ \lambda }_i [ \FP ] - \FP \lambda_i |^2
	=
		\mathcal{ O }_{ \PP } ( k / N^2 ),
$$
%	Then taking the supremum over $\FP \in \left[ 1/2, 1 \right]$ and using again Assumption~6), 
	which is $o_\PP(  N^{ - 3 / 2 } )$ and thus yields the desired result.
 
\end{itemize}

\end{proof}

%
%
%--------------------------------------------------------------------------------------------------------------
%
%
% 	SUBSECTION : TECHNICAL LEMMATA
%
%
%
%
%

\subsection{Convergence results for empirical eigenvalues}

In this section we collect a few results on the convergence speed of the empirical eigenvalues to their population counterparts, which are used at several places in this paper.

\begin{lem}\label{Lem_Bound-Ratio-Eigenvalues}
	Under the assumptions of Theorem \ref{theorem_1}, it holds that 
	\begin{itemize}
	    \item[i)]	
	    $$ \sup_{\FP \in [1/2,1]}\max_{1 \le i \le k} |\hat \lambda_i[\FP]-\FP \lambda_i| = o_{\PP}(k^{-\SM -1})$$
	    \item[ii)]
	    $$
	%	 \sup_{\FP \in [1/2,1]}
		 \max_{1 \le i \le k}\max_{q \neq i} \left|\frac{  \lambda_q^{  \beta }  }{  \hat{ \lambda }_i [ \FP ] - \FP \lambda_q  }\right|
	    =
		\mathcal{ O }_{ \PP }( 1 )
        $$
        \item[iii)]
	    $$
		 \sup_{\FP \in [1/2,1]}\max_{1 \le i \le k}\max_{q \neq i} \left|\frac{  \lambda_q^{  \beta }  }{   \lambda_i -  \lambda_q  }\right|
	    =
		\mathcal{ O }( 1 )
        $$
	\end{itemize}

\end{lem}	
\begin{proof}
The proof of $i)$ follows by an application of part~$iv)$ of Lemma~\ref{Lem_bounds_op} below, together with Theorem~\ref{Theo_StochBoundFuncCovOp}: 
According to the former $|\hat \lambda_i[\FP]-\lambda_i \FP| \le \vvvert \hat \Gamma_N[\FP]- \FP \Gamma \vvvert _\mathcal{L}$
and according to the latter $\vvvert \hat \Gamma_N[\FP]- \FP \Gamma\vvvert_\mathcal{L}$  is uniformly (in $\FP$) of order $\mathcal{O}_\PP(1/\sqrt{N})$. Recalling Assumption~\ref{ass31}(7) we note that $1/\sqrt{N}=o_\PP(k^{-\SM-1})$, which concludes the proof.\\

For $ii)$ we first notice that according to $i)$ we can replace $\hat{ \lambda }_i [ \FP ]$ by 
is population counterpart $\FP\lambda _i$. Since $1/\FP$ is bounded (by 2), the proof of $ii)$ can be reduced to the proof of $iii)$.\\

We now show $iii)$: Let us define the function
$$
 f(i,q) := \frac{  \lambda_q^{  \beta }  }{  |\lambda_i  -  \lambda_q|}.
$$
We can make the maximum in $iii)$ larger by maximizing $f$ over all $\{(i,q) \in \N^2: i \neq q\}$. Now the proof works by contradiction: Suppose there was a sequence of $\{(i_n^*, q_n^*)\}_{n \in \N}$, such that $f(i_n^*, q_n^*) \to \infty$. For each tuple the value of the function is finite and thus 
$i_n^* \to \infty$ or $ q_n^* \to \infty$. Now there are three possibilities: $|i_n^*/ q_n^*| $ is bounded, goes to $0$ or goes to $\infty$. The case $|i_n^*/ q_n^*| \to 0$ can be excluded, as for any fixed $i$ the function $q \mapsto f(i,q)$ is monotonically decreasing in $q$ for $q>i$. Thus if there was a sequence  $(i_n^*, q_n^*)$ with $f(i_n^*, q_n^*) \to \infty$ and $|i_n^*/ q_n^*| \to 0$ it also holds true for  $(i_n^{*}, i_n^{*}+1)$ that $f(i_n^*, i_n^*+1) \to \infty$, which brings us to the next case of bounded $|i_n^*/ q_n^*|$: If $|i_n^*/ q_n^*|$ was bounded, it follows directly that
$$
\frac{  \lambda_{q_n^*}^{  \beta }  }{  \lambda_{i_n^*}  -  \lambda_{q_n^*} } = \mathcal{O}\left( \frac{  (q_n^*)^{-\beta \SM} }{  (q_n^*)^{- \SM-1} } \right),
$$
which by choice of $\beta>1+1/\SM$ is asymptotically vanishing. Finally we consider the case where $|i_n^*/ q_n^*| \to \infty$. In this situation 
$$
\frac{  \lambda_{q_n^*}^{  \beta }  }{  \lambda_{i_n^*}  -  \lambda_{q_n^*} } = \mathcal{O}\left(\frac{  \lambda_{q_n^*}^{  \beta }  }{    \lambda_{q_n^*} }  \right),
$$
which is asymptotically vanishing. As a consequence we conclude that
$$
\max_{(i,q) \in \N^2: i \neq q} f(i,q)<\infty,
$$
proving the assertion.
	
\end{proof}

\newpage

%
%
%--------------------------------------------------------------------------------------------------------------
%
%
%
% 	SECTION : APPENDIX B
%
%
%
%--------------------------------------------------------------------------------------------------------------
%
%

\section{Miscellaneous} \label{misc}

%
%
%--------------------------------------------------------------------------------------------------------------
%
%
% 	SUBSECTION : SOME DETERMINISTIC INEQUALITIES
%
%
%
%
%

\subsection{Operatortheoretic (in)equalities}

	We begin with an observation concerning bounds on products of operators.

\begin{lem}\label{Lem_NormIneqProd-of-HS-operat}
Suppose three Hilbert spaces $\HH_1, \HH_2, \HH_3 $ are given.
	Let $A \in \mathcal{ S } ( \HH_1, \HH_2 )$, $B \in \mathcal{ L } ( \HH_2, \HH_3 )$ and $B' \in \mathcal{ S } ( \HH_2, \HH_3 )$, $A' \in \mathcal{ L } ( \HH_1, \HH_2 )$.
	Then, it holds that
$$ 
		\vvvert B A \vvvert \le \| B \|_{ \mathcal { L } }  \vvvert A \vvvert
	\quad
		\text{ and }
	\quad
		\vvvert B' A' \vvvert \le  \vvvert B' \vvvert  \vvvert A' \vvvert_{ \mathcal { L } }.
$$ 
	%Since $\| T^{ \ast } \|_{ \mathcal { L } } =
	%\| T \|_{ \mathcal { L } }$ and $\vvvert A^{ \ast } \vvvert =  \vvvert  A \vvvert$, it also follows that
	%$ \vvvert T A \vvvert \le \| T \|_2  \vvvert A \vvvert_{ \mathcal { L } } $.
\end{lem}

In the proof we use the notion of trace for linear operators, which is well defined for {\it trace class operators} (a class of operators, which in particular includes products of Hilbert--Schmidt operators). For a precise definition of the trace $\Tr$ we refer to Section~13.5 in \cite{HorKokBook12}.

\begin{proof}
    We only prove the first inequality, as the proof for the second is identical. For $A,B$ it holds that
    $$
    \vvvert B A \vvvert^2 = \Tr[B^{ \ast } B A A^{ \ast } ] \le \Tr[B^{ \ast } B \, \operatorname{Id} \vvvert A A^{ \ast } \vvvert_\mathcal{L}] \le \Tr[B^{ \ast } B] \vvvert A A^{ \ast }  \vvvert_\mathcal{L} = \vvvert B \vvvert^2 \vvvert A \vvvert_\mathcal{L}^2.
    $$
    Here we have used that $B^{ \ast } B, A A^{ \ast }$ are positive semi-definite and symmetric. Furthermore, for the first inequality we have used that for symmetric, positive semi-definite operators $O_1, O_2, \tilde O_2$ the inequality $\Tr[O_1 O_2] \le \Tr[O_1 \tilde O_2]$ holds, if $\tilde O_2-O_2$ is positive semi-definite (Loewner order). This result extends directly to Hilbert spaces from the finite dimensional case and is therefore not proven.
	%The first inequality is an easy consequence of the definition 
	%$\vvvert  A \vvvert^2 
	%=
	% \sum_{ i \ge 1 } \| A ( f_i ) \|^2
	%$ where $f_i$, $i \ge 1 $ is any ONB of $H_1$.
	%
	%The second inequality follows from 
%
%$$
%		\vvvert T A \vvvert 
%	=
%		\vvvert A^{ \ast } T^{ \ast } \vvvert 
%	\le
%		\| A^{ \ast } \|_{ \mathcal { L } }  \vvvert T^{ \ast } \vvvert
%	=
%		\| A \|_{ \mathcal { L } } \vvvert T \vvvert
%$$
%	where we used the first inequality.
\end{proof}
%
%
%	It is well-known that a trace operator can be defined on a class of operators, which are then called trace class operators. Moreover, the H-S operators belong to this trace class of operators. This trace operator enjoys similar properties as for matrices.

%\begin{lem}\label{Lem_TraceOp-is-Invar-CyclPermut}
%	Let $H$ be a Hilbert space and $A, B \in \mathcal{ S } ( H, H )$. 
%	Then 
%
%$$
%		\operatorname{ Tr }( A B ) 
%	=
%		\operatorname{ Tr }(  B A ) 
%$$
%\end{lem}
%\begin{proof}
%	The proof is a straightforward verification using the definition of the trace, namely
%
%$$
%		\operatorname{ Tr }( A B ) 
%	:=
%		\sum_{ i \ge 1 } \langle A B e_i, e_i \rangle
%$$
%	where $e_i$ is an ONB of $H$.
%\end{proof}

Next we want to discuss in which sense the eigensystems of two similar operators are also similar. For this purpose we have to determine how we deal with the non-uniqueness of eigenfunctions: 

\begin{rem} \label{Rem_eigenfunctions}
{\rm 	Let $A, B$ be two compact, self-adjoint, positive semi-definite operators with eigenvalues (in non-increasing order) and corresponding eigenfunctions $ \alpha_j, a_j$ 
	and $ \beta_k, b_k$ respectively. The eigenfunctions are only determined up to sign, i.e.\ both $a_i$ and $-a_i$ are eigenfunctions of $A$, belonging to the $i$-th eigenvalue $\alpha_i$. However in order to make a comparison of, say the $i-$th eigenfunctions $a_i$ of $A$ and $b_i$ of $B$ meaningful, we have to consider the minimum 
	$ \min (\|a_i-b_i\|, \|a_i+b_i\|) $ (otherwise even "the same" eigenfunctions with opposing signs would result in a difference $\|a_i-b_i\|=\sqrt{2}$). For sake of notational parsimony we always assume that, comparing two eigenfunctions of different operators, the functions have the same sign, in the sense that already $ \|a_i-b_i\|=\min (\|a_i-b_i\|, \|a_i+b_i\|)$. 
	} 
\end{rem} 

In the next lemma we provide some identities for eigenfunctions and eigenvalues of self-adjoint operators.

\begin{lem} \label{Lem_bounds_op}
	Let $A, B$ be two compact, self-adjoint, positive semi-definite operators with (in non-increasing order) and corresponding eigenfunctions $ \alpha_j, a_j$ 
	and $ \beta_k, b_k$ respectively. Furthermore suppose that all eigenvalues of $A$ are distinct, i.e.\ $\alpha_1 >\alpha_2> \cdots$.
	Then it holds that
\begin{itemize}
	\item[ i) ] for $j \neq k$ and $\alpha_j - \beta_k \neq 0$:
\be \label{Eq_IdentInprodDiff2Eigvect-with-other-Eigvect}
		\langle a_j - b_j , b_k  \rangle 
	= 
		\frac{ \langle A - B, a_j \otimes b_k \rangle }{\alpha_j - \beta_k },
\ee

	\item[ ii) ]  for any pair of normalized vectors $v, w$:

\be \label{Eq_IdentSqauNorm-of-Normal-Vect}
		\langle w, v - w \rangle 
	= 
		- \frac{ 1 }{ 2 } \| v - w \|^2,
\ee

	\item[ iii) ] for all $i \ge 1$:

\be \label{Eq_Inequalty-Norm-Diff-Eigenvect}
		\| b_i - a_i \|
	\le 
		2 \sqrt{ 2 } 
		\frac{ \| A - B \|_\mathcal{L} 
		}{ 
		\min \{ \alpha_{ i - 1 } - \alpha_{ i  } , \alpha_i - \alpha_{ i + 1 } \} }.
\ee
	%strangely enough the same has to hold with the random eigenvalues.	
%
\item[ iv) ] for all $i \ge 1$:
\be \label{Eq_Inequalty-Norm-Diff-Eigenval}
		\| \beta_i - \alpha_i \|
	\le 
		\vvvert B-A \vvvert_\mathcal{L}.
\ee

\begin{proof}
    Identities $i)$ and $ii)$ are straightforward adaptions of Lemma~1 in \cite{kokoszka2013} and for $iii)$ and $iv)$ we refer to \cite{HorKokBook12} (Lemmas~2.2 and 2.3).
\end{proof}

\end{itemize}

\end{lem}

\end{document}